\documentclass[10pt]{article}
\usepackage{graphicx,float,latexsym,amsmath,amsfonts,amssymb,subeqnarray,colordvi,epstopdf,bbm}
\usepackage[latin1]{inputenc}
\usepackage{a4wide}

\def\R{\mathbb{R}}

\def\imi{\textbf{\hskip1pt i\hskip1pt}}

\newtheorem{theorem}{Theorem}

\newtheorem{proposition}[theorem]{Proposition}

\newenvironment{proof}{\noindent {\it Proof}~}{}

\title{Convergence analysis of the Modified Craig--Sneyd scheme for two-dimensional 
convection-diffusion equations with nonsmooth initial data}
\author{Maarten Wyns\footnote{Department of Mathematics and Computer Science,
University of Antwerp, Middelheimlaan 1, B-2020 Antwerp, Belgium.
\mbox{Email}: \texttt{maarten.wyns@uantwerpen.be}.}}

\date{\today}

\begin{document}

\maketitle

\begin{abstract}
\noindent
In this paper we consider the Modified Craig--Sneyd (MCS) scheme which forms a prominent time stepping method of the Alternating Direction Implicit type for multidimensional time-dependent convection-diffusion equations with mixed spatial derivative terms. 
When the initial function is nonsmooth, which is often the case for example in financial mathematics, application of the MCS scheme can lead to spurious erratic behaviour of the numerical approximations.
We prove that this undesirable feature can be resolved by replacing the very first MCS timesteps by several (sub)steps of the implicit Euler scheme.
This technique is often called Rannacher time stepping. We derive a useful convergence bound for the MCS scheme combined with Rannacher time stepping when it is applied to a model two-dimensional convection-diffusion equation with mixed-derivative term and with Dirac-delta initial data. Ample numerical experiments are provided that show the sharpness of our obtained error bound.
\end{abstract}
\vspace{0.2cm}\noindent
\vspace{3mm}
\normalsize

\setcounter{equation}{0}
\section{Introduction}\label{intro}

In financial mathematics, the fair value $u(s_{1},s_{2},t)$ of a European style option on two underlying assets is modelled by the two-dimensional \textit{Black--Scholes partial differential equation} (PDE), see e.g.\ \cite{B98}, 
\begin{equation}
u_{t} = \tfrac{1}{2} \sigma_{1}^{2} s_{1}^{2} u_{s_{1}s_{1}} + \rho\sigma_{1}\sigma_{2} s_{1}s_{2} u_{s_{1}s_{2}} + \tfrac{1}{2} \sigma_{2}^{2} s_{2}^{2} u_{s_{2}s_{2}} +  r s_{1}u_{s_{1}} + r s_{2} u_{s_{2}} - r u,
\label{eq:2DBlackScholes}
\end{equation}
for $s_{1},s_{2} > 0, \ 0 < t \leq T$. Here, $t$ denotes the time to maturity $T$ and we assume real parameters $r, \sigma_{1} > 0, \sigma_{2}>0, \vert \rho \vert < 1$. 
The PDE \eqref{eq:2DBlackScholes} is provided with an initial condition that is defined through the payoff of the option.

The mixed spatial derivative term in \eqref{eq:2DBlackScholes} represents the correlation between both asset prices in the two-dimensional Black--Scholes model. Mixed spatial derivative terms are very important, notably, in the field of financial option valuation theory. 
Here they arise due to the correlation between underlying stochastic processes.

A well-known approach for determining the fair values $u(s_{1},s_{2},T)$ consists of numerically solving PDE (\ref{eq:2DBlackScholes}) by the \textit{method-of-lines}, whereby one first discretizes in space and subsequently in time. 
In this paper we consider a uniform Cartesian grid and second-order central finite difference schemes in space.
This semidiscretization is second-order convergent with respect to the spatial mesh width if the initial and boundary data is smooth, see e.g.\ \cite{HV03}.
When the PDE is multidimensional, then the application of classical implicit time discretization methods to the obtained semidiscrete systems can be computationally very intensive.
In view of this, for the effective time discretization, operator splitting schemes of the \textit{Alternating Direction Implicit} (ADI) type are widely considered. 
In this paper we consider the \textit{Modified Craig--Sneyd (MCS) scheme} \cite{IHW09}, which is a prominent scheme of the ADI type.
In the past years various positive stability results for the MCS scheme have been derived relevant to multidimensional convection-diffusion equations with mixed derivative terms, see e.g.\ \cite{IHM11,IHM13,IHW09,M14}. Recently, in 't Hout and Wyns \cite{IHW15} proved that, under some natural stability and smoothness assumptions, the MCS scheme is second-order convergent with respect to the time step whenever it is applied to semidiscrete two-dimensional convection-diffusion equations with mixed derivative term. The temporal convergence result from \cite{IHW15} has the crucial property that it holds uniformly in the spatial mesh width. Hence, the fully discrete numerical solution is second-order convergent in space and time for smooth initial and boundary data.

A relevant convergence analysis for the MCS scheme and nonsmooth data is still open in the literature.
In financial applications, however, the initial function is in general nonsmooth.
It is well-known that convergence can then be seriously impaired.
As an illustration, consider a two-asset cash-or-nothing option with strikes $K_{1}>0$ and $K_{2}>0$, so that
$$ u(s_{1},s_{2},0) = \mathbbm{1}_{\{s_{1} \geq K_{1}\}} \mathbbm{1}_{\{s_{2} \geq K_{2}\}}, $$
where $\mathbbm{1}$ denotes the indicator function.
In the upper left plot in Figure \ref{fig:DigitalCall2}, the numerical solution for $u(s_{1},s_{2},T)$ is shown for (natural) financial parameter values $r=0.05,$ $\sigma_{1} = 0.2,$ $\sigma_{2} = 0.25,$ $\rho = -0.7,$ $K_{1} = 1,$ $K_{2} = 1,$ $T = 2$.  
Irregularities can be observed around the strikes, leading to a loss of accuracy in the maximum norm. 
For hedging purposes it is important to consider also the Greeks, for example the cross gamma $\Gamma = u_{s_{1}s_{2}}$. 
The corresponding PDE is given by
\begin{eqnarray}
\Gamma_{t} &=& \tfrac{1}{2} \sigma_{1}^{2} s_{1}^{2} \Gamma_{s_{1}s_{1}} + \rho\sigma_{1}\sigma_{2} s_{1}s_{2} \Gamma_{s_{1}s_{2}} + \tfrac{1}{2} \sigma_{2}^{2} s_{2}^{2} \Gamma_{s_{2}s_{2}} \nonumber \\
&& + \ (r+\sigma_{1}^{2} + \rho \sigma_{1} \sigma_{2}) s_{1}\Gamma_{s_{1}} + (r+\sigma_{2}^{2} + \rho \sigma_{1} \sigma_{2}) s_{2} \Gamma_{s_{2}} + (r + \rho\sigma_{1}\sigma_{2}) \Gamma,
\label{eq:PDECrossGamma}
\end{eqnarray}
for $s_{1},s_{2} > 0, \ 0 < t \leq T$. This is supplemented with initial function
\begin{equation*}
\Gamma(s_{1},s_{2},0) = u_{s_{1}s_{2}}(s_{1},s_{2},0) = \delta(s_{1}-K_{1}) \delta(s_{2} - K_{2}),
\end{equation*} 
where $\delta$ is the \textit{Dirac delta function}. 
The lower left plot in  Figure \ref{fig:DigitalCall2} shows the numerical solution for the cross gamma at maturity $T$ for the same financial parameter values as above.
Around the point $(s_{1},s_{2}) = (K_{1},K_{2})$ strong, spurious erratic behaviour shows up and, hence, this approximation is useless in practice. 
If the cross gamma is approximated by applying finite difference schemes directly to the numerical solution for the option value, which is a common alternative technique in practice, the same observations are found.

For one-dimensional applications in finance, the impact of nonsmooth initial data on convergence has already been studied extensively and various techniques have been proposed in order to recover standard convergence results, see e.g.\ \cite{GC06,PVF03}. 
A common technique consists of first applying several implicit Euler (sub)steps and then continue with the time stepping scheme under consideration, \cite{R84}. This is called \textit{Rannacher time stepping} or \textit{implicit Euler damping}. 

Consider again PDEs (\ref{eq:2DBlackScholes}) and \eqref{eq:PDECrossGamma} for the two-asset cash-or-nothing option. 
Replacing the MCS scheme in the first two timesteps by four half-timesteps of the implicit Euler scheme, the two right plots in Figure \ref{fig:DigitalCall2} are obtained. 
Clearly, there are no longer irregularities or oscillations present. 
In many other multidimensional applications, see e.g.\ \cite{IHF10}, the same observations were made. To the best of our knowledge, however, there are no theoretical results available in the literature concerning the favourable effect of Rannacher time stepping on the convergence of the MCS scheme if the initial data is nonsmooth.

In the present paper we will prove a useful convergence bound for the MCS scheme when it is applied to a model two-dimensional convection-diffusion equation with mixed derivative term, provided with Dirac delta initial data. Here, semidiscretization is performed with second-order central finite difference schemes. 
The precise influence of Rannacher time stepping on the order of convergence will be investigated.
Our analysis in this paper is inspired by that of Giles and Carter \cite{GC06}, who deal with the Crank-Nicolson scheme applied to a model one-dimensional convection-diffusion equation.
We make use of a two-dimensional mixed discrete/continuous Fourier transformation and analyse the asymptotic behaviour of the Fourier transform. 
By applying then the inverse transformation we arrive at an error bound for the total error.
The sharpness of the error bound is confirmed by ample numerical experiments.

\begin{figure}
\begin{center}
\includegraphics[scale=0.5]{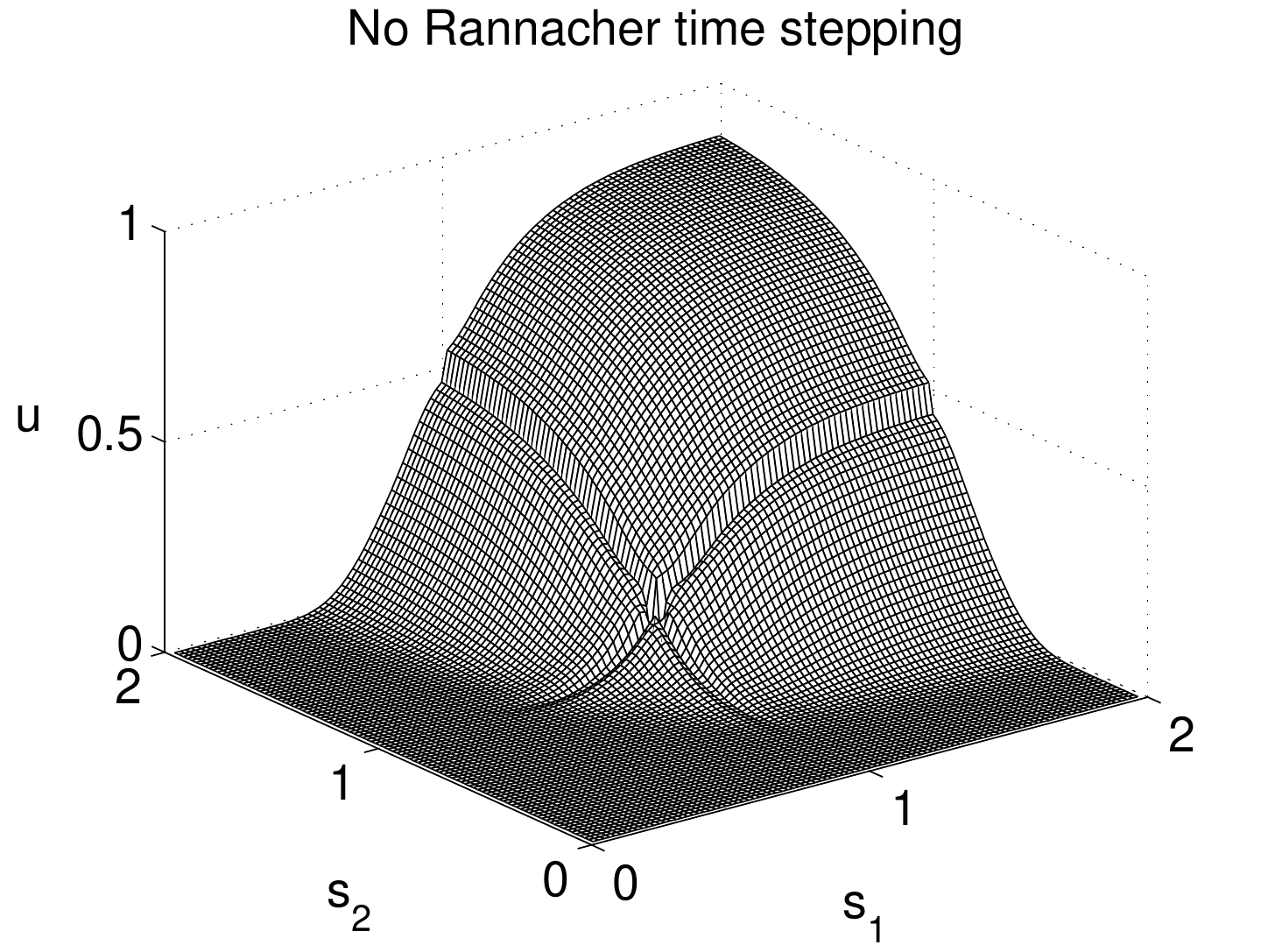}
\includegraphics[scale=0.5]{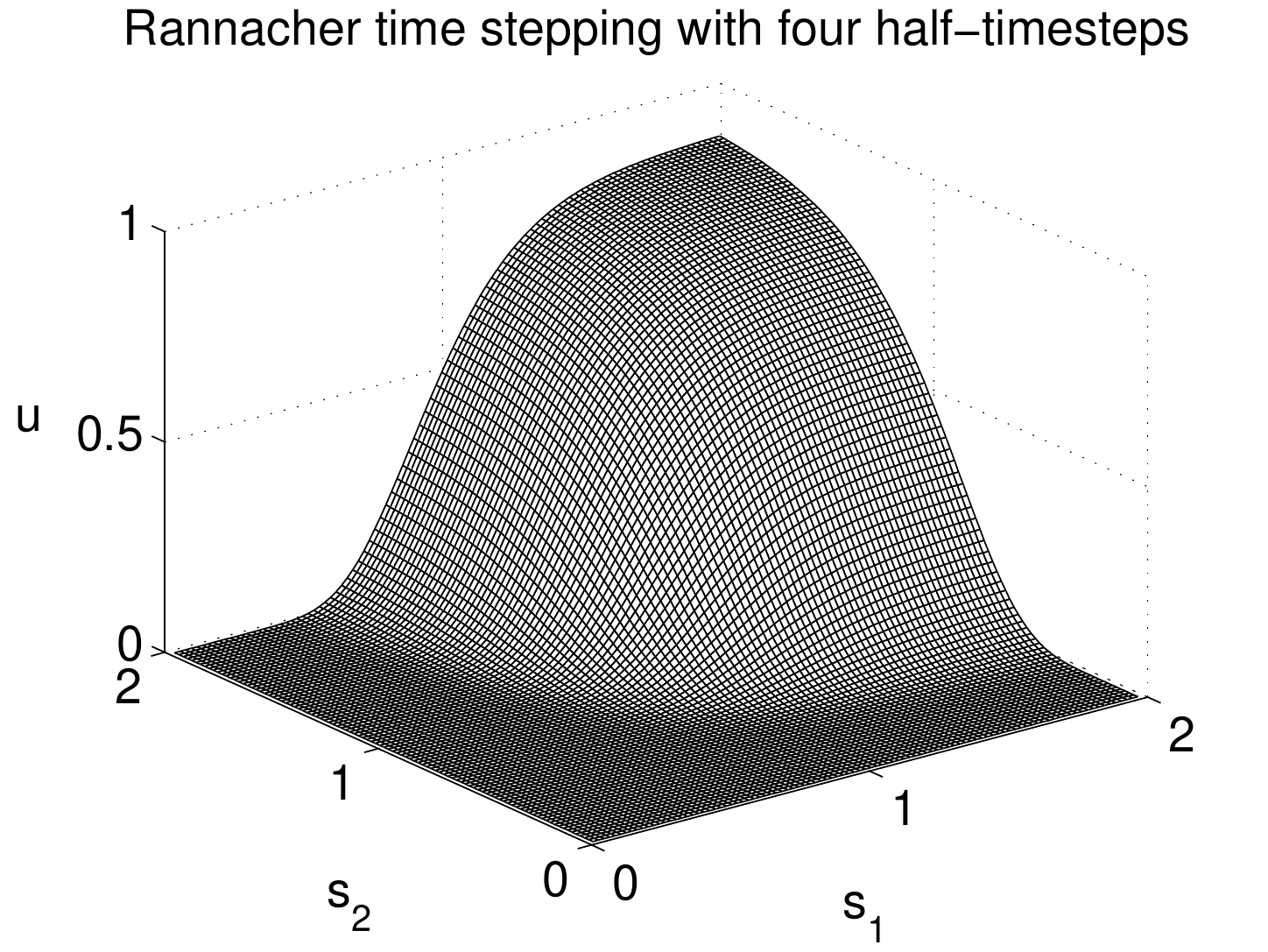} \newline \newline \newline
\includegraphics[scale=0.5]{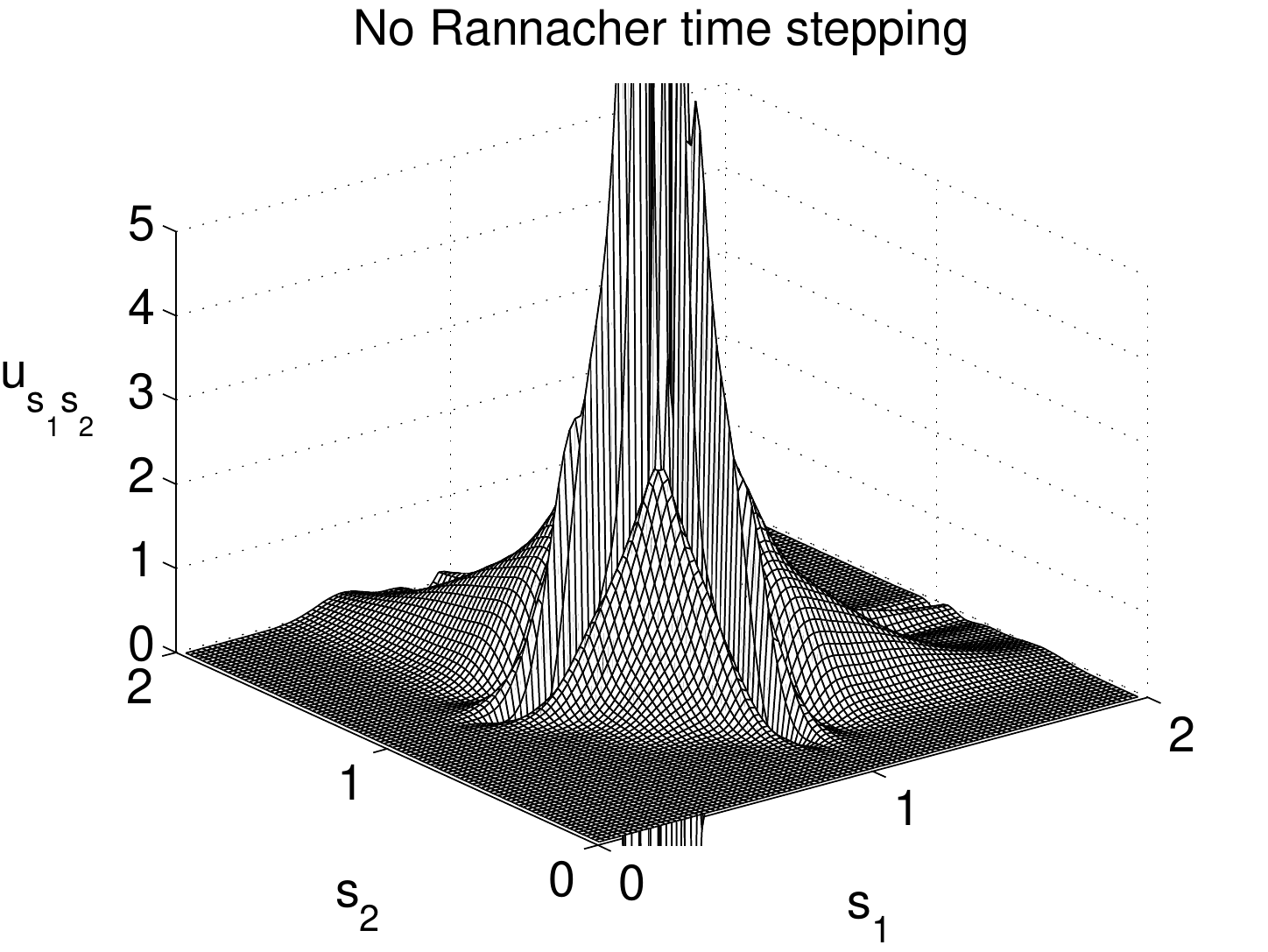}
\includegraphics[scale=0.5]{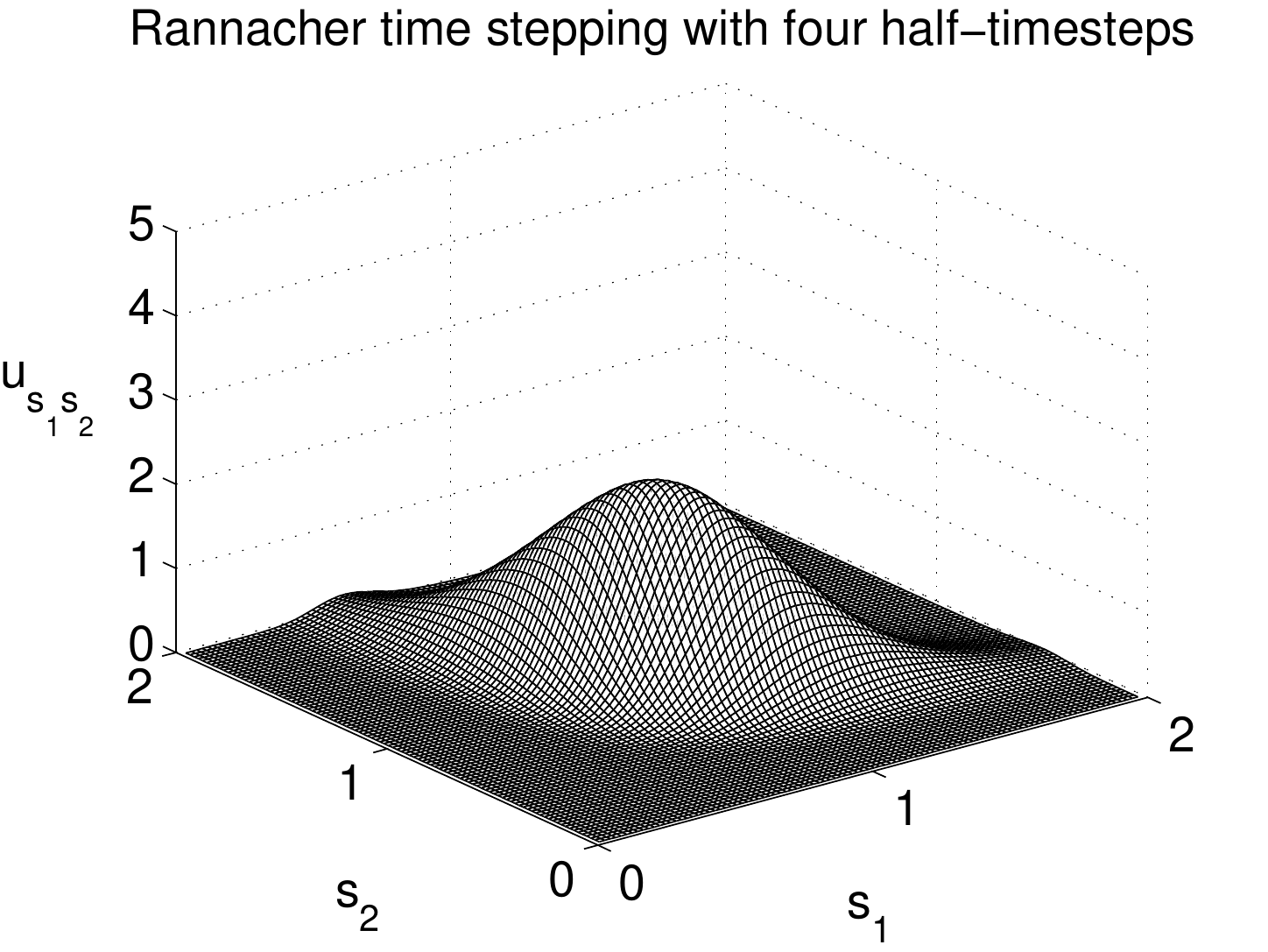}
\caption{Numerical approximations of the cash-or-nothing option value (top) and of its cross gamma (bottom) without (left) and with (right) Rannacher time stepping with four half-timesteps. The financial parameter values are $r=0.05,\ \sigma_{1} = 0.2,\ \sigma_{2} = 0.25,\ \rho = -0.7,\ K_{1} = 1,\ K_{2} = 1,\ T = 2$.}
\label{fig:DigitalCall2}
\end{center}
\end{figure}

\newpage
\setcounter{equation}{0}
\section{The Modified Craig--Sneyd scheme}\label{MCSScheme}

Semidiscretization by finite difference methods of initial-boundary value problems for time-dependent convection-diffusion equations leads to large systems of stiff ordinary differential equations (ODEs),
\begin{equation*}
U'(t)=F(t,U(t)) \quad (0 \leq t \leq T), \quad U(0)=U_{0},
\end{equation*}
with given operator $F$ and given initial value $U_{0}$. 
Assume the PDE is two-dimensional and the semidiscrete operator $F$ is decomposed into a sum
\begin{equation*}
F(t,v) = F_{0}(t,v) + F_{1}(t,v) + F_{2}(t,v) \quad (0 \leq t \leq T),
\end{equation*}
where $F_{0}$ represents the mixed spatial derivative term and $F_{1}, F_{2}$, represent all spatial derivative terms in the first, respectively, the second spatial direction.
Let $\theta > 0$ be a given parameter, $N \geq 1$ the number of timesteps and set $t_{n} = n\Delta t$ with $\Delta t = T/N$.
Then the \textit{Modified Craig--Sneyd (MCS) scheme} generates, in a one-step fashion, approximations $U_{n}$ to $U(t_{n})$ successively for $n= 1,2,\ldots,N$ through
\begin{equation}
\label{eq:MCS}
\left\{\begin{array}{l}
Y_0 = U_{n-1}+\Delta t\, F(t_{n-1},U_{n-1}), \\\\
Y_i = Y_{i-1}+\theta\Delta t \left(F_i(t_n,Y_i)-F_i(t_{n-1},U_{n-1})\right),
\quad i=1,2, \\\\
\widehat{Y}_0 = Y_0+ \theta \Delta t \left(F_0(t_n,Y_2)-F_0(t_{n-1},U_{n-1})\right),\\\\
\widetilde{Y}_0 = \widehat{Y}_0+ (\tfrac{1}{2}-\theta )\Delta t \left(F(t_n,Y_2)-F(t_{n-1},U_{n-1})\right),\\\\
\widetilde{Y}_i = \widetilde{Y}_{i-1}+\theta\Delta t \,(F_i(t_n,\widetilde{Y}_i)-F_i(t_{n-1},U_{n-1})),
\quad i=1,2, \\\\
U_n = \widetilde{Y}_2.
\end{array}\right.
\end{equation}
The MCS scheme \eqref{eq:MCS} was introduced by in 't Hout \& Welfert \cite{IHW09} for general multidimensional convection-diffusion problems with mixed derivative terms.
It can be viewed as an extension of the Craig--Sneyd (CS) scheme, proposed in \cite{CS88}. For $\theta = 1/2$, the MCS scheme reduces to the CS scheme.
Besides $\theta = 1/2$, common choices for $\theta$ in the literature are $\theta = 1/3$ and $\theta=1$.
Scheme \eqref{eq:MCS} starts with an explicit Euler stage applied to the full system, which is followed by two implicit corrections corresponding to each of the two spatial directions. Subsequently an explicit update is performed, followed again by two implicit unidirectional corrector stages.
Note that both $F$ and $F_{0}$, which contain the mixed derivative term, are always treated in an explicit manner. Each implicit stage handles spatial derivatives in only one spatial direction.
This can lead to a major computational advantage in comparison to classical non-splitted implicit time stepping methods.

\setcounter{equation}{0}
\section{Model problem}\label{ModelProblem}

Consider the coordinate transformation $x = \sqrt{2} \log(s_{1})/\sigma_{1}$ and $y = \sqrt{2} \log(s_{2})/\sigma_{2}$.
The PDE (\ref{eq:2DBlackScholes}) is then transformed into
\begin{equation*}
u_{t} = u_{xx} + 2\rho u_{xy} + u_{yy} + (\tfrac{\sqrt{2}r}{\sigma_{1}} - \tfrac{\sigma_{1}}{\sqrt{2}})u_{x} + (\tfrac{\sqrt{2}r}{\sigma_{2}} - \tfrac{\sigma_{2}}{\sqrt{2}})u_{y} - ru, 
\end{equation*}
for $-\infty < x_{1},x_{2} < \infty, \ 0 < t \leq T$.
This provides a motivation for considering a constant coefficient model convection-diffusion equation with mixed derivative term
\begin{equation}
u_{t} = u_{xx} + 2\rho u_{xy} + u_{yy} + a_{1}u_{x} + a_{2}u_{y}, 
\label{eq:ModelPDE}
\end{equation}
for $-\infty < x,y < \infty, \ 0 < t \leq T=1$ and with $\vert \rho \vert < 1$. 
We supplement equation (\ref{eq:ModelPDE}) with the initial condition
\begin{equation*}
u(x,y,0) = \delta(x)\delta(y),
\end{equation*}
which arises for example in the case of the cross gamma of a two-asset cash-or-nothing option. 
The Dirac delta initial function, however, has other important applications as well. For instance, it arises naturally in the adjoint equation for the joint density.
By using the \textit{Fourier transform} pair
\begin{eqnarray*}
\hat{u}(\kappa,\eta,t) &=& \int_{-\infty}^{\infty}\int_{-\infty}^{\infty} u(x,y,t) \exp(-\imi\kappa x)\exp(-\imi\eta y) dx dy, \\
u(x,y,t) &=& \frac{1}{4\pi^{2}} \int_{-\infty}^{\infty}\int_{-\infty}^{\infty} \hat{u}(\kappa,\eta,t) \exp(\imi\kappa x)\exp(\imi\eta y) d\kappa d\eta, 
\end{eqnarray*}
an exact closed-form analytical solution will be derived. Here $\imi$ denotes the imaginary unit. Taking the Fourier transformation of equation (\ref{eq:ModelPDE}) yields the ODE
$$ \hat{u}_{t} = - \kappa^{2} \hat{u} - 2\rho\kappa\eta \hat{u} - \eta^{2}\hat{u} + \imi a_{1}\kappa\hat{u} + \imi a_{2}\eta\hat{u}, $$
subject to initial condition $\hat{u}(\kappa,\eta,0)=1$. The solution of this transformed equation is given by
\begin{equation}
\hat{u}(\kappa,\eta,t) = \exp( -(\kappa^{2} + 2\rho\kappa\eta + \eta^{2} - \imi a_{1} \kappa - \imi a_{2} \eta) t). 
\label{eq:FourierTransformExact}
\end{equation}
Next, if $(X_{1},X_{2})$ is a multivariate normal distributed random variable with mean $(\mu_{1},\mu_{2})$ and covariance matrix $\Sigma$, its characteristic function is defined by
$$ \mathbb{E}[\exp(\imi\kappa X_{1})\exp(\imi\eta X_{2})] = \exp(\imi\kappa \mu_{1} + \imi \eta \mu_{2} - \tfrac{1}{2} (\kappa \ \eta) \Sigma (\kappa \ \eta)^{\top}). $$
By exploring the connection between the characteristic function of a random variable and the Fourier transform of its density function, it follows that $u(x,y,t)$ can be seen as the density function of a two-dimensional normal distributed random variable with mean $(\mu_{1}, \mu_{2})$ and covariance matrix $\Sigma$ given by
$$ (\mu_{1}, \mu_{2}) = (-a_{1}t, -a_{2}t) \quad \mbox{and} \quad \Sigma = \left( \begin{array}{cc}
2t & 2\rho t \\
2\rho t & 2t
\end{array} \right). $$
Since $\vert \rho \vert < 1$, this yields the closed-form analytical solution
\begin{equation*}
u(x,y,t) = \tfrac{1}{4\pi t \sqrt{1-\rho^{2}}} \exp\left( -\tfrac{1}{4t}\tfrac{1}{1-\rho^{2}}[ (x+a_{1}t)^{2} + (y+a_{2}t)^{2} - 2\rho(x+a_{1}t)(y+a_{2}t)] \right).
\end{equation*}

\setcounter{equation}{0}
\section{Discretization}\label{discretization}

As mentioned in Section \ref{intro}, spatial discretisation of (\ref{eq:ModelPDE}) will be performed on a uniform Cartesian grid with second-order central finite difference schemes. For the time integration the MCS scheme will be considered.
Let $h_{1}$ denote the spatial mesh width in the $x$-direction, $h_{2}$ the spatial mesh width in the $y$-direction
and define spatial gridpoints  $(x_{j}, y_{k}) = (jh_{1}, kh_{2})$ for all $j,k \in \mathbb{Z}$.
Semidiscretization of (\ref{eq:ModelPDE}) with second-order central finite difference schemes then gives rise to approximations $U_{j,k}(t)$ of the exact solution value $u(x_{j},y_{k},t)$ which are defined by the system
\begin{equation} 
U'_{j,k}(t) = AU_{j,k}(t),
\label{eq:SemidiscreteSystem}
\end{equation}
where $A = A_{0} + A_{1} + A_{2}$ and
\begin{eqnarray*}
A_{0} &=& \frac{\rho}{2h_{1}h_{2}}\delta_{2x}\delta_{2y}, \\\\
A_{1} &=& \frac{1}{h_{1}^{2}}\delta^{2}_{x} + \frac{a_{1}}{2h_{1}}\delta_{2x}, \\\\
A_{2} &=& \frac{1}{h_{2}^{2}}\delta^{2}_{y} + \frac{a_{2}}{2h_{2}}\delta_{2y},
\end{eqnarray*}
with $\delta_{2x},\ \delta^{2}_{x},\ \delta_{2y},\ \delta^{2}_{y}$ the usual second-order central finite difference operators.  For example,
\begin{eqnarray*} 
\delta_{2x} U_{j,k}(t) &=& U_{j+1,k}(t) - U_{j-1,k}(t), \\
\delta^{2}_{x} U_{j,k}(t) &=& U_{j-1,k}(t) - 2U_{j,k}(t) + U_{j+1,k}(t), \\
\delta_{2x}\delta_{2y}U_{j,k}(t) &=& U_{j+1,k+1}(t) + U_{j-1,k-1}(t) - U_{j+1,k-1} - U_{j-1,k+1}(t). 
\end{eqnarray*}
Semidiscrete system (\ref{eq:SemidiscreteSystem}) is provided with initial data
\[ U_{j,k}(0) = 
\begin{cases}
\frac{1}{h_{1}h_{2}} \qquad & \mbox{if } \ j = k =0, \\
0 & \mbox{else},
\end{cases} \]
in order to approximate the Dirac delta initial function.
For convenience we define
$$ Z = \Delta t A, \quad Z_{i} = \Delta t A_{i} \quad \mbox{for} \ i = 0,1,2,$$
and we denote by $I$ the identity operator.
Then, starting from $U_{0,j,k}=U_{j,k}(0)$, application of the MCS scheme to semidiscrete system (\ref{eq:SemidiscreteSystem}) yields approximations $U_{n,j,k}$ of $U_{j,k}(t_{n})$ successively for $n = 1,2,\ldots,N$ through
\begin{equation}
\label{eq:MCSOnModel}
\left\{\begin{array}{rcll}
Y_{0,j,k} &=& (I + Z) U_{n-1,j,k},& \\\\
(I - \theta Z_{i} ) Y_{i,j,k} &=& Y_{i-1,j,k} - \theta Z_{i} U_{n-1,j,k} &\quad i=1,2, \\\\
\widehat{Y}_{0,j,k} &=& Y_{0,j,k} + \theta Z_{0} Y_{2,j,k} - \theta Z_{0} U_{n-1,j,k},& \\\\
\widetilde{Y}_{0,j,k} &=& \widehat{Y}_{0,j,k} + (\tfrac{1}{2}-\theta ) Z Y_{2,j,k} - (\tfrac{1}{2}-\theta ) Z U_{n-1,j,k},& \\\\
(I - \theta Z_{i} ) \widetilde{Y}_{i,j,k} &=& \widetilde{Y}_{i-1,j,k} - \theta Z_{i} U_{n-1,j,k} &\quad i=1,2, \\\\
U_{n,j,k} &=& \widetilde{Y}_{2,j,k}.&
\end{array}\right.
\end{equation}
Concerning the Rannacher time stepping, let $N_{0}$ denote the number of initial MCS time steps replaced by $2N_{0}$ half-time steps of implicit Euler integration. Whenever $N_{0}>0$ scheme (\ref{eq:MCSOnModel}) is replaced by
\begin{equation}
\label{eq:RannacherStartUp}
\left\{ \begin{array}{rcl}
(I - \frac{1}{2}Z) U_{n-1/2,j,k} &=& U_{n-1,j,k}, \\\\
(I - \frac{1}{2}Z) U_{n,j,k} &=& U_{n-1/2,j,k},
\end{array} \right.
\end{equation}
for $n = 1,2,\ldots,\min\{N_{0},N\}$. 
This provides a numerical approximation $U_{N}$ of the exact solution.

The goal of our convergence analysis consists of quantifying the \textit{total error} 
\begin{equation}
\label{eq:TotalErrorDef}
U_{N,j,k} - u(x_{j},y_{k},1).
\end{equation}
To do so, we will analyse the asymptotic behaviour of a mixed discrete/continuous Fourier transform for $h_{1},h_{2},\Delta t$ simultaneously tending to zero.
Applying the inverse Fourier transformation on the resulting error in Fourier space will yield a useful bound for the total error \eqref{eq:TotalErrorDef}.
Special attention will be paid to the influence of $N_{0}$, i.e.\ the influence of Rannacher time stepping, on the total error.

\setcounter{equation}{0}
\section{Asymptotic analysis in Fourier space}\label{results}

We consider a mixed discrete/continuous Fourier transform pair, cf.\ e.g.\ \cite{S89},
\begin{eqnarray*}
 \widehat{V}(\vartheta_{1},\vartheta_{2}) = h_{1}h_{2} \sum_{j= -\infty}^{\infty}\sum_{k= -\infty}^{\infty} V_{j,k} \exp(-\imi j\vartheta_{1}) \exp(-\imi k\vartheta_{2}), \quad &\quad - \pi \leq \vartheta_{1}, \vartheta_{2} \leq \pi,& \\\\
 V_{j,k} = \frac{1}{4\pi^{2}h_{1}h_{2}} \int_{-\pi}^{\pi} \int_{-\pi}^{\pi} \widehat{V}(\vartheta_{1},\vartheta_{2}) \exp(\imi j\vartheta_{1}) \exp(\imi k\vartheta_{2})d\vartheta_{1} d\vartheta_{2}, & \quad j,k \in \mathbb{Z}. &
\end{eqnarray*}
For ease of presentation, the dependency of the Fourier transform on $\vartheta_{1}$ and $\vartheta_{2}$ will be omitted in the notation. 

Fourier transformation of $U_{0,j,k}$ yields $ \widehat{U}_{0} = 1$. 
Concerning operator $Z_{0}$ it follows that
\begin{eqnarray*}
\widehat{Z_{0}V} &=& h_{1}h_{2} \sum_{j= -\infty}^{\infty}\sum_{k= -\infty}^{\infty} Z_{0}V_{j,k}\exp(-\imi j\vartheta_{1}) \exp(-\imi k\vartheta_{2})  \\
		&=& \tfrac{\rho \Delta t}{2} \sum_{j= -\infty}^{\infty}\sum_{k= -\infty}^{\infty} ( V_{j+1,k+1} + V_{j-1,k-1} - V_{j+1,k-1} - V_{j-1,k+1} ) \exp(-\imi j\vartheta_{1}) \exp(-\imi k\vartheta_{2}) \\
		&=& \tfrac{\rho \Delta t}{2} \exp(\imi \vartheta_{1})\exp( \imi \vartheta_{2} ) \sum_{j= -\infty}^{\infty}\sum_{k= -\infty}^{\infty} V_{j+1,k+1} \exp(-\imi(j+1)\vartheta_{1}) \exp(-\imi(k+1)\vartheta_{2}) \\
		&& + \ \tfrac{\rho \Delta t}{2} \exp(-\imi \vartheta_{1})\exp( -\imi \vartheta_{2} ) \sum_{j= -\infty}^{\infty}\sum_{k= -\infty}^{\infty} V_{j-1,k-1} \exp(-\imi(j-1)\vartheta_{1}) \exp(-\imi(k-1)\vartheta_{2}) \\
		&& - \ \tfrac{\rho \Delta t}{2} \exp(\imi \vartheta_{1})\exp( -\imi \vartheta_{2} ) \sum_{j= -\infty}^{\infty}\sum_{k= -\infty}^{\infty} V_{j+1,k-1} \exp(-\imi(j+1)\vartheta_{1}) \exp(-\imi(k-1)\vartheta_{2}) \\
		&& - \ \tfrac{\rho \Delta t}{2} \exp(-\imi \vartheta_{1})\exp( \imi \vartheta_{2} ) \sum_{j= -\infty}^{\infty}\sum_{k= -\infty}^{\infty} V_{j-1,k+1} \exp(-\imi(j-1)\vartheta_{1}) \exp(-\imi(k+1)\vartheta_{2}) \\
		&=& \tfrac{\rho \Delta t}{2h_{1}h_{2}}\left[\exp(\imi \vartheta_{1})\exp( \imi \vartheta_{2} ) + \exp(-\imi \vartheta_{1})\exp(-\imi \vartheta_{2} ) - \exp(\imi \vartheta_{1})\exp( -\imi \vartheta_{2} ) - \exp(-\imi \vartheta_{1})\exp( \imi \vartheta_{2} ) \right]\widehat{V} \\
		&=& - \tfrac{2 \rho \Delta t}{h_{1}h_{2}} (\sin\vartheta_{1}\sin\vartheta_{2}) \widehat{V}.
\end{eqnarray*}
Analogously one finds
\begin{equation*}
\begin{array}{l}
\widehat{Z_{1}V} = \left( - \tfrac{4\Delta t}{h_{1}^{2}} \sin^{2} \tfrac{\vartheta_{1}}{2} + \imi a_{1} \tfrac{\Delta t}{h_{1}}\sin\vartheta_{1} \right) \widehat{V}, \\\\
\widehat{Z_{2}V} = \left( - \tfrac{4\Delta t}{h_{2}^{2}} \sin^{2} \tfrac{\vartheta_{2}}{2} + \imi a_{2} \tfrac{\Delta t}{h_{2}}\sin\vartheta_{2} \right) \widehat{V}. 
\end{array}
\end{equation*}
Define functions 
\begin{equation*}
\begin{array}{l}
z_{0} = z_{0}(\vartheta_{1},\vartheta_{2}) = - \frac{2 \rho \Delta t}{h_{1}h_{2}} \sin\vartheta_{1}\sin\vartheta_{2}, \\\\
z_{1} = z_{1}(\vartheta_{1},\vartheta_{2}) = - \frac{4\Delta t}{h_{1}^{2}} \sin^{2} \frac{\vartheta_{1}}{2} + \imi a_{1} \frac{\Delta t}{h_{1}}\sin\vartheta_{1}, \\\\
z_{2} = z_{2}(\vartheta_{1},\vartheta_{2}) = - \frac{4\Delta t}{h_{2}^{2}} \sin^{2} \frac{\vartheta_{2}}{2} + \imi a_{2} \frac{\Delta t}{h_{2}}\sin\vartheta_{2},
\end{array}
\end{equation*}
and $z = z_{0} + z_{1} + z_{2}$. Then, Fourier transformation of the implicit Euler scheme (\ref{eq:RannacherStartUp}) gives
$$ \widehat{U}_{n} = \left( \frac{1}{1 - \tfrac{1}{2}z} \right)^{2} \widehat{U}_{n-1}. $$
After some calculations, Fourier transformation of the MCS scheme (\ref{eq:MCSOnModel}) yields
\begin{equation*}
\widehat{U}_{n} = R \widehat{U}_{n-1}, 
\end{equation*}
with
\begin{equation*} 
R = 1 + \frac{z}{p} + \frac{(\theta z_{0} + (\tfrac{1}{2} - \theta) z) z}{p^{2}},  
\end{equation*}
where 
\begin{equation}
\label{eq:defp}
p = (1-\theta z_{1})(1 - \theta z_{2} ).
\end{equation}
Assume that $N_{0} \leq N$. Since $\widehat{U}_{0} = 1$ it follows that
\begin{equation}
\label{eq:FourierTransformNumericalSolution}
\widehat{U}_{N} = R^{N-N_{0}} \left( \frac{1}{1 - \tfrac{1}{2}z} \right)^{2N_{0}}. 
\end{equation}
By applying the inverse Fourier transformation, the numerical approximation at $t=T=1$ can be written as
\begin{eqnarray*}
U_{N,j,k} &=& \frac{1}{4\pi^{2} h_{1}h_{2}} \int_{-\pi}^{\pi}\int_{-\pi}^{\pi} \widehat{U}_{N}(\vartheta_{1},\vartheta_{2})\exp(\imi j\vartheta_{1}) \exp(\imi k\vartheta_{2})  d\vartheta_{1} d\vartheta_{2} \\
			&=&  \frac{1}{4\pi^{2}} \int_{-\pi/h_{2}}^{\pi/h_{2}}\int_{-\pi/h_{1}}^{\pi/h_{1}} \widehat{U}_{N}(\kappa h_{1},\eta h_{2}) \exp(\imi x_{j} \kappa) \exp(\imi y_{k} \eta) d\kappa d\eta,
\end{eqnarray*}
where we made use of the substitutions 
$$\vartheta_{1} = \kappa h_{1}, \quad \vartheta_{2} = \eta h_{2}.$$
From Section \ref{ModelProblem} it can be seen that the exact solution is given by
$$ u(x,y,1) = \frac{1}{4\pi^{2}} \int_{-\infty}^{\infty}\int_{-\infty}^{\infty} \widehat{u}(\kappa,\eta,1) \exp(\imi x \kappa)\exp(\imi y \eta) d\kappa d \eta. $$
In our analysis, we will examine the \textit{Fourier error}
\begin{equation}
\widehat{U}_{N}(\kappa h_{1}, \eta h_{2}) - \widehat{u}(\kappa,\eta,1) \qquad \mbox{for} \ -\pi \leq \kappa h_{1}, \eta h_{2}  \leq \pi.
\label{eq:FourierErrorDef}
\end{equation}
For $h_{1},h_{2}$ tending to zero, the \textit{total error} \eqref{eq:TotalErrorDef} is approximated by
\begin{equation}
\frac{1}{4\pi^{2}} \int_{-\pi/h_{2}}^{\pi/h_{2}}\int_{-\pi/h_{1}}^{\pi/h_{1}} \left( \widehat{U}_{N}(\kappa h_{1}, \eta h_{2}) - \widehat{u}(\kappa,\eta,1) \right) \exp(\imi x_{j} \kappa) \exp(\imi y_{k} \eta) d\kappa d\eta.
\label{eq:ApproxTotalError}
\end{equation} 
Note that expression \eqref{eq:ApproxTotalError} can be viewed as the inverse mixed discrete/continuous Fourier transform of the Fourier error \eqref{eq:FourierErrorDef}.

In Figure \ref{fig:NormTransformExact}, $\vert \widehat{u} \vert$ is shown in the $(\vartheta_{1},\vartheta_{2})$-domain for parameter values $\rho = -0.7, a_{1} = 2, a_{2} = 3$. 
This has to be compared with Figure \ref{fig:NormTransform} where $\vert \widehat{U}_{N} \vert$ is shown for the same parameter values. Discretization is performed with $h_{1}=h_{2}=1/6, \Delta t = 1/8$ and well-known MCS parameters $\theta = 1/3, 1/2, 1$. For the Rannacher time stepping we considered values $N_{0}=0,2.$
From Figure \ref{fig:NormTransformExact} and Figure \ref{fig:NormTransform} it is clear that the difference $\widehat{U}_{N}-\widehat{u}$ has different properties in different regions of the Fourier domain.
These regions are illustrated in Figure \ref{fig:Regions}.

First there is a \textit{low-wavenumber region} \raisebox{.5pt}{\textcircled{\raisebox{-.9pt} {1}}}, where both $\vert \vartheta_{1} \vert$ and $\vert \vartheta_{2} \vert$ are small, in which there is a good agreement between $\widehat{U}_{N}$ and $\widehat{u}$. 
Next, if either $\vert \vartheta_{1} \vert$ or $\vert \vartheta_{2} \vert$ is medium and the other one is small or medium (region \raisebox{.5pt}{\textcircled{\raisebox{-.9pt} {2}}}), then both the Fourier transforms of the numerical solution and analytical solution are negligible.
In the \textit{high-wavenumber region} \raisebox{.5pt}{\textcircled{\raisebox{-.9pt} {3}}}, i.e.\ where both $\vert\vartheta_{1}\vert, \vert\vartheta_{2}\vert$ are large, we observe that the modulus of the Fourier transform $\widehat{u}$ is close to zero. 
The modulus $\vert \widehat{U}_{N}\vert$, however, is strongly dependent on $N_{0}$ and the MCS parameter $\theta$.
For larger values of $\theta$ we see that $\widehat{U}_{N}$ has a larger magnitude in the high-wavenumber region.
Hence, a larger high-wavenumber error can be expected for larger values of $\theta$.
Further we observe that the modulus of $\widehat{U}_{N}$ in the high-wavenumber region is always damped whenever Rannacher time stepping is applied. 
This matches our observations from Figure \ref{fig:DigitalCall2} where unwanted erratic behaviour was avoided by using Rannacher time stepping. 
Finally, we have the case where either $\vert \vartheta_{1} \vert$ or $\vert \vartheta_{2} \vert$ is large but the other one is not. 
In our analysis, the region \raisebox{.5pt}{\textcircled{\raisebox{-.9pt} {4}}} where $\vert \vartheta_{1} \vert$ is large and the region \raisebox{.5pt}{\textcircled{\raisebox{-.9pt} {5}}} where $\vert \vartheta_{2} \vert$ is large will be treated separately.
In both regions the Fourier transform $\widehat{u}$ is negligible but $\widehat{U}_{N}$ has to be further analysed.
In particular, we will show that $\widehat{U}_{N}$ is not negligible if the MCS scheme reduces to the CS scheme.

Following Giles \& Carter \cite{GC06} we will perform an asymptotic analysis of the Fourier error $\widehat{U}_{N}-\widehat{u}$ in each of these (five) disjoint regions which form a partition of the Fourier domain. 
We consider the limit $h_{1},h_{2},\Delta t \rightarrow 0$ and since the same discretization is performed in both spatial directions, 
$$ c = h_{2}/h_{1}$$
\textit{is held fixed}. For ease of presentation we denote $h=h_{1}$. Further, since both the semidiscretization and the time integration are convergent of order two for smooth initial data, it seems natural to keep 
$$\lambda = \Delta t/ h$$ 
\textit{constant}. 
Substitutions $\vartheta_{1} = \kappa h_{1}, \vartheta_{2} = \eta h_{2}$ yield
\begin{subeqnarray}
&& z_{0} = - \tfrac{2\rho\lambda}{ch}\sin \kappa h \sin c \eta h = - \tfrac{\rho\lambda}{ch}( \cos((\kappa - c\eta)h) - \cos((\kappa + c\eta)h)), \slabel{eq:z0}	\\
&& z_{1} = - \tfrac{4\lambda}{h}\sin^{2} \tfrac{\kappa h}{2} + \imi a_{1} \lambda \sin \kappa h = - \tfrac{2\lambda}{h}(1-\cos \kappa h) + \imi a_{1} \lambda \sin \kappa h, \slabel{eq:z1} \\
&& z_{2} = - \tfrac{4\lambda}{c^{2}h}\sin^{2} \tfrac{c\eta h}{2} + \imi a_{2} \tfrac{ \lambda}{c} \sin c\eta h = - \tfrac{2\lambda}{c^{2}h}(1-\cos c\eta h) + \imi a_{2} \tfrac{ \lambda}{c} \sin c\eta h. \slabel{eq:z2}
\label{eq:z0z1z2}
\end{subeqnarray} 
The expressions in (\ref{eq:z0z1z2}) will be used to analyse the asymptotic behaviour of (\ref{eq:FourierTransformNumericalSolution}) as $h \rightarrow 0$. 
\textit{Throughout the analysis, by the notation $\mathcal{O}\left( f(\kappa,\eta,h) \right)$ we shall always mean that the modulus $\vert \cdot \vert$ of the term under consideration is bounded by a positive constant times $f(\kappa,\eta,h)$ where the constant is independent of $\kappa, \eta$ and the mesh width $h$.}
In order to deal with the powers in expression \eqref{eq:FourierTransformNumericalSolution} a $\log$-transformation of $\widehat{U}_{N}$ will be considered.
Since $T=1$, thus $N = 1/(\lambda h)$, it follows that
\begin{subeqnarray}
\log \widehat{U}_{N} &=& (N-N_{0}) \log\left( R  \right) + 2N_{0} \log\left( \tfrac{1}{1-z/2} \right) \slabel{eq:LogTransformFourierTransformNumercialSolution} \\
			&=&  \tfrac{1}{\lambda h} \left[ \log( p^{2} + pz + \theta z_{0}z + (\tfrac{1}{2}-\theta)z^{2}) - 2 \log(p) \right] \slabel{eq:LogTransformFourierTransformNumercialSolution1} \\
			&& + \ N_{0} \left[ 2\log(p) - \log(p^{2} + pz + \theta z_{0}z + (\tfrac{1}{2}-\theta)z^{2}) - 2 \log(1 - \tfrac{1}{2}z) \right].
			\slabel{eq:LogTransformFourierTransformNumercialSolution2}
\end{subeqnarray}

\begin{figure}
\begin{center}
\includegraphics[scale=0.5]{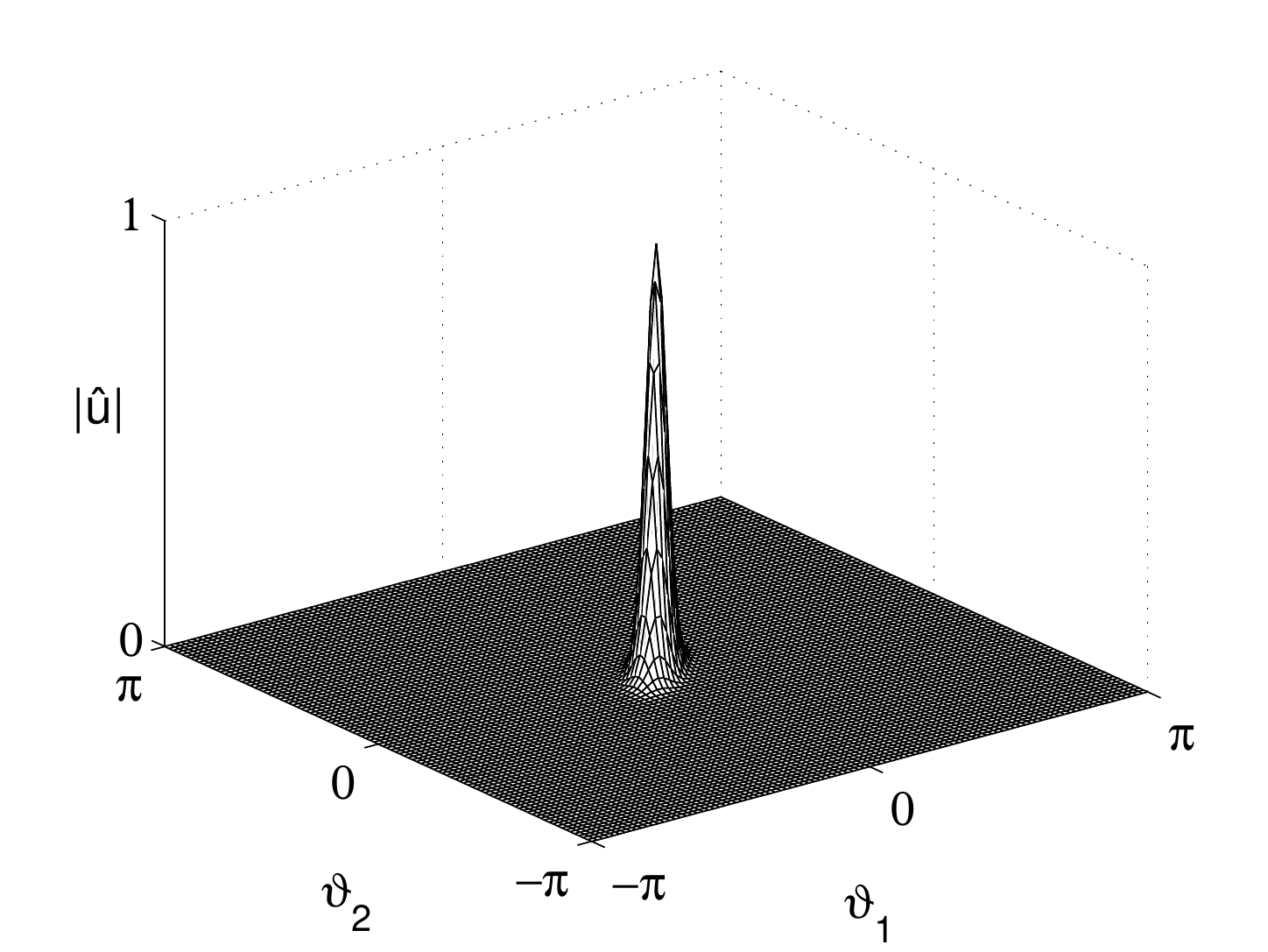} \\
\caption{Magnitude of the Fourier transform of the exact solution $u(x,y,1)$ for parameter values $\rho = -0.7, a_{1} = 2, a_{2} = 3$.}
\label{fig:NormTransformExact}
\end{center}
\end{figure}

\begin{figure}
\begin{center}
\includegraphics[scale=0.5]{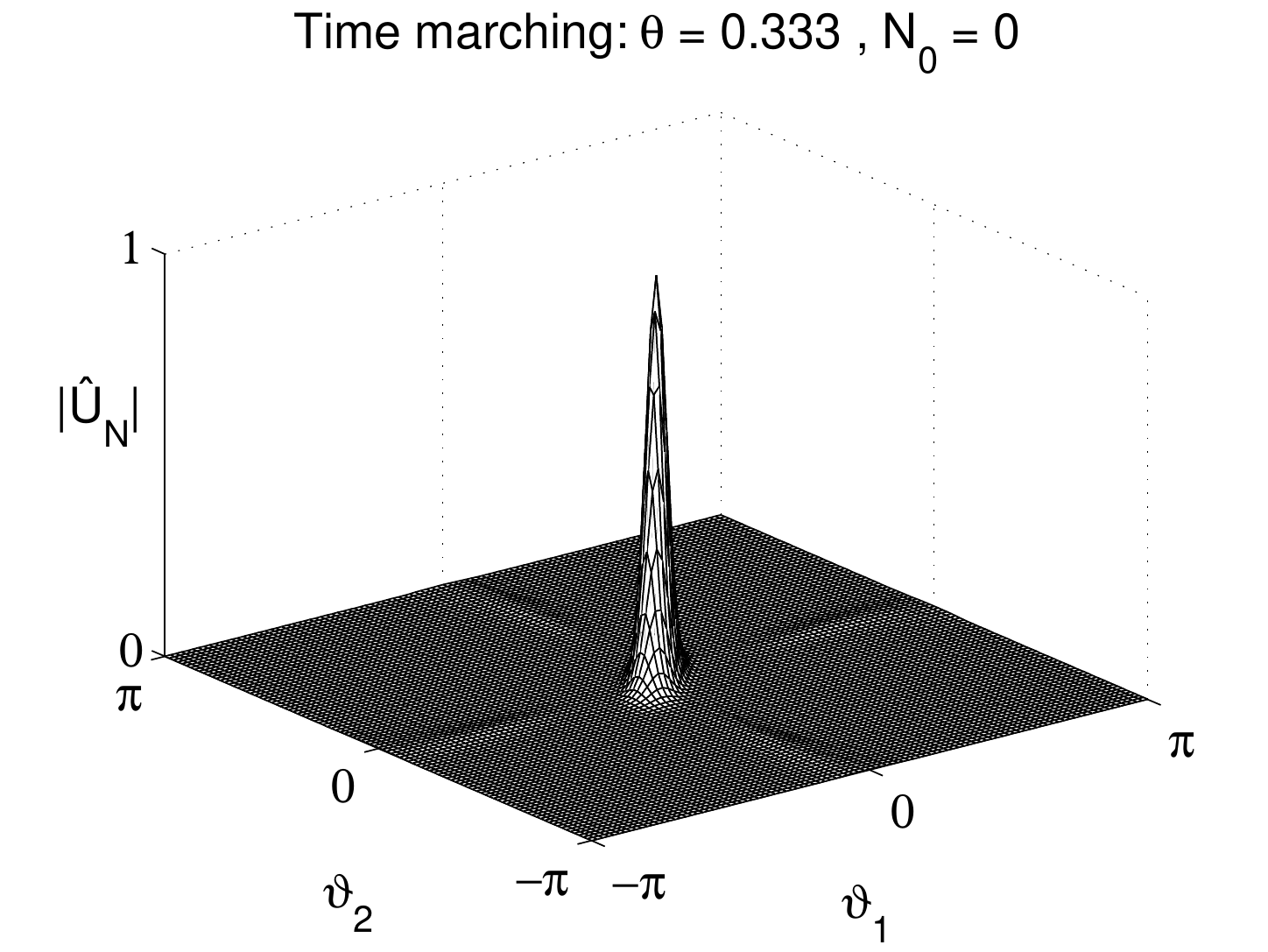} 
\includegraphics[scale=0.5]{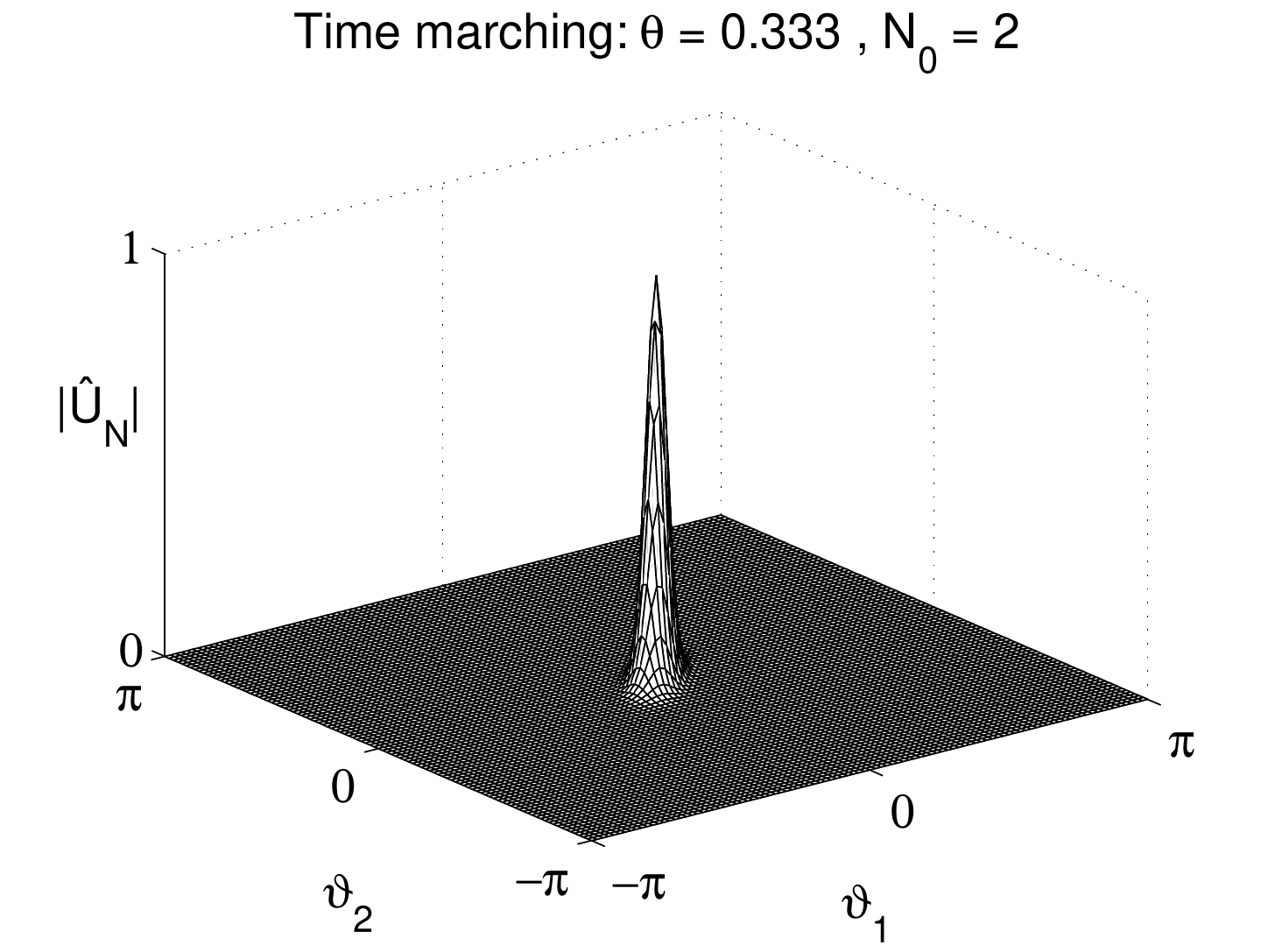} \newline \newline \newline
\includegraphics[scale=0.5]{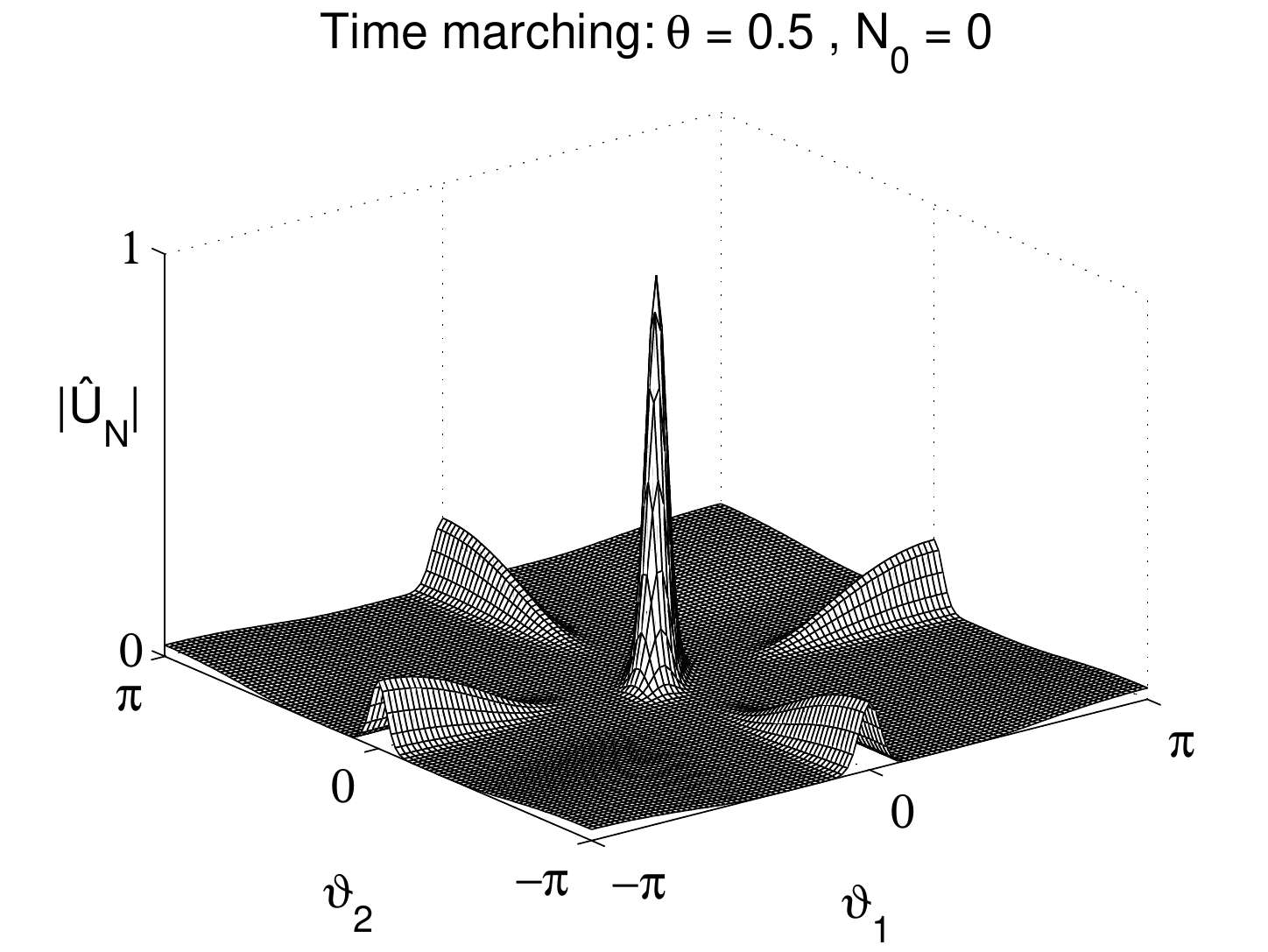} 
\includegraphics[scale=0.5]{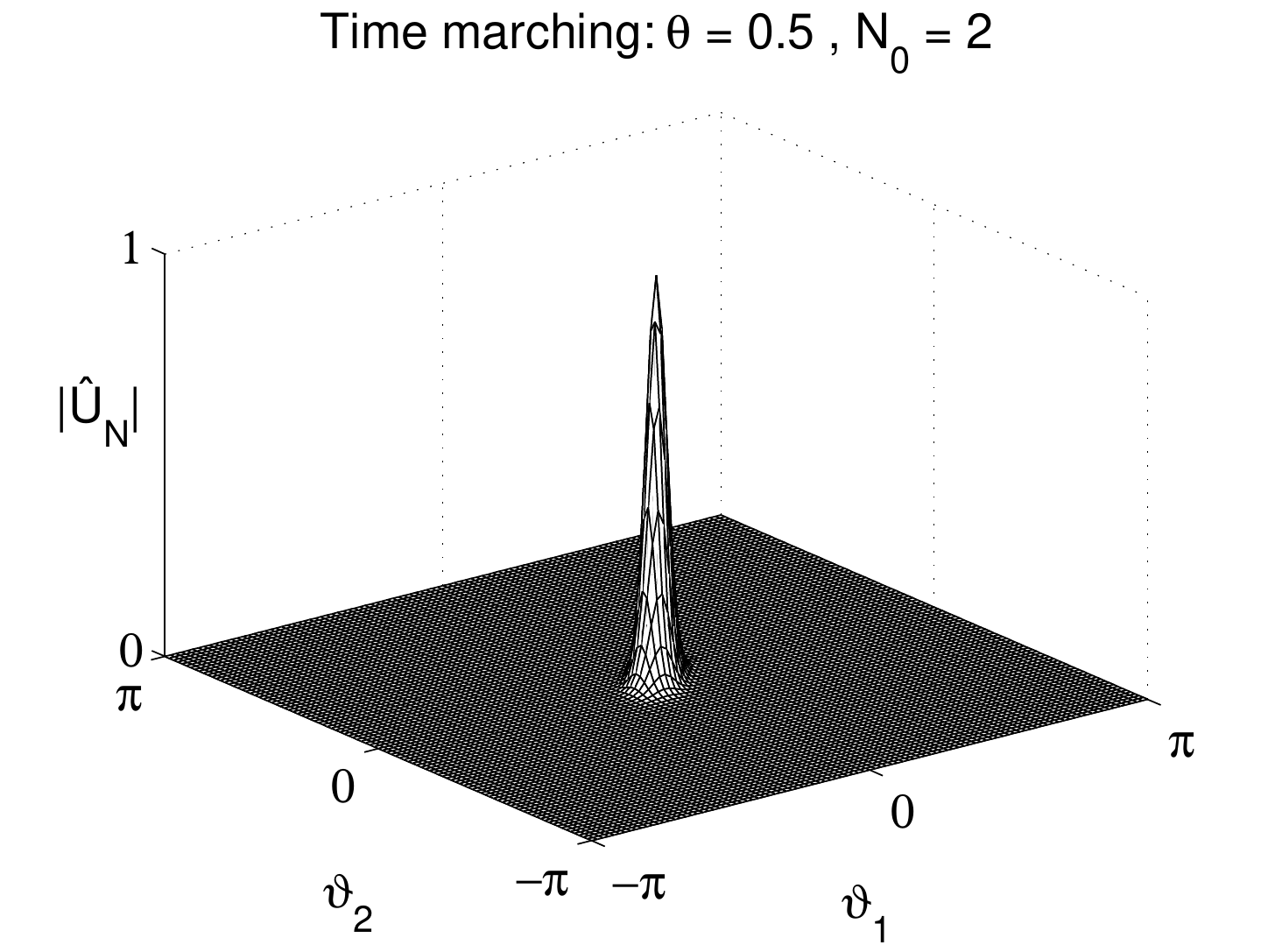} \newline \newline \newline
\includegraphics[scale=0.5]{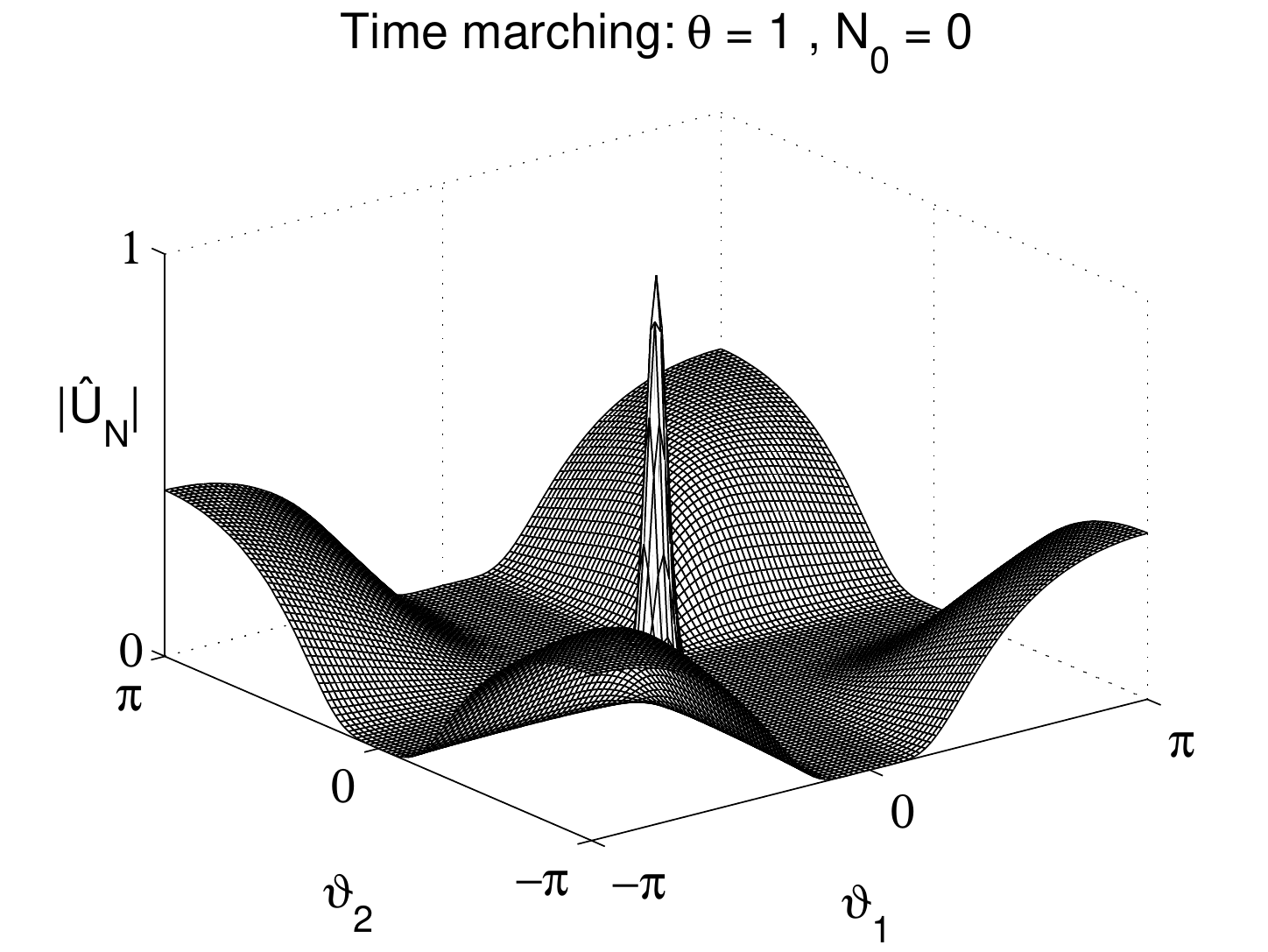} 
\includegraphics[scale=0.5]{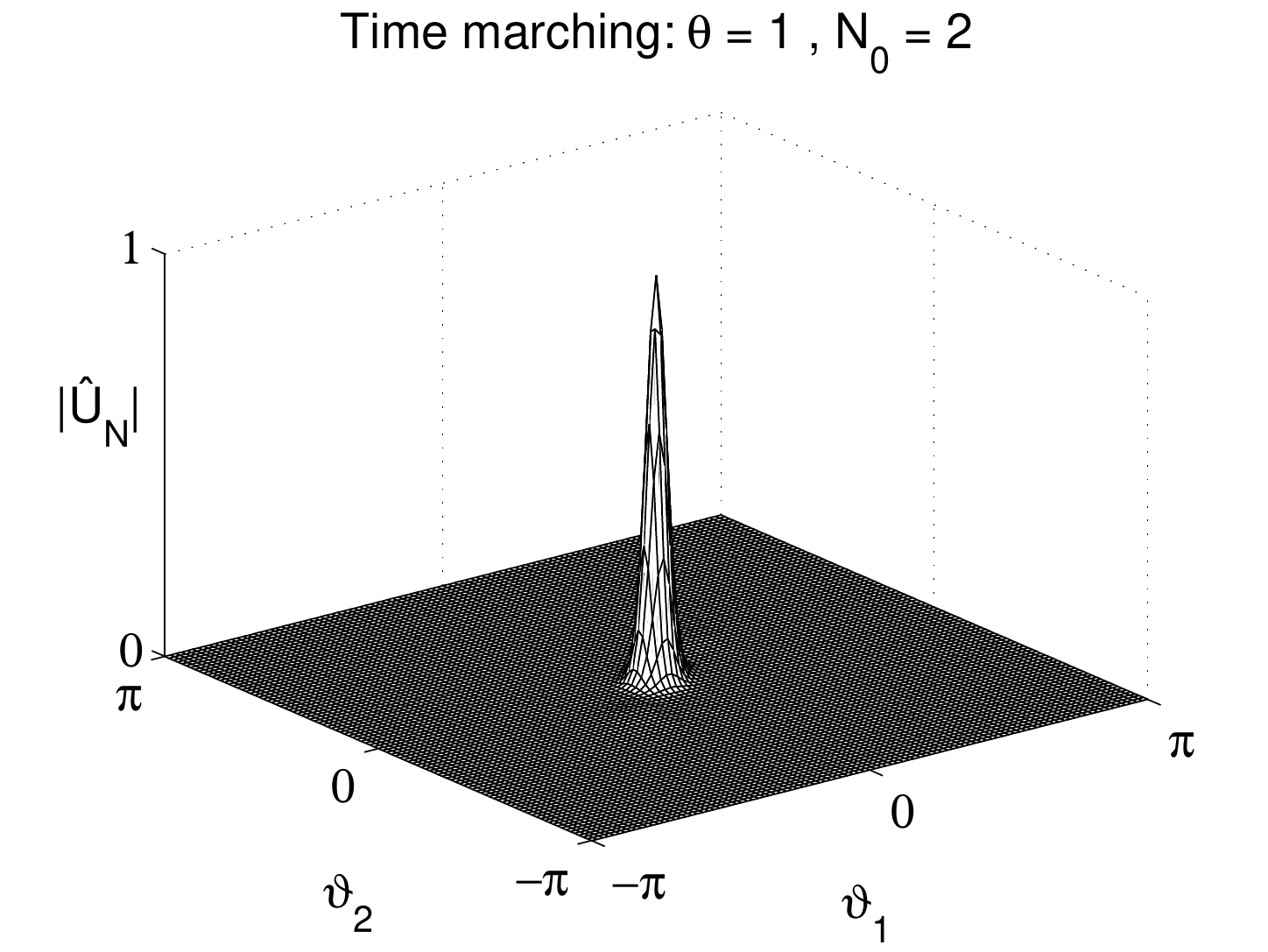}
\caption{Magnitude of the Fourier transform $\widehat{U}_{N}$ with $N_{0}=0$ (left) and $N_{0}=2$ (right) for MCS parameter $\theta = 1/3$ (top), $\theta = 1/2$ (middle) and $\theta = 1$ (bottom). The other parameter values are: $\rho = -0.7, a_{1} = 2, a_{2} = 3, h_{1}=h_{2}=1/6, \Delta t = 1/8$. }
\label{fig:NormTransform}
\end{center}
\end{figure}

\begin{figure}
\begin{center}
\includegraphics[scale=0.5]{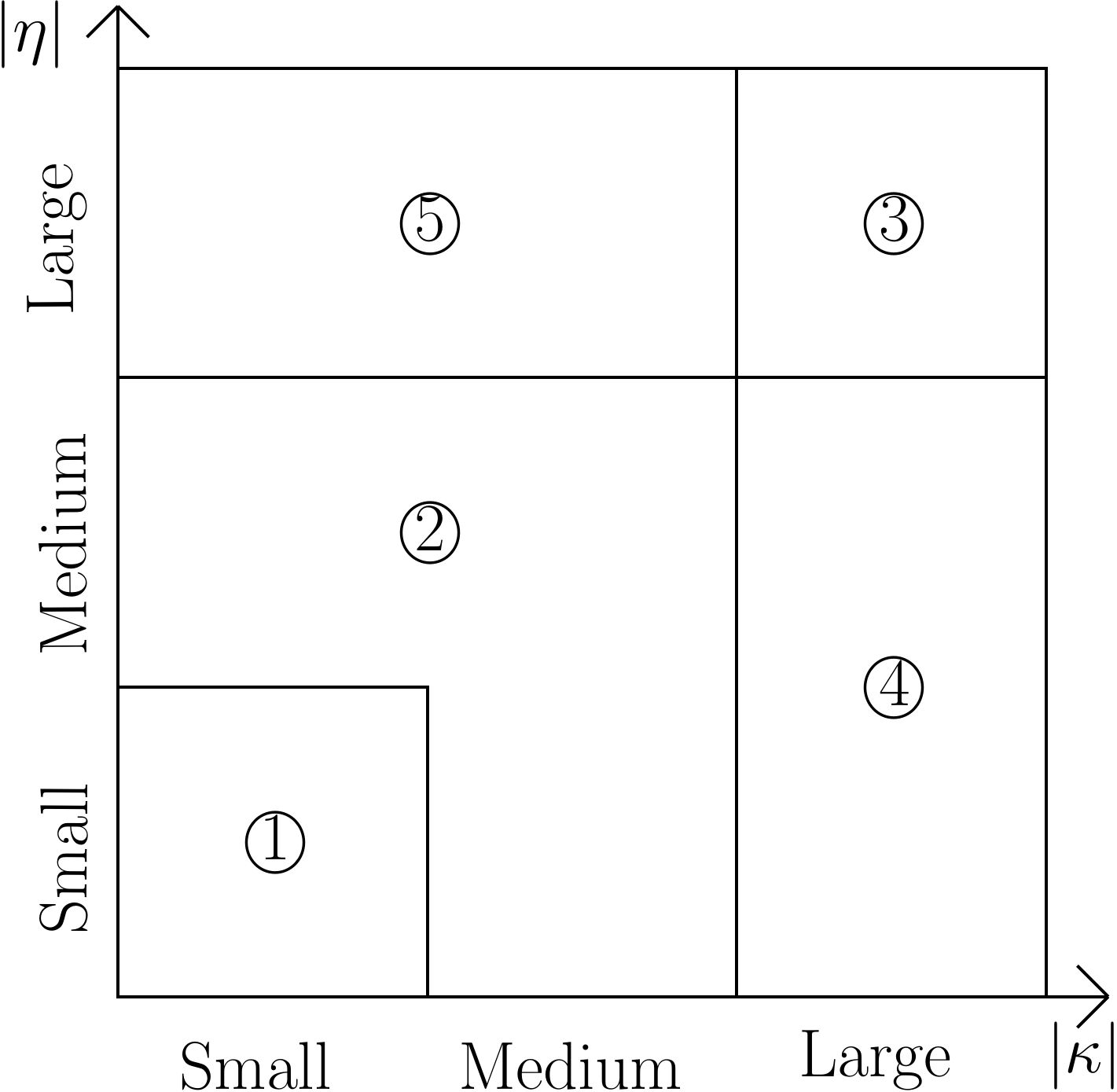} 
\caption{Illustration of the different disjoint regions of the Fourier domain.}
\label{fig:Regions}
\end{center}
\end{figure}

\subsection{Taylor expansion of $\widehat{U}_{N}$}
\label{subsec:Taylor}

Multiple regions will encounter values $\vert \kappa \vert, \vert c\eta \vert \leq h^{-q}$ with certain $q \leq 1/2$. 
By \textit{Taylor expansion} of (\ref{eq:z0}) it directly follows that
\begin{eqnarray*}
z_{0}(h) &=& - \frac{\rho\lambda}{ch} \left( \frac{(\kappa + c\eta)^{2}h^{2}}{2} - \frac{(\kappa - c\eta)^{2}h^{2}}{2} - \frac{(\kappa + c\eta)^{4}h^{4}}{4!} + \frac{(\kappa - c\eta)^{4}h^{4}}{4!} + \cdots \right) \\
	  &=& - \frac{\rho\lambda}{c} \left( 2\kappa c \eta h - \frac{1}{3}(\kappa^{2} + c^{2}\eta^{2} ) \kappa c \eta h^{3} + \cdots	\right) \\
	  &=& z_{0}^{[1]} h + z_{0}^{[3]} h^{3} + z_{0}^{[5]} h^{5},
\end{eqnarray*}
where
\begin{eqnarray*}
z_{0}^{[1]} &=& -2\rho \lambda \kappa \eta,	\\
z_{0}^{[3]} &=& \tfrac{1}{3}\rho \lambda (\kappa^{2} + c^{2}\eta^{2})\kappa\eta,	\\
\vert z_{0}^{[5]} \vert &\leq & \tfrac{4}{6!}\tfrac{\rho\lambda}{c} (\vert \kappa \vert + c \vert \eta \vert )^{6}.  	
\end{eqnarray*}
Analogously as above, Taylor expansion of (\ref{eq:z1}) and (\ref{eq:z2}) yields
\begin{eqnarray*}
z_{1}(h) &=& z_{1}^{[1]} h + z_{1}^{[3]} h^{3} + z_{1}^{[5]} h^{5},  \\
z_{2}(h) &=& z_{2}^{[1]} h + z_{2}^{[3]} h^{3} + z_{2}^{[5]} h^{5}, 
\end{eqnarray*}
where
\begin{eqnarray*}
z_{1}^{[1]} &=&  -\lambda \kappa^{2} + \imi a_{1} \lambda \kappa, \\
z_{1}^{[3]} &=&  \tfrac{1}{12}\lambda\kappa^{4} - \tfrac{1}{6} \imi a_{1} \lambda \kappa^{3}, \\
\vert z_{1}^{[5]} \vert &\leq & \tfrac{2}{6!}\lambda\kappa^{6} + \tfrac{1}{5!} \vert a_{1} \vert \lambda \vert \kappa \vert^{5},   \\\\
z_{2}^{[1]} &=& - \lambda \eta^{2} + \imi a_{2} \lambda \eta,  \\
z_{2}^{[3]} &=& \tfrac{1}{12}\lambda c^{2} \eta^{4} - \tfrac{1}{6} \imi a_{2} \lambda c^{2} \eta^{3},  \\
\vert z_{2}^{[5]} \vert &\leq &  \tfrac{2}{6!} \lambda c^{4}\eta^{6} + \tfrac{1}{5!}\vert a_{2} \vert \lambda c^{4} \vert \eta \vert^{5}. 
\end{eqnarray*}
Since $q \leq 1/2$, it is ensured that all terms in the above expansions stay bounded as $h$ tends to zero.
Using these expansions and the definition \eqref{eq:defp} of $p$ it follows that
\begin{equation*}
p(h) = 1 + p^{[1]}h + p^{[2]}h^{2} + p^{[3]}h^{3} + p^{[4]}h^{4} + p^{[5]}h^{5},
\end{equation*}
where
\begin{eqnarray*}
p^{[1]} &=& -\theta (z_{1}^{[1]} + z_{2}^{[1]}), \\
p^{[2]} &=&  \theta^{2}z_{1}^{[1]}z_{2}^{[1]},  \\
p^{[3]} &=& - \theta (z_{1}^{[3]} + z_{2}^{[3]}), \\
p^{[4]} &=& \theta^{2} (z_{1}^{[1]}z_{2}^{[3]} + z_{1}^{[3]}z_{2}^{[1]}), \\
p^{[5]} &=& \mathcal{O}\left(1+(\kappa^{2} + c^{2}\eta^{2})^{3}\right).
\end{eqnarray*}
Under the condition $\vert \kappa \vert, \vert c\eta \vert \leq h^{-q}$ with certain $q \leq 1/2$, the variables $\kappa$ and $\eta$ can become very large as $h$ tends to zero. In this case the highest powers of $\kappa, \eta$ will dominate the order term in $p^{[5]}$.
Under the same condition, however, $\kappa$ and $\eta$ can both be very small and then the lowest powers of $\kappa,\eta$ will dominate. 
By considering the sum of $1$ and the highest powers of $\kappa,\eta$ in the remaining order term, we ensure that both cases are covered. 

As mentioned above we will make use of $\log$-transformation (\ref{eq:LogTransformFourierTransformNumercialSolution}) to analyse the asymptotic behaviour. Let $f$ be a strictly positive and sufficiently smooth function and set 
$$g(h) = \log(f(h)) \quad \mbox{for} \ h\geq 0.$$
Taylor expansion yields
\begin{equation}
\label{eq:TaylorExpansionLogFunction}
g(h) = \log(f(0)) + g^{[1]}h + g^{[2]}h^{2} + g^{[3]}h^{3} + g^{[4]}h^{4},
\end{equation}
where
\begin{eqnarray*}
g^{[1]} &=& \tfrac{f'(0)}{f(0)}, \\
g^{[2]} &=& \tfrac{1}{2}\left( \tfrac{f''(0)}{f(0)} - \tfrac{f'(0)^{2}}{f(0)^{2}} \right), \\
g^{[3]} &=& \tfrac{1}{6}\left( \tfrac{f'''(0)}{f(0)} - 3\tfrac{f'(0)f''(0)}{f(0)^{2}} + 2 \tfrac{f'0)^{3}}{f(0)^{3}} \right), \\
g^{[4]} &=& \tfrac{1}{4!} \left( \tfrac{f^{(4)}(\xi)}{f(\xi)} - \tfrac{4f'(\xi)f'''(\xi) + 3 f''(\xi)^{2}}{f(\xi)^{2}} + 12 \tfrac{f'(\xi)^{2}f''(\xi)}{f(\xi)^{3}} - 6 \tfrac{f'(\xi)^{4}}{f(\xi)^{4}} \right) \quad \mbox{with certain} \ 0 < \xi < h.
\end{eqnarray*}
In order to encounter the first part of (\ref{eq:LogTransformFourierTransformNumercialSolution1}) consider
$$ f_{M}(h) = p(h)^{2} + p(h)z(h) + \theta z_{0}(h)z(h) + (\tfrac{1}{2} - \theta) z(h)^{2}, $$
so that
\begin{eqnarray*}
f_{M}'(h) &=& 2p(h)p'(h) + p'(h)z(h) + p(h)z'(h) + \theta(z_{0}'(h) z(h) + z_{0}(h) z'(h)) + 2(\tfrac{1}{2} - \theta)z(h)z'(h) , \\
f_{M}''(h) &=& 2p'(h)^{2} + 2p(h)p''(h) + p''(h)z(h) + 2p'(h)z'(h) + p(h)z''(h) \\ 
		&& + \ \theta( z_{0}''(h)z(h) + 2 z_{0}'(h)z'(h) + z_{0}(h)z''(h)) + 2(\tfrac{1}{2} - \theta) ( z'(h)^{2} + z(h)z''(h)), \\
f_{M}'''(h) &=& 6p'(h)p''(h) + 2p(h)p'''(h) + p'''(h)z(h) + 3 p''(h)z'(h) + 3 p'(h)z''(h) + p(h)z'''(h) \\
			&& + \ \theta( z_{0}'''(h)z(h) + 3z_{0}''(h)z'(h) + 3z_{0}'(h)z''(h) + z_{0}(h)z'''(h)) \\
			&& + \ 2(\tfrac{1}{2}-\theta)( 3z'(h)z''(h) + z(h)z'''(h)), \\
f_{M}^{(4)}(h) &=& \mathcal{O}\left( 1+(\kappa^{2} + c^{2}\eta^{2} )^{4} \right),
\end{eqnarray*}
and thus
\begin{eqnarray*}
f_{M}(0) &=& 1, \\
f_{M}'(0) &=& 2p'(0) + z'(0), \\
f_{M}''(0) &=& 2p'(0)^{2} + 2p''(0) + 2p'(0)z'(0) + 2 \theta z_{0}'(0) z'(0) + 2(\tfrac{1}{2} - \theta) z'(0)^{2}, \\
f_{M}'''(0) &=& 6p'(0)p''(0) + 2p'''(0) + 3p''(0)z'(0) + z'''(0).
\end{eqnarray*}
Concerning the Rannacher time stepping, define
$$ f_{N_{0}}(h) = 1 - \tfrac{1}{2} z(h), $$
such that
\begin{eqnarray*}
f_{N_{0}}(0) &=& 1, \\
f^{(i)}_{N_{0}}(0) &=& - \tfrac{1}{2} z^{(i)}(0) \qquad \mbox{for } \ i=1,2,3,  \\
f_{N_{0}}^{(4)}(h) &=& \mathcal{O}\left( 1+(\kappa^{2} + c^{2}\eta^{2} )^{2} \right).
\end{eqnarray*}
All of these expressions will be used in the forthcoming subsections, where we analyse the asymptotic behaviour of $\widehat{U}_{N}$ in five different regions of the Fourier domain, i.e.\ the $(\kappa,\eta)$-domain with $\vert \kappa \vert, \vert \eta \vert \leq \pi/h$.

\subsection{Region 1: $\mathbf{\vert \boldsymbol{\kappa} \vert, \vert c \boldsymbol{\eta} \vert \leq h^{-q}}$ with $\mathbf{q < 1/3}$}
\label{Region1}

In order to analyse $\log\widehat{U}_{N}$ in this region, the parts stemming from the MCS scheme and Rannacher time stepping will be considered separately.
Write (\ref{eq:LogTransformFourierTransformNumercialSolution1}) as
\begin{equation*}
\tfrac{1}{\lambda h} \left[ \log( p^{2} + pz + \theta z_{0}z + (\tfrac{1}{2}-\theta)z^{2}) - 2 \log(p) \right] = \tfrac{1}{\lambda h} \left[ \log(f_{M}(h)) - 2 \log(p(h))\right].
\end{equation*}
Using the analysis above it follows that
\begin{equation}
\label{eq:MCSTaylor}
\tfrac{1}{\lambda h} \left[ \log(f_{M}(h)) - 2 \log(p(h))\right] = s^{[0]} + s^{[1]}h + s^{[2]}h^{2} + s^{[3]}h^{3},
\end{equation}
where
\begin{eqnarray*}
\lambda s^{[0]} &=& 2p'(0) + z'(0) - 2p'(0), \\
\lambda s^{[1]} &=& \tfrac{1}{2}\left[ 2p'(0)^{2} + 2p''(0) + 2p'(0)z'(0) + 2\theta z_{0}'(0)z'(0) + 2(\tfrac{1}{2} - \theta )z'(0)^{2} - (2p'(0) + z'(0))^{2} \right] \\
				&& - \ \left[ p''(0) - p'(0)^{2}\right]  \\
\lambda s^{[2]} &=& \tfrac{1}{6} \left[ 6p'(0)p''(0) + 2p'''(0) + 3p''(0)z'(0) + z'''(0) \right. \\
				&& - \ 3(2p'(0) + z'(0))(2p'(0)^{2} + 2p''(0) + 2p'(0)z'(0) + 2\theta z_{0}'(0)z'(0) + 2(\tfrac{1}{2} - \theta)z'(0)^{2} ) \\
				&& \left. + \ 2 (2p'(0) + z'(0))^{3} \right] - \tfrac{1}{3} \left[ p'''(0) -3p'(0)p''(0) + 2p'(0)^{3} \right]  \\
s^{[3]} &=& \mathcal{O}\left( 1+(\kappa^{2} + c^{2}\eta^{2})^{4} \right).
\end{eqnarray*}
By using the expansions in Subsection \ref{subsec:Taylor} and after simplifying the resulting expressions, one gets
\begin{eqnarray*}
s^{[0]} &=& - \kappa^{2} - 2\rho\kappa\eta - \eta^{2} + \imi a_{1}\kappa + \imi a_{2} \eta, \\
s^{[1]} &=& 0, \\
s^{[2]} &=& \tfrac{1}{12} \kappa^{4} + \tfrac{1}{3} \rho (\kappa^{2} + c^{2}\eta^{2} ) \kappa \eta + \tfrac{1}{12} c^{2}\eta^{4} - \tfrac{1}{6} \imi a_{1} \kappa^{3} - \tfrac{1}{6} \imi a_{2} c^{2} \eta^{3}  \\
				&& - \ \lambda^{2}\theta^{2} (- \kappa^{2} + \imi a_{1}\kappa ) (- \eta^{2} + \imi a_{2} \eta) (-\kappa^{2} -2\rho \kappa \eta - \eta^{2} + \imi a_{1}\kappa + \imi a_{2} \eta) \\
				&& + \ \tfrac{\lambda^{2}}{12} (-\kappa^{2} -2\rho \kappa \eta - \eta^{2} + \imi a_{1}\kappa + \imi a_{2} \eta)^{3} \\
				&& - \ \lambda^{2} (-\kappa^{2} -2\rho \kappa \eta - \eta^{2} + \imi a_{1}\kappa + \imi a_{2} \eta)( - \rho \kappa \eta + (\tfrac{1}{2} - \theta) (-\kappa^{2} - \eta^{2} + \imi a_{1}\kappa + \imi a_{2} \eta) )^{2}.
\end{eqnarray*}
As for the part  stemming from the Rannacher time stepping, write (\ref{eq:LogTransformFourierTransformNumercialSolution2}) as
\begin{equation*}
N_{0}[2 \log(p) - \log(f_{M}(h)) - 2 \log(f_{N_{0}}(h)) ].
\end{equation*}
Using the same analysis as above one gets
\begin{equation}
\label{eq:RannacherTaylor}
N_{0}[2 \log(p) - \log(f_{M}(h)) - 2 \log(f_{N_{0}}(h)) ] = N_{0}^{[1]}h + N_{0}^{[2]}h^{2} + N_{0}^{[3]}h^{3},
\end{equation}
where
\begin{eqnarray*}
N_{0}^{[1]} &=& N_{0} \left[ - \lambda s^{[0]} - 2(-\tfrac{1}{2} z'(0)) \right] = 0, \\
N_{0}^{[2]} &=& N_{0} \left[ - \lambda s^{[1]} - (-\tfrac{1}{2} z''(0) -  \tfrac{1}{4} z'(0)^{2} ) \right] = \tfrac{1}{4} N_{0} z'(0)^{2} \\
		&=& \tfrac{1}{4} N_{0} \lambda^{2} (-\kappa^{2} -2\rho \kappa \eta - \eta^{2} + \imi a_{1}\kappa + \imi a_{2} \eta)^{2}, \\
N_{0}^{[3]} &=& \mathcal{O}\left( 1+(\kappa^{2} + c^{2}\eta^{2} )^{3} \right).
\end{eqnarray*}
By combining (\ref{eq:LogTransformFourierTransformNumercialSolution}), (\ref{eq:MCSTaylor}) and (\ref{eq:RannacherTaylor}) it directly follows that 
\begin{equation*}
\log \widehat{U}_{N} = -\kappa^{2} -2\rho \kappa \eta - \eta^{2} + \imi a_{1}\kappa + \imi a_{2} \eta + (s^{[2]} + N_{0}^{[2]})h^{2} + (s^{[3]} + N_{0}^{[3]})h^{3},
\end{equation*}
and hence
\begin{equation*}
\widehat{U}_{N} = \exp(-\kappa^{2} -2\rho \kappa \eta - \eta^{2} + \imi a_{1}\kappa + \imi a_{2} \eta)\exp((s^{[2]} + N_{0}^{[2]})h^{2} + (s^{[3]} + N_{0}^{[3]})h^{3}).
\end{equation*}
Next, we will expand the second exponential in order to compare this expression with the Fourier transform $\widehat{u}$ from (\ref{eq:FourierTransformExact}) at $t=1$.
Let
$$ e(h) = \exp( c^{[2]} h^{2} + c^{[3]} h^{3}), $$
where
$$ c^{[2]} = \mathcal{O}\left( 1+(\kappa^{2} + c^{2}\eta^{2} )^{3} \right), \ c^{[3]} = \mathcal{O}\left( 1+(\kappa^{2} + c^{2}\eta^{2} )^{4} \right), $$
then
\begin{eqnarray*}
e'(h) &=& (2c^{[2]} h + 3 c^{[3]} h^{2} ) e(h), \\
e''(h) &=& (2c^{[2]} + 6c^{[3]}h)e(h) + (2c^{[2]} h + 3 c^{[3]} h^{2} )^{2} e(h), \\
e'''(h) &=& 6c^{[3]} e(h) + 3(2c^{[2]} + 6c^{[3]} h)(2c^{[2]} h + 3 c^{[3]} h^{2} )e(h) + (2c^{[2]} h + 3 c^{[3]} h^{2} )^{3} e(h).
\end{eqnarray*}
Since $\vert \kappa \vert, \vert c \eta \vert \leq h^{-q}$ with $q < 1/3$, we have that $e(0)=1$ and $\vert e(h) \vert \leq \exp(1)$ whenever $h$ is sufficiently small. 
Hence it follows that
$$ e(h) = 1 + e^{[2]} h^{2} + e^{[3]} h^{3}, $$
with
\begin{eqnarray*}
e^{[2]} &=& c^{[2]}, \\
e^{[3]} &=& \mathcal{O}\left( 1+(\kappa^{2} + c^{2}\eta^{2})^{4} \right) + \mathcal{O}\left( 1+(\kappa^{2} + c^{2}\eta^{2})^{6} \right)h  + \mathcal{O}\left( 1+(\kappa^{2}+c^{2}\eta^{2})^{9}  \right)h^{3}  \\
		&=& \mathcal{O}\left( 1+(\kappa^{2} + c^{2}\eta^{2})^{4} \right) + \mathcal{O}\left( 1+(\kappa^{2} + c^{2}\eta^{2})^{6} \right)h,
\end{eqnarray*}
where the latter equality follows from the assumption $\vert \kappa \vert, \vert c \eta \vert \leq h^{-q}$ with $q < 1/3$.
Finally, for this region, one arrives at the following expression for the Fourier error \eqref{eq:FourierErrorDef}:
\begin{equation}
h^{2}\widehat{u}(\kappa,\eta,1)\left( (s^{[2]} + N_{0}^{[2]}) + \mathcal{O}\left( 1+(\kappa^{2} + c^{2}\eta^{2})^{4} \right)h + \mathcal{O}\left( 1+(\kappa^{2} + c^{2}\eta^{2})^{6} \right)h^{2}\right).
\label{eq:LowWavenumberTransform}
\end{equation}
Note that $s^{[2]}$ and $N_{0}^{[2]}$ actually depend on $\kappa$ and $\eta$. For ease of presentation, this is omitted in the notation.

\subsection{Region 2: $\mathbf{\vert \boldsymbol{\kappa} \vert \leq h^{-q_{1}}, \vert c \boldsymbol{\eta} \vert \leq h^{-q_{2}}}$ with $\mathbf{q_{1},q_{2} \leq 1/2}$ and with $\mathbf{q_{1} \geq 1/3}$ or $\mathbf{q_{2} \geq 1/3}$}

First consider the case where both $q_{1} < 1/2$ and $q_{2} < 1/2$.
Based on the analysis in Subsection \ref{Region1}, expression (\ref{eq:LogTransformFourierTransformNumercialSolution1}) can be rewritten as
$$ N\log(R) = \tfrac{1}{\lambda h} \left[ \log( p^{2} + pz + \theta z_{0}z + (\tfrac{1}{2}-\theta)z^{2}) - 2\log(p) \right] = s^{[0]} + s^{[2']}h^{2},  $$
where
\begin{eqnarray*}
s^{[0]} &=& -\kappa^{2} - 2\rho\kappa\eta - \eta^{2} + \imi a_{1}\kappa + \imi a_{2} \eta, \\
s^{[2']} &=& \mathcal{O}\left( (\kappa^{2} + c^{2} \eta^{2})^{3} \right).
\end{eqnarray*} 
Since either $\kappa$ or $\eta$ becomes large in this region as $h$ tends to zero, only the highest powers of $\kappa,\eta$ are taken into account in the order term in $s^{[2']}$.
From $\vert \rho \vert < 1 $ one gets
$$ \kappa^{2} + 2\rho \kappa \eta + \eta^{2} = (1 - \vert \rho \vert) (\kappa^{2} + \eta^{2} ) + \vert \rho \vert ( \kappa + \text{sgn}(\rho) \eta)^{2} > 0, $$
such that
$$  \mathcal{R}(s^{[0]}) \leq -(1- \vert \rho \vert )(\kappa^{2} + \eta^{2} ) < 0. $$
Using that both both $q_{1} < 1/2$ and $q_{2} < 1/2$, it directly follows that
$$ \lim_{h \rightarrow 0} (\kappa^{2} + c^{2}\eta^{2} )^{2}h^{2} = 0, $$
and thus
$$ \exists \ \delta > 0 \quad \exists \ h_{0} >0 \quad \forall \ h \leq h_{0}: \ \mathcal{R}( N\log(R) )  \leq - \delta (\kappa^{2} + \eta^{2}).  $$
Hence, for $h \leq h_{0}$ 
$$ \vert R^{N} \vert \leq \exp( - \delta (\kappa^{2} + \eta^{2}) ), $$
and since $\vert \kappa \vert \geq h^{-1/3}$ or $\vert c\eta \vert \geq h^{-1/3}$ we may conclude
\begin{equation}
 \vert R^{N} \vert = \mathcal{O}\left( h^{w} \right) \quad \forall w>0 .
\label{eq:Region2Deel1}
\end{equation}
Next, consider the case where at least one of the equalities, $q_{1}=1/2$ or $q_{2} = 1/2$, holds.
For analysing the asymptotic behaviour of $R$ we then make use of the following proposition. 
Its proof is a direct modification of the proof of one of the statements in \cite[Theorem 1]{IHM10} and is therefore omitted.

\begin{proposition}
\label{prop:Stabiliteit}
Let $\widetilde{z_{0}},\widetilde{z_{1}},\widetilde{z_{2}}$ denote real numbers with 
\begin{equation} 
\label{eq:ScalarRestrictions}
\widetilde{z_{1}} \leq 0, \qquad \widetilde{z_{2}} \leq 0, \qquad \vert \widetilde{z_{0}} \vert \leq 2 \vert \rho \vert \sqrt{\widetilde{z_{1}}\widetilde{z_{2}}}, 
\end{equation}
and $\vert \rho \vert < 1$. Set $ \widetilde{z} := \widetilde{z_{0}} + \widetilde{z_{1}} + \widetilde{z_{2}} $ and $\widetilde{p} := (1-\theta \widetilde{z_{1}})(1-\theta \widetilde{z_{2}})$. 
If $\widetilde{z_{1}}<0$ or $\widetilde{z_{2}}<0$, then
$$ \left\vert \frac{\widetilde{p}^{2} + \widetilde{p_{}}\widetilde{z} + \theta \widetilde{z_{0}}\widetilde{z} + (\tfrac{1}{2} - \theta)\widetilde{z}^{2}}{\widetilde{p}^{2}}  \right\vert < 1, $$
whenever $\theta \geq \tfrac{1}{4}$ and $\theta > \tfrac{\vert \rho \vert + 1}{6}.$
\end{proposition}
Recall that in the current region of the Fourier domain the assumption $\vert \kappa \vert \leq h^{-q_{1}}, \vert c \eta \vert \leq h^{-q_{2}}$ with $q_{1},q_{2} \leq 1/2$ holds. This yields
\begin{eqnarray*}
&& \lim_{h \rightarrow 0} z_{0}(h) = \lim_{h \rightarrow 0} - 2 \rho \lambda \kappa \eta h =: \widetilde{z_{0}} \in \R, \\
&& \lim_{h \rightarrow 0} z_{1}(h) = \lim_{h \rightarrow 0} - \lambda \kappa^{2} h =: \widetilde{z_{1}} \in \R^{-}, \\
&& \lim_{h \rightarrow 0} z_{2}(h) = \lim_{h \rightarrow 0} - \lambda \eta^{2} h =: \widetilde{z_{2}} \in \R^{-}.
\end{eqnarray*}
Since $\vert \kappa \vert = h^{-1/2}$ or $\vert c\eta \vert = h^{-1/2}$ it follows that $z_{1}^{*} < 0$ or $z_{2}^{*} <0 $.
Hence, all the assumptions on $\widetilde{z_{0}}, \widetilde{z_{1}}, \widetilde{z_{2}}$ in Proposition \ref{prop:Stabiliteit} are fulfilled such that
$$ \lim_{h \rightarrow 0} \vert R \vert  =  \left\vert \frac{\widetilde{p}^{2} + \widetilde{p_{}}\widetilde{z} + \theta \widetilde{z_{0}}\widetilde{z} + (\tfrac{1}{2} - \theta)\widetilde{z}^{2}}{\widetilde{p}^{2}}  \right\vert < 1, $$
and thus
\begin{equation}
\vert R^{N} \vert = \vert R \vert^{1/(\lambda h)} = \mathcal{O}(h^{w}) \quad \forall w>0, 
\label{eq:Region2Deel2}
\end{equation}
for
\begin{equation}
\theta \geq \tfrac{1}{4} \quad \mbox{and} \quad \theta > \tfrac{1+\vert \rho \vert}{6}.
\label{eq:ConditionsOnTheta}
\end{equation}
Further, it always holds that $\mathcal{R}(z) \leq 0$ such that 
$$ \vert 1 - \tfrac{1}{2}z \vert^{-1} \leq 1.$$
By combining this with \eqref{eq:Region2Deel1} and \eqref{eq:Region2Deel2}, and by using that $N_{0}$ is independent from $h$, one may conclude that in this region it holds that
$$\vert \widehat{U}_{N} \vert = \vert R^{N} \vert \vert R^{-N_{0}} \vert \vert 1 - \tfrac{1}{2}z \vert^{-2N_{0}} = \mathcal{O}\left( h^{w} \right) \quad \forall w >0,$$
under restriction \eqref{eq:ConditionsOnTheta} on $\theta$.
This means that $\vert \widehat{U}_{N} \vert$ quickly becomes negligible as $h$ tends to zero. It decays faster to zero than any polynomial in $h$.

\subsection{Region 3: $\mathbf{\vert \boldsymbol{\kappa} \vert, \vert c \boldsymbol{\eta} \vert \geq h^{-q}}$ with $\mathbf{q > 1/2}$}
\label{subsec:HighWaveTransf}

Here we reconsider the substitutions $\vartheta_{1} = \kappa h_{1} = \kappa h, \ \vartheta_{2} = \eta h_{2} = \eta c h$ in order to get
\begin{eqnarray*}
z_{0} &=& - 2 \rho \tfrac{\lambda}{ch} \sin \vartheta_{1} \sin \vartheta_{2}, \\
z_{1} &=& - 4 \tfrac{\lambda}{h} \sin^{2} \tfrac{\vartheta_{1}}{2} + \imi a_{1} \lambda \sin \vartheta_{1}, \\
z_{2} &=& - 4 \tfrac{\lambda}{c^{2}h} \sin^{2} \tfrac{\vartheta_{2}}{2} + \imi a_{2} \tfrac{\lambda}{c} \sin \vartheta_{2}.
\end{eqnarray*}
Further, in this region $\vartheta_{1},\vartheta_{2}$ are different from zero and since we consider values $-\pi \leq \vartheta_{1},\vartheta_{2} \leq \pi$, we may write
\begin{eqnarray*}
\frac{c^{2}}{16\tfrac{\lambda^{2}}{h^{2}}\sin^{2} \tfrac{\vartheta_{1}}{2}\sin^{2} \tfrac{\vartheta_{2}}{2}}z_{0} &=& - \rho \frac{c \cot \tfrac{\vartheta_{1}}{2}\cot \tfrac{\vartheta_{2}}{2}}{2 \lambda}h, \\
\frac{c^{2}}{16\tfrac{\lambda^{2}}{h^{2}}\sin^{2} \tfrac{\vartheta_{1}}{2}\sin^{2} \tfrac{\vartheta_{2}}{2}}z_{1} &=& - \frac{c^{2}}{4\lambda \sin^{2} \tfrac{\vartheta_{2}}{2}}h + \imi a_{1} \frac{c^{2}\cot\tfrac{\vartheta_{1}}{2}}{8\lambda \sin^{2} \tfrac{\vartheta_{2}}{2}}h^{2}, \\
\frac{c^{2}}{16\tfrac{\lambda^{2}}{h^{2}}\sin^{2} \tfrac{\vartheta_{1}}{2}\sin^{2} \tfrac{\vartheta_{2}}{2}}z_{2} &=& - \frac{1}{4\lambda \sin^{2} \tfrac{\vartheta_{1}}{2}}h + \imi a_{2} \frac{c \cot\tfrac{\vartheta_{2}}{2}}{8\lambda \sin^{2} \tfrac{\vartheta_{1}}{2}}h^{2}, \\
\frac{1}{4\tfrac{\lambda}{h} \sin^{2} \tfrac{\vartheta_{1}}{2}}(1-\theta z_{1}) &=& \theta + \left( \frac{1}{4\lambda \sin^{2} \tfrac{\vartheta_{1}}{2}} - \theta \imi a_{1} \frac{ \cot \tfrac{\vartheta_{1}}{2}}{2}\right)h, \\ 
\frac{c^{2}}{4\tfrac{\lambda}{h} \sin^{2} \tfrac{\vartheta_{2}}{2}}(1-\theta z_{2}) &=& \theta + \left( \frac{c^{2}}{4\lambda \sin^{2} \tfrac{\vartheta_{2}}{2}} - \theta \imi a_{2} \frac{ c \cot \tfrac{\vartheta_{2}}{2}}{2}\right) h,
\end{eqnarray*}
and thus
\begin{eqnarray*}
\frac{c^{2}}{16\tfrac{\lambda^{2}}{h^{2}}\sin^{2} \tfrac{\vartheta_{1}}{2}\sin^{2} \tfrac{\vartheta_{2}}{2}} p &=& \theta^{2} +  \theta \left( \frac{1}{4\lambda \sin^{2} \tfrac{\vartheta_{1}}{2}} - \theta \imi a_{1} \frac{ \cot \tfrac{\vartheta_{1}}{2}}{2} + \frac{c^{2}}{4\lambda \sin^{2} \tfrac{\vartheta_{2}}{2}} - \theta \imi a_{2} \frac{ c \cot \tfrac{\vartheta_{2}}{2}}{2} \right) h \\
		&& + \ \left( \frac{1}{4\lambda \sin^{2} \tfrac{\vartheta_{1}}{2}} - \theta \imi a_{1} \frac{ \cot \tfrac{\vartheta_{1}}{2}}{2}  \right) \left( \frac{c^{2}}{4\lambda \sin^{2} \tfrac{\vartheta_{2}}{2}} - \theta \imi a_{2} \frac{ c \cot \tfrac{\vartheta_{2}}{2}}{2}  \right)h^{2}.
\end{eqnarray*}
Making use of an expansion similar to (\ref{eq:TaylorExpansionLogFunction}) it follows that
\begin{eqnarray*}
&& \log \left[ \left(  \frac{c^{2}}{16\tfrac{\lambda^{2}}{h^{2}}\sin^{2} \tfrac{\vartheta_{1}}{2}\sin^{2} \tfrac{\vartheta_{2}}{2}}  \right)^{2} \left(  p^{2} + pz + \theta z_{0}z + (\tfrac{1}{2} - \theta)z^{2}   \right) \right] \\
&=& \log \theta^{4} \\
&& + \ \left[ 2 \theta^{3} \left( \frac{1}{4\lambda \sin^{2} \tfrac{\vartheta_{1}}{2}} - \theta \imi a_{1} \frac{ \cot \tfrac{\vartheta_{1}}{2}}{2} + \frac{c^{2}}{4\lambda \sin^{2} \tfrac{\vartheta_{2}}{2}} - \theta \imi a_{2} \frac{ c \cot \tfrac{\vartheta_{2}}{2}}{2}  \right) \right. \\
&& - \ \left. \theta^{2} \rho \frac{c \cot \tfrac{\vartheta_{1}}{2}\cot \tfrac{\vartheta_{2}}{2}}{2\lambda} - \theta^{2} \frac{1}{4\lambda \sin^{2}\tfrac{\vartheta_{1}}{2}}  - \theta^{2} \frac{c^{2}}{4\lambda \sin^{2}\tfrac{\vartheta_{2}}{2}} \right] \frac{h}{\theta^{4}} \\
&& + \ \mathcal{O}\left( \left(  \frac{1}{\vert \sin \tfrac{\vartheta_{1}}{2} \vert} + \frac{c}{\vert \sin \tfrac{\vartheta_{2}}{2} \vert}  \right)^{4}h^{2} \right),
\end{eqnarray*}
and
\begin{eqnarray*}
\log \left[  \frac{c^{2}}{16\tfrac{\lambda^{2}}{h^{2}}\sin^{2} \tfrac{\vartheta_{1}}{2}\sin^{2} \tfrac{\vartheta_{2}}{2}} p \right] &=& \log \theta^{2} \\
&& + \ \left( \frac{1}{4\lambda \sin^{2} \tfrac{\vartheta_{1}}{2}} - \theta \imi a_{1} \frac{ \cot \tfrac{\vartheta_{1}}{2}}{2} + \frac{c^{2}}{4\lambda \sin^{2} \tfrac{\vartheta_{2}}{2}} - \theta \imi a_{2} \frac{ c \cot \tfrac{\vartheta_{2}}{2}}{2}  \right) \frac{h}{\theta} \\
&& + \ \mathcal{O}\left(  \left( \frac{1}{\vert \sin \tfrac{\vartheta_{1}}{2} \vert} + \frac{c}{\vert \sin \tfrac{\vartheta_{2}}{2} \vert}  \right)^{4}h^{2} \right).
\end{eqnarray*}
Combining both expressions yields 
\begin{eqnarray*}
\log \left( R  \right) &=& - \frac{1}{4\lambda \theta^{2}} \left( 2 \rho c \cot \tfrac{\vartheta_{1}}{2} \cot \tfrac{\vartheta_{2}}{2} + \frac{1}{\sin^{2} \tfrac{\vartheta_{1}}{2}} + \frac{c^{2}}{\sin^{2} \tfrac{\vartheta_{2}}{2}}  \right) h \\
 && + \ \mathcal{O}\left(  \left( \frac{1}{\vert \sin \tfrac{\vartheta_{1}}{2} \vert} + \frac{c}{\vert \sin \tfrac{\vartheta_{2}}{2} \vert}  \right)^{4}h^{2} \right) \\
 &=&  - \frac{1}{4\lambda \theta^{2}} \frac{c^{2}\sin^{2} \tfrac{\vartheta_{1}}{2} + 2\rho c \cos \tfrac{\vartheta_{1}}{2} \sin \tfrac{\vartheta_{1}}{2} \cos \tfrac{\vartheta_{2}}{2} \sin \tfrac{\vartheta_{2}}{2} + \sin^{2} \tfrac{\vartheta_{2}}{2} }{\sin^{2}\tfrac{\vartheta_{1}}{2} \sin^{2} \tfrac{\vartheta_{2}}{2}} h \\
 && + \ \mathcal{O}\left(  \left( \frac{1}{\vert \sin \tfrac{\vartheta_{1}}{2} \vert} + \frac{c}{\vert \sin \tfrac{\vartheta_{2}}{2} \vert}  \right)^{4}h^{2} \right).
\end{eqnarray*}
Further, recall that in this region $\vartheta_{1},\vartheta_{2}$ are both different from zero such that
\begin{equation} \iota(\vartheta_{1},\vartheta_{2}) := c^{2} \sin^{2} \tfrac{\vartheta_{1}}{2} + 2\rho c \cos \tfrac{\vartheta_{1}}{2} \sin \tfrac{\vartheta_{1}}{2} \cos \tfrac{\vartheta_{2}}{2} \sin \tfrac{\vartheta_{2}}{2} + \sin^{2} \tfrac{\vartheta_{2}}{2} > 0,
\label{eq:PositivityDenominator}
\end{equation}
and hence
\begin{eqnarray*}
\log R^{N} &=& N\log R \\
&=& - \frac{1}{4\lambda^{2} \theta^{2}} \frac{\iota(\vartheta_{1},\vartheta_{2}) }{\sin^{2}\tfrac{\vartheta_{1}}{2} \sin^{2} \tfrac{\vartheta_{2}}{2}}       \left( 1 + \mathcal{O}\left(  \left( \frac{1}{\vert \sin \tfrac{\vartheta_{1}}{2} \vert} + \frac{c}{\vert \sin \tfrac{\vartheta_{2}}{2} \vert}  \right)^{2}h \right) \right).
\end{eqnarray*}
As for the implicit Euler time stepping scheme we note
\begin{equation*}
\tfrac{c^{2}h}{2\lambda}\left( 1 - \tfrac{1}{2}z \right) = c^{2}\sin^{2}\tfrac{\vartheta_{1}}{2} + \tfrac{1}{2}\rho c \sin \vartheta_{1}  \sin \vartheta_{2} + \sin^{2}\tfrac{\vartheta_{2}}{2} + \left(  \tfrac{c^{2}}{2\lambda} - i a_{1} \tfrac{c^{2}}{4} \sin \vartheta_{1} - i a_{2} \tfrac{c}{4}  \sin \vartheta_{2}  \right)h.
\end{equation*}
Using once again an expansion analogous to (\ref{eq:TaylorExpansionLogFunction}) it follows that
$$ \log \left( 1 - \tfrac{1}{2}z \right) = \log\left( \tfrac{2\lambda}{c^{2}h} \right) + \log\left( \iota(\vartheta_{1},\vartheta_{2}) \right) + \ \mathcal{O}\left(  \left( \frac{1}{\vert \sin \tfrac{\vartheta_{1}}{2} \vert} + \frac{c}{\vert \sin \tfrac{\vartheta_{2}}{2} \vert}  \right)^{2}h \right), $$
which yields
\begin{equation*}
\log\left( R^{-N_{0}} (1-\tfrac{1}{2}z)^{-2N_{0}} \right) = -2N_{0} \log\left( \tfrac{2\lambda}{c^{2}h} \right) -2N_{0} \log\left( \iota(\vartheta_{1},\vartheta_{2}) \right) + \mathcal{O}\left(  \left( \frac{1}{\vert \sin \tfrac{\vartheta_{1}}{2} \vert} + \frac{c}{\vert \sin \tfrac{\vartheta_{2}}{2} \vert}  \right)^{2}h \right).
\end{equation*}
Making use of relationship (\ref{eq:FourierTransformNumericalSolution}) one becomes an expression for the logarithm of the Fourier transform $\widehat{U}_{N}$:
\begin{eqnarray*}
\log \widehat{U}_{N} &=& - \frac{1}{4\lambda^{2} \theta^{2}} \frac{\iota(\vartheta_{1},\vartheta_{2}) }{\sin^{2}\tfrac{\vartheta_{1}}{2} \sin^{2} \tfrac{\vartheta_{2}}{2}}       \left( 1 + \mathcal{O}\left(  \left( \frac{1}{\vert \sin \tfrac{\vartheta_{1}}{2} \vert} + \frac{c}{\vert \sin \tfrac{\vartheta_{2}}{2} \vert}  \right)^{2}h \right) \right) \\
&& - \ 2N_{0} \log\left( \tfrac{2\lambda}{c^{2}h} \right) -2N_{0} \log\left( \iota(\vartheta_{1},\vartheta_{2}) \right) +  \mathcal{O}\left(  \left( \frac{1}{\vert \sin \tfrac{\vartheta_{1}}{2} \vert} + \frac{c}{\vert \sin \tfrac{\vartheta_{2}}{2} \vert}  \right)^{2}h \right),
\end{eqnarray*}
such that in this region
\begin{equation}
\widehat{U}_{N} = \frac{(c^{2}h)^{2N_{0}}}{\left[ 2 \lambda \iota(\vartheta_{1},\vartheta_{2}) \right]^{2N_{0}}} \exp\left(  - \frac{1}{4\lambda^{2} \theta^{2}} \frac{\iota(\vartheta_{1},\vartheta_{2}) }{\sin^{2}\tfrac{\vartheta_{1}}{2} \sin^{2} \tfrac{\vartheta_{2}}{2}} \right) \left( 1 + \mathcal{O}\left(  \frac{h}{\left( \vert \vartheta_{1} \vert + \vert \vartheta_{2} \vert \right)^{2}} \right) \right).
\label{eq:HighWavenumberTransform}
\end{equation}
In Figure \ref{fig:NormTransform} we noticed that in the high-wavenumber region, i.e.\ where both $\vert \vartheta_{1} \vert, \vert \vartheta_{2} \vert$ are large, the norm $\vert \widehat{U}_{N} \vert$ is highly dependent on the MCS parameter $\theta$.
This is confirmed by (\ref{eq:HighWavenumberTransform}) since inequality (\ref{eq:PositivityDenominator}) holds in the high-wavenumber region. Hence, \textit{for larger values of the MCS parameter $\theta$ one can expect a larger high-wavenumber error}.

\subsection{Region 4: $\mathbf{\vert \boldsymbol{\kappa} \vert \geq h^{-q_{1}}, \vert c \boldsymbol{\eta} \vert \leq h^{-q_{2}}}$ with $\mathbf{q_{1} > 1/2, q_{2} \leq 1/2}$}
\label{Region4}

Reconsider the substitution $\vartheta_{1} = \kappa h$ and recall that $-\pi \leq \vartheta_{1} \leq \pi$. Then, as $\vartheta_{1}$ is non-zero in this region, one may write
\begin{eqnarray*}
\frac{1}{4\tfrac{\lambda}{h}\sin^{2} \tfrac{\vartheta_{1}}{2}} z_{1} &=& -1 + \tfrac{1}{2} \imi a_{1} h \cot \tfrac{\vartheta_{1}}{2}, \\
\frac{1}{4\tfrac{\lambda}{h}\sin^{2} \tfrac{\vartheta_{1}}{2}} z_{2} &=& -\frac{1}{c^{2}\sin^{2} \tfrac{\vartheta_{1}}{2}} \sin^{2} \tfrac{c\eta h}{2} + \imi a_{2} \frac{h}{4c \sin^{2} \tfrac{\vartheta_{1}}{2}} \sin c\eta h, \\
\frac{1}{4\tfrac{\lambda}{h}\sin^{2} \tfrac{\vartheta_{1}}{2}} z_{0} &=& - \frac{\rho}{c} \cot \tfrac{\vartheta_{1}}{2} \sin c\eta h, \\
\frac{1}{4\tfrac{\lambda}{h}\sin^{2} \tfrac{\vartheta_{1}}{2}} p &=& \left( \theta + \left( \frac{1}{4\lambda \sin^{2} \tfrac{\vartheta_{1}}{2}} - \tfrac{1}{2} \theta \imi a_{1} \cot \tfrac{\vartheta_{1}}{2}  \right)h \right) \left( 1 - \theta z_{2}  \right),
\end{eqnarray*}
such that
\begin{eqnarray*}
\lim_{h\rightarrow 0} \frac{1}{4\tfrac{\lambda}{h}\sin^{2} \tfrac{\vartheta_{1}}{2}} z_{1} &=& -1, \\
\lim_{h\rightarrow 0} \frac{1}{4\tfrac{\lambda}{h}\sin^{2} \tfrac{\vartheta_{1}}{2}} z_{2} &=& 0, \\
\lim_{h\rightarrow 0} \frac{1}{4\tfrac{\lambda}{h}\sin^{2} \tfrac{\vartheta_{1}}{2}} z_{0} &=& 0, \\
\lim_{h\rightarrow 0} \frac{1}{4\tfrac{\lambda}{h}\sin^{2} \tfrac{\vartheta_{1}}{2}} p &=& \theta ( 1 + \widetilde{\widetilde{z_{2}}}),
\end{eqnarray*}
where $\widetilde{\widetilde{z_{2}}}$ denotes a positive real number. Hence, concerning $R$ it follows that
$$ \lim_{h \rightarrow 0} \left\vert \frac{p^{2} + pz + \theta z_{0}z + (\tfrac{1}{2} - \theta)z^{2}}{p^{2}}  \right\vert  = \left\vert  \frac{(1+\widetilde{\widetilde{z_{2}}})^{2} \theta^{2} - (2+\widetilde{\widetilde{z_{2}}})\theta + \tfrac{1}{2}}{(1+\widetilde{\widetilde{z_{2}}})^{2} \theta^{2}} \right\vert. $$
For the latter expression we obtain the following positive result.
\begin{proposition}
\label{prop:Regio5}
If $\theta > 1/4$ and $\theta \neq 1/2$, then
$$  \left\vert  \frac{(1+\widetilde{\widetilde{z_{2}}})^{2} \theta^{2} - (2+\widetilde{\widetilde{z_{2}}})\theta + \tfrac{1}{2}}{(1+\widetilde{\widetilde{z_{2}}})^{2} \theta^{2}} \right\vert < 1 $$
for all real numbers $\widetilde{\widetilde{z_{2}}} \geq 0$. 
If $\theta = 1/4$ or $\theta = 1/2$, then the inequality holds for numbers $\widetilde{\widetilde{z_{2}}}>0$.
\end{proposition} 
\begin{proof}
Let $\widetilde{\widetilde{z_{2}}}$ be a positive real number.
First, it is clear that the inequality holds whenever both
\begin{subeqnarray}
&& -(2+\widetilde{\widetilde{z_{2}}})\theta + \tfrac{1}{2} < 0, \slabel{eq:Proposition2ineq1} \\
&&  2( 1 + \widetilde{\widetilde{z_{2}}} )^{2} \theta^{2} - (2 + \widetilde{\widetilde{z_{2}}})\theta + \tfrac{1}{2} > 0. \slabel{eq:Proposition2ineq2}
\end{subeqnarray}
It is readily seen that (\ref{eq:Proposition2ineq1}) is satisfied for $\theta > 1/4$. For strictly positive $\widetilde{\widetilde{z_{2}}}$ the inequality is satisfied whenever $\theta \geq 1/4$. 
Regarding inequality (\ref{eq:Proposition2ineq2}) we consider the left-hand side as a second-order polynomial in $\theta$ with discriminant
$$ \Delta = (2+ \widetilde{\widetilde{z_{2}}})^{2} - 4 (1 + \widetilde{\widetilde{z_{2}}})^{2} = -\widetilde{\widetilde{z_{2}}}(3\widetilde{\widetilde{z_{2}}}+4). $$
If $\widetilde{\widetilde{z_{2}}} >0$, then $\Delta < 0$ and the polynomial is strictly positive for all real numbers $\theta$. 
If $\widetilde{\widetilde{z_{2}}} = 0$, the polynomial reduces to $2\theta^{2} - 2\theta + 1/2 $ which reaches its minimum (zero) in $\theta = 1/2$. 
\begin{flushright}
\mbox{\tiny $\blacksquare$}
\end{flushright}
\end{proof}
Let $\theta > 1/4$ and $\theta \neq 1/2$. 
Then, applying Proposition \ref{prop:Regio5} in this region yields
$$ \lim_{h \rightarrow 0} \vert R \vert = \lim_{h \rightarrow 0} \left\vert \frac{p^{2} + pz + \theta z_{0}z + (\tfrac{1}{2} - \theta)z^{2}}{p^{2}}  \right\vert  < 1,$$
and thus
$$ \vert R^{N} \vert = \vert R \vert^{1/(\lambda h)} = \mathcal{O}\left( h^{w} \right) \quad \forall w>0. $$
Since $N_{0}$ is independent from $h$ and $\vert 1/(1-\tfrac{1}{2}z) \vert \leq 1$ one may conclude that 
$$ \vert \widehat{U}_{N} \vert = \mathcal{O}(h^{w}) \quad \forall w >0. $$
Next, consider the case $\theta = 1/2$. Recall that the MCS scheme then reduces to the original CS scheme.
If $\vert c\eta \vert = h^{-1/2}$, it follows that 
$$\lim_{h\rightarrow 0} \frac{1}{4\tfrac{\lambda}{h}\sin^{2} \tfrac{\vartheta_{1}}{2}} p = \theta ( 1 + \widetilde{\widetilde{z_{2}}}),$$
with $\widetilde{\widetilde{z_{2}}} >0 $ such that proposition \ref{prop:Regio5} can be applied and $ \vert R^{N} \vert = \mathcal{O}\left( h^{w} \right)$ for all $w>0 $. 
Now, assume $\vert c \eta \vert \leq h^{-q_{2}}$ with $q_{2} < 1/2$.
An expansion similar to (\ref{eq:TaylorExpansionLogFunction}) yields
\begin{eqnarray*}
&& \log \left[ \frac{-1}{16 \tfrac{\lambda^2}{h^2} \sin^{4} \tfrac{\vartheta_{1}}{2}} (p^{2} + pz + \theta z_{0} z + (\tfrac{1}{2} - \theta)z^{2}) \right] \\
&=& \log( -\theta^2 + 2\theta - \tfrac{1}{2} ) \\
&& - \ \left[ 2\theta \left( \frac{1}{4\lambda \sin^{2}\tfrac{\vartheta_{1}}{2}} - \tfrac{1}{2} \theta \imi a_{1} \cot \tfrac{\vartheta_{1}}{2} - \theta^{2} z_{2}^{[1]} \right) \right. +   \theta \left( -\rho \cot \tfrac{\vartheta_{1}}{2} \eta  + \tfrac{1}{2} \imi a_{1} \cot \tfrac{\vartheta_{1}}{2} \right) \\
&& - \ \left( \frac{1}{4\lambda \sin^{2} \tfrac{\vartheta_{1}}{2} } - \tfrac{1}{2} \theta \imi a_{1} \cot \tfrac{\vartheta_{1}}{2} - \theta^{2} z_{2}^{[1]}  \right) +  \theta \rho \cot \tfrac{\vartheta_{1}}{2} \eta \\
&&  - \ 2(\tfrac{1}{2} - \theta) \left( - \rho \cot \tfrac{\vartheta_{1}}{2} \eta + \tfrac{1}{2} \imi a_{1} \cot \tfrac{\vartheta_{1}}{2} \right) \Bigg] \frac{h}{-\theta^{2} + 2\theta - \tfrac{1}{2}} \\
&& + \ \mathcal{O}\left( \left( \frac{1}{\sin^{2} \tfrac{\vartheta_{1}}{2}} + \frac{\vert \eta \vert}{\vert \sin \tfrac{\vartheta_{1}}{2} \vert} + \eta^{2} \right)^{2} h^{2}  \right),
\end{eqnarray*}
and
\begin{eqnarray*}
\log \left[ \frac{1}{16 \tfrac{\lambda^2}{h^2} \sin^{4} \tfrac{\vartheta_{1}}{2}} p^{2} \right] &=& \log \theta^{2} \\
&& + \ 2\theta \left(  \frac{1}{4\lambda \sin^{2}\tfrac{\vartheta_{1}}{2}} - \tfrac{1}{2} \theta \imi a_{1} \cot \tfrac{\vartheta_{1}}{2} - \theta^{2} z_{2}^{[1]} \right) \frac{h}{\theta^{2}} \\
&& + \ \mathcal{O}\left( \left( \frac{1}{\sin^{2} \tfrac{\vartheta_{1}}{2}} + \eta^{2} \right)^{2} h^{2}  \right).
\end{eqnarray*}
Making use of $\theta = 1/2$ it follows that
\begin{eqnarray*}
\log ( -R ) &=& \left( \frac{-1}{\lambda \sin^{2}\tfrac{\vartheta_{1}}{2}} + z_{2}^{[1]} \right) h + \mathcal{O}\left( \left( \frac{1}{\sin^{2} \tfrac{\vartheta_{1}}{2}} + \frac{\eta}{\vert \sin \tfrac{\vartheta_{1}}{2} \vert} + \eta^{2} \right)^{2} h^{2}  \right) \\
&=& \left( \frac{-1}{\lambda \sin^{2}\tfrac{\vartheta_{1}}{2}} - \lambda \eta^{2} + \imi a_{2} \lambda \eta \right) h + \mathcal{O}\left( \left( \frac{1}{\sin^{2} \tfrac{\vartheta_{1}}{2}} + \frac{\vert \eta \vert}{\vert \sin \tfrac{\vartheta_{1}}{2} \vert} + \eta^{2} \right)^{2} h^{2}  \right), \\
\end{eqnarray*}
and hence
\begin{equation*}
\log \left( (-R)^{N} \right) = \left( \frac{-1}{\lambda^{2} \sin^{2}\tfrac{\vartheta_{1}}{2}} - \eta^{2} + \imi a_{2} \eta \right)\left( 1 + \mathcal{O}\left( \left( \frac{1}{\sin^{2} \tfrac{\vartheta_{1}}{2}} + \frac{\vert \eta \vert}{\vert \sin \tfrac{\vartheta_{1}}{2} \vert} + \eta^{2} \right) h  \right) \right).
\end{equation*}
In order to analyse the Rannacher time stepping we note
$$ \frac{1 - \tfrac{1}{2}z}{2\tfrac{\lambda}{h}\sin^{2} \tfrac{\vartheta_{1}}{2}} = 1 + \mathcal{O}\left( \left( \frac{1}{\sin^{2} \tfrac{\vartheta_{1}}{2}} + \frac{\vert \eta \vert}{\vert \sin \tfrac{\vartheta_{1}}{2} \vert} + \eta^{2} \right) h  \right), $$
which yields
\begin{equation*}
\log \left( (-R)^{-N_{0}} \left( 1-\tfrac{1}{2}z \right)^{-2N_{0}} \right) = 2N_{0} \log \left( \frac{h}{2\lambda \sin^{2} \tfrac{\vartheta_{1}}{2}} \right) + \mathcal{O}\left( \left( \frac{1}{\sin^{2} \tfrac{\vartheta_{1}}{2}} + \frac{\vert \eta \vert}{\vert \sin \tfrac{\vartheta_{1}}{2} \vert} + \eta^{2} \right) h \right).
\end{equation*}
By exploring relationship (\ref{eq:FourierTransformNumericalSolution}) one may conclude that
\begin{equation}
\widehat{U}_{N} = (-1)^{N-N_{0}} \frac{h^{2N_{0}}}{(2\lambda \sin^{2}\tfrac{\vartheta_{1}}{2})^{2N_{0}}} \exp\left( \frac{-1}{\lambda^{2} \sin^{2}\tfrac{\vartheta_{1}}{2}} - \eta^{2} + \imi a_{2} \eta \right)\left( 1 + \mathcal{O}\left( \left( \tfrac{1}{\vartheta_{1}^{2}} + \tfrac{\vert \eta \vert}{\vert \vartheta_{1} \vert} + \eta^{2} \right) h  \right) \right).
\label{eq:CSWavenumberError}
\end{equation}
Whenever $\vert c\eta \vert = h^{-1/2}$, the right-hand side of (\ref{eq:CSWavenumberError}) is $\mathcal{O}\left( h^{w} \right)$ for all $w>0$ such that we can use expression (\ref{eq:CSWavenumberError}) for the whole region in the case of $\theta = 1/2$.

\subsection{Region 5: $\mathbf{\vert \boldsymbol{\kappa} \vert \leq h^{-q_{1}}, \vert c \boldsymbol{\eta} \vert \geq h^{-q_{2}}}$ with $\mathbf{q_{1} \leq 1/2, q_{2} > 1/2}$}

The analysis for this region is completely analogous to the analysis in Subsection \ref{Region4}. Hence, for $\theta > 1/4$ and $ \theta \neq 1/2$ it follows that
$$ \vert \widehat{U}_{N} \vert = \mathcal{O}(h^{w}) \quad \forall w >0. $$
Whenever the CS scheme is considered, i.e.\ $\theta = 1/2$, one gets the expression
\begin{equation}
\widehat{U}_{N} = (-1)^{N-N_{0}} \frac{(ch)^{2N_{0}}}{(2\lambda \sin^{2}\tfrac{\vartheta_{2}}{2})^{2N_{0}}} \exp\left( \frac{-c^{2}}{\lambda^{2} \sin^{2}\tfrac{\vartheta_{2}}{2}} - \kappa^{2} + \imi a_{2} \kappa \right)\left( 1 + \mathcal{O}\left( \left( \kappa^{2} + \tfrac{\vert \kappa \vert}{\vert \vartheta_{2} \vert}  + \tfrac{1}{\vartheta_{2}^{2}}  \right) h  \right) \right).
\label{eq:CSWavenumberError2}
\end{equation}

\subsection{Connection with stability of the MCS scheme}
\label{subsec:Stability}

In the above analysis natural bounds on the MCS parameter $\theta$ arise under which the asymptotic results are valid. 
These bounds can be interpreted as stability bounds.
In particular, the conditions $\theta \geq \tfrac{1}{4}, \ \theta > \tfrac{ 1 + \vert \rho \vert}{6}$ are needed to ensure that the Fourier transform $\widehat{U}_{N}$ is negligible in the second region. This restriction is only slightly stronger than the lower bound on $\theta$ derived in \cite{IHM10}, guaranteeing unconditional stability of the MCS scheme in the von Neumann sense pertinent to two-dimensional diffusion equations with mixed derivative term.
This is, indeed, not very surprising. In \cite{IHM10} it is stated that the stability analysis of the MCS scheme in this case reduces to bounding by one of the modulus of the scalar expression
$$ 1 + \frac{\widetilde{z}}{\widetilde{p}} + \frac{(\theta \widetilde{z}_{0} + (\tfrac{1}{2} - \theta) \widetilde{z}) \widetilde{z}}{\widetilde{p}^{2}} ,$$
where $\widetilde{z}=\widetilde{z}_{0} + \widetilde{z_{1}} + \widetilde{z_{2}}, \ \widetilde{p} = (1-\theta \widetilde{z_{1}})(1-\theta \widetilde{z_{2}})$ and $\widetilde{z_{0}},\widetilde{z_{1}},\widetilde{z_{2}}$ denote real numbers satisfying the condition (\ref{eq:ScalarRestrictions}).
This explains why Proposition \ref{prop:Stabiliteit} is just a slight modification of one of the statements in \cite[Theorem 1]{IHM10}.

\setcounter{equation}{0}
\section{Asymptotic analysis in physical space}\label{results2}

In this section we will use the asymptotic results in Fourier space from Section \ref{results} to perform an error analysis in physical space. First note that the Fourier transform $\widehat{u}$ is only sizeable in region 1 of the Fourier domain. 
In the other regions it holds that $\kappa \geq h^{-1/3}$ or $c\eta \geq h^{-1/3}$ and hence 
$$ \widehat{u}(\kappa,\eta,1) = \mathcal{O}\left( h^{w} \right) \quad \forall w>0. $$
Based on equalities \eqref{eq:FourierTransformExact}, (\ref{eq:LowWavenumberTransform}) and (\ref{eq:HighWavenumberTransform}) we define
$$ \widehat{E}^{low} = h^{2} \exp(-\kappa^{2} - 2\rho \kappa \eta - \eta^{2} + i a_{1} \kappa + i a_{2} \eta ) (s^{[2]}(\kappa,\eta) + N_{0}^{[2]}(\kappa,\eta)) $$
and
$$ \widehat{E}^{high} = \frac{(c^{2}h)^{2N_{0}}}{\left[ 2 \lambda  \iota(\vartheta_{1},\vartheta_{2})  \right]^{2N_{0}}} \exp\left(  - \frac{1}{4\lambda^{2} \theta^{2}} \frac{\iota(\vartheta_{1},\vartheta_{2}) }{\sin^{2}\tfrac{\vartheta_{1}}{2} \sin^{2} \tfrac{\vartheta_{2}}{2}} \right).$$
Recall that $\vartheta_{1} = \kappa h_{1},$ $\vartheta_{2} = \eta h_{2}$ and $h_{2} = ch_{1} = ch$.
As a consequence, $\widehat{E}^{low}$ is only sizeable in region 1 and $\widehat{E}^{high}$ is only sizeable in region 3.
In the other regions of the Fourier domain $\widehat{U}_{N}$ is negligible whenever $\theta > \max\{\tfrac{1}{4},\tfrac{1+\vert \rho \vert}{6}\}$ and $ \theta \neq 1/2$.
Hence, for these values of $\theta$, the results can be combined to
\begin{equation}
 \widehat{U}_{N}(\kappa h_{1},\eta h_{2}) - \widehat{u}(\kappa,\eta,1) \approx \widehat{E}^{low} + \widehat{E}^{high}, \qquad \vert \kappa \vert, \vert c\eta \vert \leq \pi/h.  
\label{eq:FourierError}
\end{equation}
When $\theta = 1/2$, i.e.\ when the MCS scheme reduces to the CS scheme, $\widehat{U}_{N}$ is also sizeable in region 4 and region 5 of the Fourier domain. This case will be treated separately.

\subsection{MCS scheme with $\theta \neq 1/2$}
\label{ConvergenceMCS}

Consider the case where the MCS scheme is different from the CS scheme, i.e.\ $\theta \neq 1/2$, and suppose that the restriction $\theta > \max\{\tfrac{1}{4},\tfrac{1+\vert \rho \vert}{6}\}$ is satisfied. 
Approximation (\ref{eq:FourierError}) is then valid and based on \eqref{eq:ApproxTotalError} we have for the total error:
$$ U_{N,j,k} - u(x_{j},y_{k},1) \approx E^{low}_{j,k} + E^{high}_{j,k},  $$
where
\begin{equation}
E^{low}_{j,k} = \frac{h^{2}}{4\pi^{2}} \int_{-\infty}^{\infty}\int_{-\infty}^{\infty} \widehat{u}(\kappa,\eta,1) (s^{[2]}(\kappa,\eta) + N_{0}^{[2]}(\kappa,\eta)) \exp(\imi \kappa x_{j}) \exp( \imi \eta y_{k}) d\kappa d\eta 
\label{eq:LowWavenumberError}
\end{equation} 
and
\begin{eqnarray*}
E^{high}_{j,k} &=& \frac{h^{2N_{0}}c^{4N_{0}}}{4\pi^{2}} \int_{-\pi/h_{2}}^{\pi/h_{2}} \int_{-\pi/h_{1}}^{\pi/h_{1}} \frac{\exp(\imi \kappa x_{j})\exp(\imi \eta y_{k})}{\left[ 2 \lambda \iota(\vartheta_{1},\vartheta_{2}) \right]^{2N_{0}}} \exp\left(  - \frac{1}{4\lambda^{2} \theta^{2}} \frac{\iota(\vartheta_{1},\vartheta_{2}) }{\sin^{2}\tfrac{\vartheta_{1}}{2} \sin^{2} \tfrac{\vartheta_{2}}{2}} \right) d\kappa d\eta \\
		&=& \frac{h^{2N_{0}-2}c^{4N_{0}-1}}{4\pi^{2}} \int_{-\pi}^{\pi}\int_{-\pi}^{\pi} \frac{\exp(\imi j \vartheta_{1})\exp(\imi k \vartheta_{2})}{\left[ 2 \lambda \iota(\vartheta_{1},\vartheta_{2}) \right]^{2N_{0}}} \exp\left(  - \frac{1}{4\lambda^{2} \theta^{2}} \frac{\iota(\vartheta_{1},\vartheta_{2}) }{\sin^{2}\tfrac{\vartheta_{1}}{2} \sin^{2} \tfrac{\vartheta_{2}}{2}} \right) d\vartheta_{1} d\vartheta_{2}.
\end{eqnarray*}
First, consider the low-wavenumber error. 
The inverse mixed discrete/continuous Fourier transform of $\widehat{E}^{low}$ is given by 
\begin{equation*}
\frac{h^{2}}{4\pi^{2}} \int_{-\pi /h_{2}}^{\pi /h_{2}}\int_{-\pi /h_{1}}^{\pi /h_{1}} \widehat{u}(\kappa,\eta,1) (s^{[2]}(\kappa,\eta) + N_{0}^{[2]}(\kappa,\eta)) \exp(\imi \kappa x_{j}) \exp( \imi \eta y_{k}) d\kappa d\eta, 
\end{equation*} 
which can be approximated by \eqref{eq:LowWavenumberError} as $h_{1},h_{2}$ tend to zero.
Let 
$$ \phi_{\rho}(x,y) = \tfrac{1}{\sqrt{4\pi^{2}(1-\rho^{2})}} \exp\left( - \tfrac{x^{2} - 2\rho x y + y^{2}}{2(1-\rho^{2})} \right), $$
the density function of a two-dimensional standard-normally distributed random variable with correlation $\rho$.
Its Fourier transform is
$$ \widehat{\phi}_{\rho}(\kappa,\eta) = \exp\left( -\tfrac{\kappa^{2}}{2} - \rho \kappa \eta - \tfrac{\eta^{2}}{2}\right). $$
Hence, for all positive integers $n_{1}, n_{2}$ the Fourier transform of $ \frac{\partial^{n_{1}+n_{2}}}{\partial x^{n_{1}}\partial y^{n_{2}}} \phi_{\rho}\left( \frac{x+a_{1}}{\sqrt{2}},\frac{y+a_{2}}{\sqrt{2}} \right) $ is
$$ 2(\imi \sqrt{2} \kappa)^{n_{1}} (\imi \sqrt{2} \eta)^{n_{2}} \exp(-\kappa^{2} - 2\rho \kappa \eta - \eta^{2} + \imi a_{1} \kappa + \imi a_{2} \eta) , $$
such that the inverse Fourier transform of
$$ h^{2} \exp( -\kappa^{2} - 2\rho \kappa \eta - \eta^{2} + \imi a_{1} \kappa + \imi a_{2} \eta )\kappa^{n_{1}}\eta^{n_{2}} $$
is given by
$$ \frac{h^{2}}{2}\frac{1}{(\imi\sqrt{2})^{n_{1}+n_{2}}} \frac{\partial^{n_{1}+n_{2}}}{\partial x^{n_{1}} \partial y^{n^{2}}} \phi_{\rho}\left( \frac{x+a_{1}}{\sqrt{2}},\frac{y+a_{2}}{\sqrt{2}} \right).  $$
Recalling the formulas for $s^{[2]}$ and $R^{[2]}$ from Subsection \ref{Region1}, this leads to the following expression for the low-wavenumber error:
\begin{equation}
E^{low}_{j,k} = h^{2} C_{x_{j},y_{k}}^{low},
\label{eq:LowWavenumberErrorDetailed}
\end{equation}
with
\begin{eqnarray*}
C_{x_{j},y_{k}}^{low} &=& \frac{1}{2} \Bigg[ \frac{1}{48}\frac{\partial^{4}}{\partial x^{4}} + \frac{\rho}{12}\left( \frac{\partial^{4}}{\partial x^{3} \partial y} + \frac{c^{2}\partial^{4}}{\partial x \partial y^{3}} \right) + \frac{c^{2}}{48}\frac{\partial^{4}}{\partial y^{4}}  + \frac{a_{1}}{12\sqrt{2}} \frac{\partial^{3}}{\partial x^{3}} + \frac{a_{2}c^{2}}{12\sqrt{2}} \frac{\partial^{3}}{\partial y^{3}}   \\
&& - \ \lambda^{2}\theta^{2} \left( \frac{1}{2} \frac{\partial^{2}}{\partial x^{2}} + \frac{a_{1}}{\sqrt{2}} \frac{\partial}{\partial x} \right) \left( \frac{1}{2} \frac{\partial^{2}}{\partial y^{2}} + \frac{a_{2}}{\sqrt{2}}\frac{\partial}{\partial y} \right) \left(  \frac{1}{2} \frac{\partial^{2}}{\partial x^{2}} + \rho \frac{\partial^{2}}{\partial x \partial y} + \frac{1}{2} \frac{\partial^{2}}{\partial y^{2}} + \frac{a_{1}}{\sqrt{2}} \frac{\partial}{\partial x} + \frac{a_{2}}{\sqrt{2}}\frac{\partial}{\partial y}  \right) \\
&& + \ \frac{\lambda^{2}}{12} \left(  \frac{1}{2} \frac{\partial^{2}}{\partial x^{2}} + \rho \frac{\partial^{2}}{\partial x \partial y} + \frac{1}{2} \frac{\partial^{2}}{\partial y^{2}} + \frac{a_{1}}{\sqrt{2}} \frac{\partial}{\partial x} + \frac{a_{2}}{\sqrt{2}}\frac{\partial}{\partial y}  \right)^{3} \\
&& - \ \lambda^{2} \left(  \frac{1}{2} \frac{\partial^{2}}{\partial x^{2}} + \rho \frac{\partial^{2}}{\partial x \partial y} + \frac{1}{2} \frac{\partial^{2}}{\partial y^{2}} + \frac{a_{1}}{\sqrt{2}} \frac{\partial}{\partial x} + \frac{a_{2}}{\sqrt{2}}\frac{\partial}{\partial y}  \right)\times \\
&& \qquad \qquad \left(  \frac{\rho}{2} \frac{\partial^{2}}{\partial x \partial y} + \left(\frac{1}{2} - \theta \right) \left( \frac{1}{2} \frac{\partial^{2}}{\partial x^{2}} + \frac{1}{2} \frac{\partial^{2}}{\partial y^{2}} + \frac{a_{1}}{\sqrt{2}} \frac{\partial}{\partial x} + \frac{a_{2}}{\sqrt{2}}\frac{\partial}{\partial y} \right) \right)^{2} \\
&& + \ \frac{N_{0}\lambda^{2}}{4} \left(  \frac{1}{2} \frac{\partial^{2}}{\partial x^{2}} + \rho \frac{\partial^{2}}{\partial x \partial y} + \frac{1}{2} \frac{\partial^{2}}{\partial y^{2}} + \frac{a_{1}}{\sqrt{2}} \frac{\partial}{\partial x} + \frac{a_{2}}{\sqrt{2}}\frac{\partial}{\partial y}  \right)^{2}  \Bigg] \phi_{\rho}\left( \frac{x_{j}+a_{1}}{\sqrt{2}},\frac{y_{k}+a_{2}}{\sqrt{2}} \right).
\end{eqnarray*}
Next, consider the high-wavenumber error and note that
$$\exp(\imi j \vartheta_{1})\exp(\imi k \vartheta_{2}) = \cos( j\vartheta_{1} + k \vartheta_{2}) + \imi \sin( j\vartheta_{1} + k \vartheta_{2}).$$
Symmetry yields
\begin{equation}
\label{eq:HighWavenumberError}
E^{high}_{j,k} = h^{2N_{0}-2} C^{high}_{j,k},
\end{equation}
where
\begin{eqnarray*}
C^{high}_{j,k} &=& \frac{c^{4N_{0}-1}}{2\pi^{2}} \int_{0}^{\pi} \int_{0}^{\pi} \frac{\cos( j\vartheta_{1} + k \vartheta_{2})}{\left[ 2 \lambda \iota(\vartheta_{1},\vartheta_{2}) \right]^{2N_{0}}} \exp\left(  - \frac{1}{4\lambda^{2} \theta^{2}} \frac{\iota(\vartheta_{1},\vartheta_{2}) }{\sin^{2}\tfrac{\vartheta_{1}}{2} \sin^{2} \tfrac{\vartheta_{2}}{2}} \right) d\vartheta_{1} d\vartheta_{2} \\
		&& + \ \frac{c^{4N_{0}-1}}{2\pi^{2}} \int_{-\pi}^{0} \int_{0}^{\pi} \frac{\cos( j\vartheta_{1} + k \vartheta_{2})}{\left[ 2 \lambda \iota(\vartheta_{1},\vartheta_{2}) \right]^{2N_{0}}} \exp\left(  - \frac{1}{4\lambda^{2} \theta^{2}} \frac{\iota(\vartheta_{1},\vartheta_{2}) }{\sin^{2}\tfrac{\vartheta_{1}}{2} \sin^{2} \tfrac{\vartheta_{2}}{2}} \right) d\vartheta_{1} d\vartheta_{2}. 
\end{eqnarray*}
Combining both expressions \eqref{eq:LowWavenumberErrorDetailed} and \eqref{eq:HighWavenumberError} gives an approximation for the \textit{total error}:
\begin{equation} 
 U_{N,j,k} - u(x_{j},y_{k},1) \approx h^{2} C^{low}_{x_{j},y_{k}} + h^{2N_{0}-2} C^{high}_{j,k}. 
\label{eq:TotalErrorApprox} 
\end{equation}
The values $C^{low}_{x_{j},y_{k}}$ are only dependent on the position $(x_{j},y_{k})=(jh_{1},kh_{2})$, the parameter values of the problem and the ratios $c$ and $\lambda$. 
The constants $C^{high}_{j,k}$ only depend on the index $(j,k)$, the correlation parameter $\rho$ and the ratios $c, \lambda$.
For the numerical experiments, cf.\ infra, the values $C^{low}_{x_{j},y_{k}}$ are calculated by determining all the partial derivatives. The integrals in $C^{high}_{j,k}$ are approximated by numerical integration.
It is readily seen that
\begin{equation*}
\max_{j,k} \vert C_{j,k}^{high} \vert = \vert C_{0,0}^{high} \vert ,
\end{equation*}
so $E_{j,k}^{high}$ has a maximum magnitude where $(x_{j},y_{k}) = (0,0)$. \textit{This is exactly at the position of the discontinuity of the initial function.} At the end of Subsection \ref{subsec:HighWaveTransf} it was conjectured that for larger values of the MCS parameter $\theta$ one can expect a larger high-wavenumber error. This conjecture is confirmed by the above analysis given that $\iota(\vartheta_{1},\vartheta_{2})$ is always positive. In order to avoid spurious erratic behaviour in the numerical solution, \textit{it is therefore recommended to use smaller values of the parameter $\theta$.} However, one has to take into account the lower bound on $\theta$ described in Subsection \ref{subsec:Stability}.

We showed in \eqref{eq:TotalErrorApprox} that the total error is $\mathcal{O}(h^{\min\{2,2N_{0}-2\}})$ so that $N_{0}=2$ is a lower bound on $N_{0}$ for the Rannacher time stepping in order to ensure convergence of the numerical solution to the exact solution.
This is confirmed by the plots in Figure \ref{fig:Convergence13} which display total errors (in the maximum norm) in actual numerical experiments for model problem (\ref{eq:ModelPDE}) as a function of $1/h$, with parameter values $\rho = -0.7, a_{1} = 2, a_{2} = 3$, MCS parameter $\theta = 1/3$ and with $c=1$, $0.2 \leq \lambda \leq 0.8$. 
Since it is not possible to handle infinite domains in numerical experiments, the computational domain is restricted to spatial gridpoints $(x_{j},y_{k}) \in [-10, 10] \times [-10, 10]$. At the boundaries, homogeneous Dirichlet boundary conditions are applied.
In the left plots the case $N_{0}=0$ is considered, whereas the right plots show the corresponding results for $N_{0}=2$.
In the upper plots the maximum error between our numerical solution and the exact solution is shown as a function of $1/h$ for different values of $\lambda$. In the lower plots we show the same maximum error for one value of $\lambda$, together with our theoretical estimates for the corresponding low-wavenumber error and high-wavenumber error.
In these lower plots it is clearly seen that our theoretical estimates for the total error are sharp.

For the case where no Rannacher time stepping is applied, the left plots in Figure \ref{fig:Convergence13} reveal second-order convergence behaviour until $h$ reaches a critical value where the high-wavenumber error starts exceeding the low-wavenumber error.
It can be observed that this value of $h$, and thus the high-wavenumber error, is highly dependent on the ratio $\lambda = \Delta t / h$.
For smaller values of $\lambda$, $E_{j,k}^{high}$ is only sizeable whenever $h$ is very small, whereas for larger values of $\lambda$, $E_{j,k}^{high}$ already dominates the total error for larger values of $h$.
Moreover, the error constant for the low-wavenumber error is also dependent on $\lambda$. However, this is much less pronounced than for the high-wavenumber error.

The right plots in Figure \ref{fig:Convergence13} show the corresponding results in the case where the first two MCS timesteps are replaced by four backward Euler half-timesteps, thus $N_{0}=2$. One observes that the numerical approximations now exhibit second-order convergence for all values of $\lambda$. In the bottom right plot the high-wavenumber error is not visible since it is strongly dominated by the low-wavenumber error. The same observation is made for other values of $\lambda$. Hence, whenever Rannacher time stepping is applied with $N_{0}=2$, the total error can be approximated by $E_{j,k}^{low}$, which is of second-order in $h$. We find that the error constant for the low-wavenumber error is mildly dependent on the ratio $\lambda = \Delta t / h$. This can be explained through the fact that for a fixed value of $h$ but smaller value of $\lambda$ the same semidiscrete system is solved with a smaller timestep $\Delta t$. Finally, we notice that the latter error constant is slightly larger than for the case where $N_{0}=0$. Thus, by applying Rannacher time stepping with $N_{0}=2$, second-order convergence can be recovered at the small cost of a marginally larger error constant for the low-wavenumber error.

As stated above, the high-wavenumber error is very sensitive to the MCS parameter $\theta$. To illustrate this, Figure \ref{fig:Convergence1} shows the same plots as in Figure \ref{fig:Convergence13} but with the MCS parameter replaced by $\theta=1$.
It can be seen that all the conclusions from Figure \ref{fig:Convergence13} remain valid. In order to get decent plots, however, it is necessary to consider smaller values for $\lambda$. This confirms that, for fixed $\lambda$, $E_{j,k}^{high}$ is strongly increasing as a function of $\theta$. Moreover, by comparing the upper plots of Figure \ref{fig:Convergence13} and Figure \ref{fig:Convergence1} for $\lambda=0.2$ it can be seen that the error constant of the low-wavenumber error is substantially larger for MCS parameter $\theta=1$ than for $\theta=1/3$. We conjecture that for fixed $\lambda$ and fixed $h$, the low-wavenumber error is also increasing as a function of $\theta$. 
Therefore, regardless of the number of Rannacher timesteps $N_{0}$, it seems more favourable to consider smaller values of $\theta$. In particular, the lowest value of $\theta$ which satisfies the restrictions from Subsection \ref{subsec:Stability} for all values $\vert \rho \vert <1$ is given by $\theta = 1/3$.

\begin{figure}
\begin{center}
\includegraphics[scale=0.5]{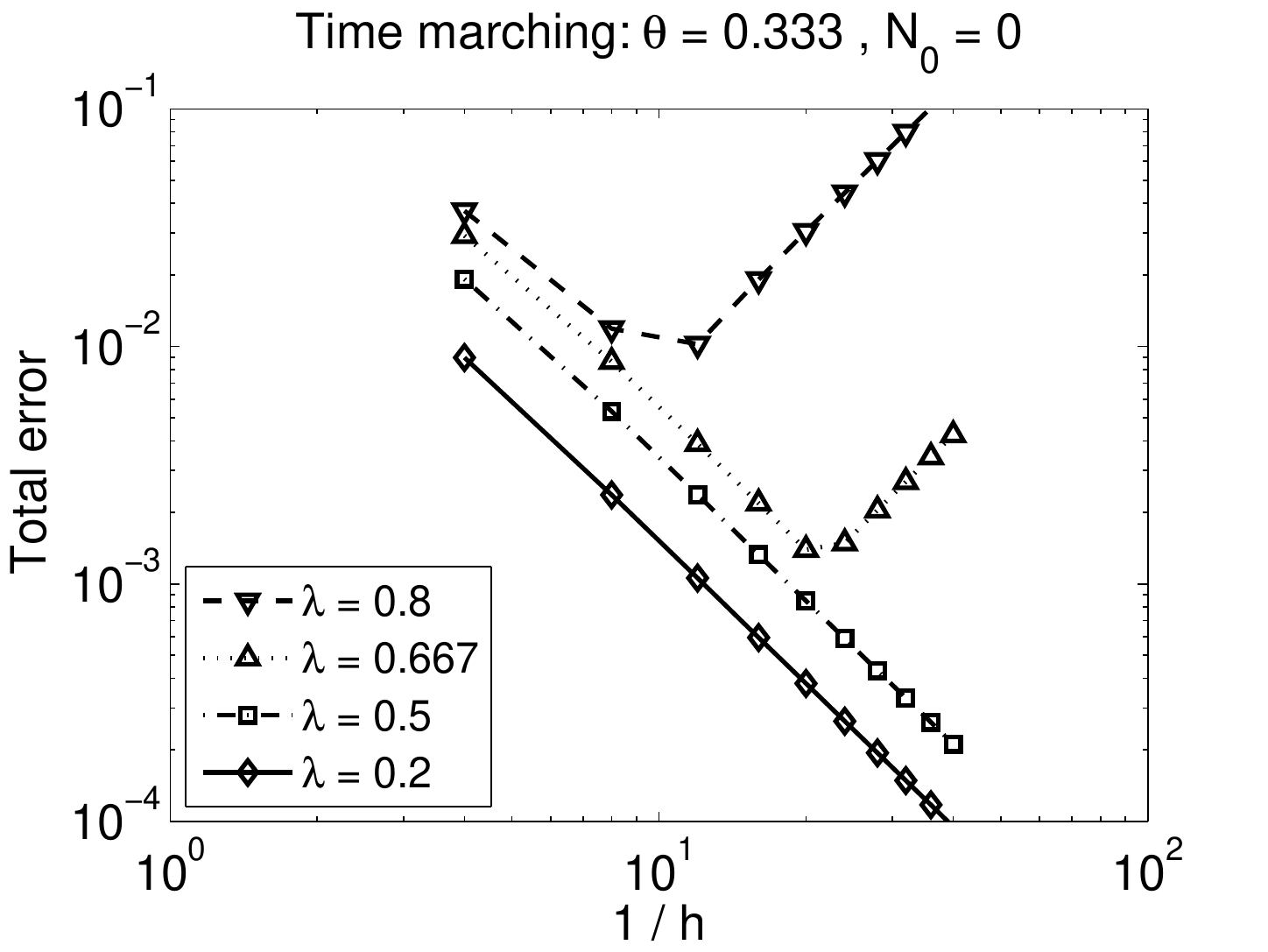} 
\includegraphics[scale=0.5]{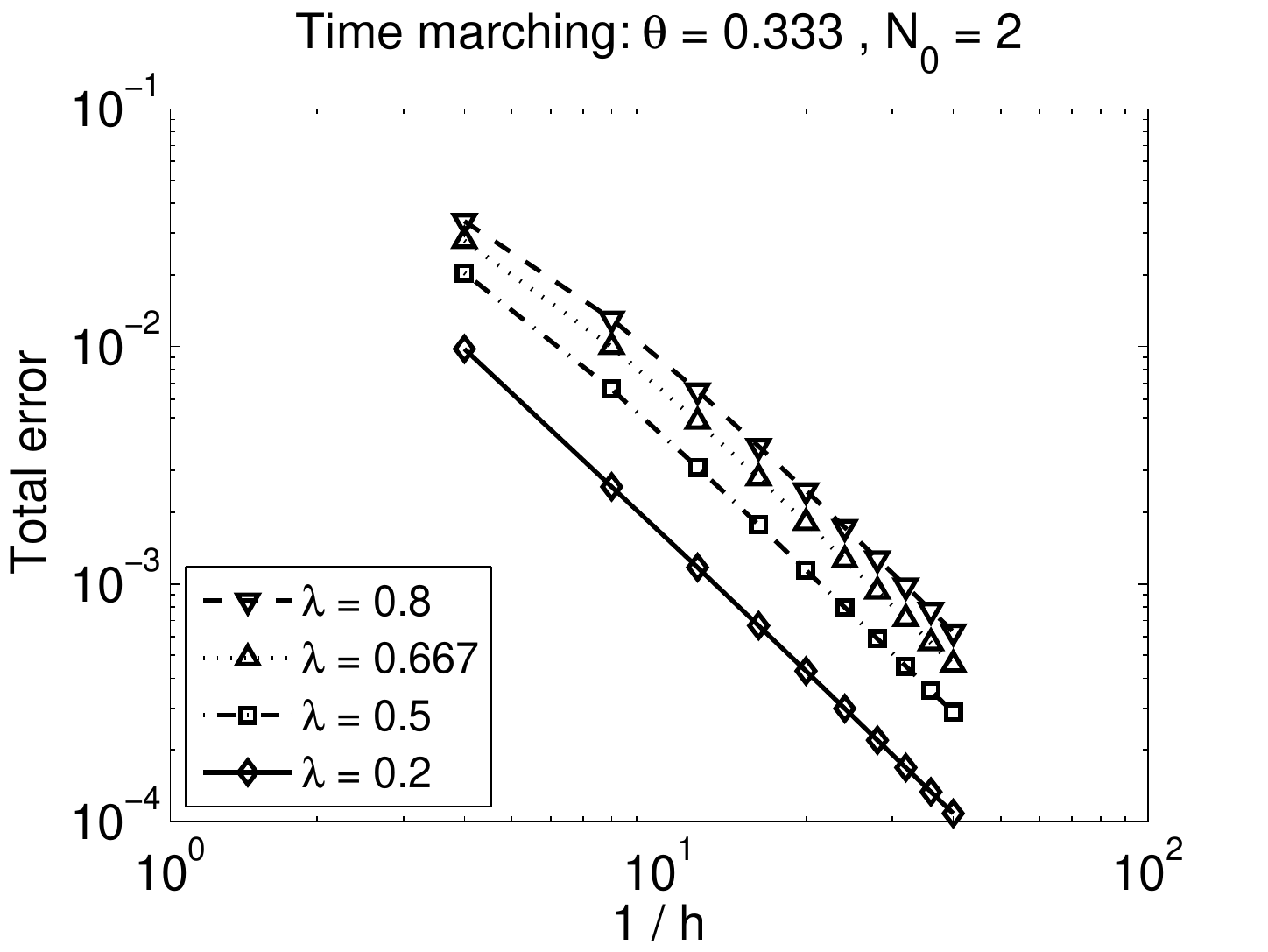} \newline \newline \newline
\includegraphics[scale=0.5]{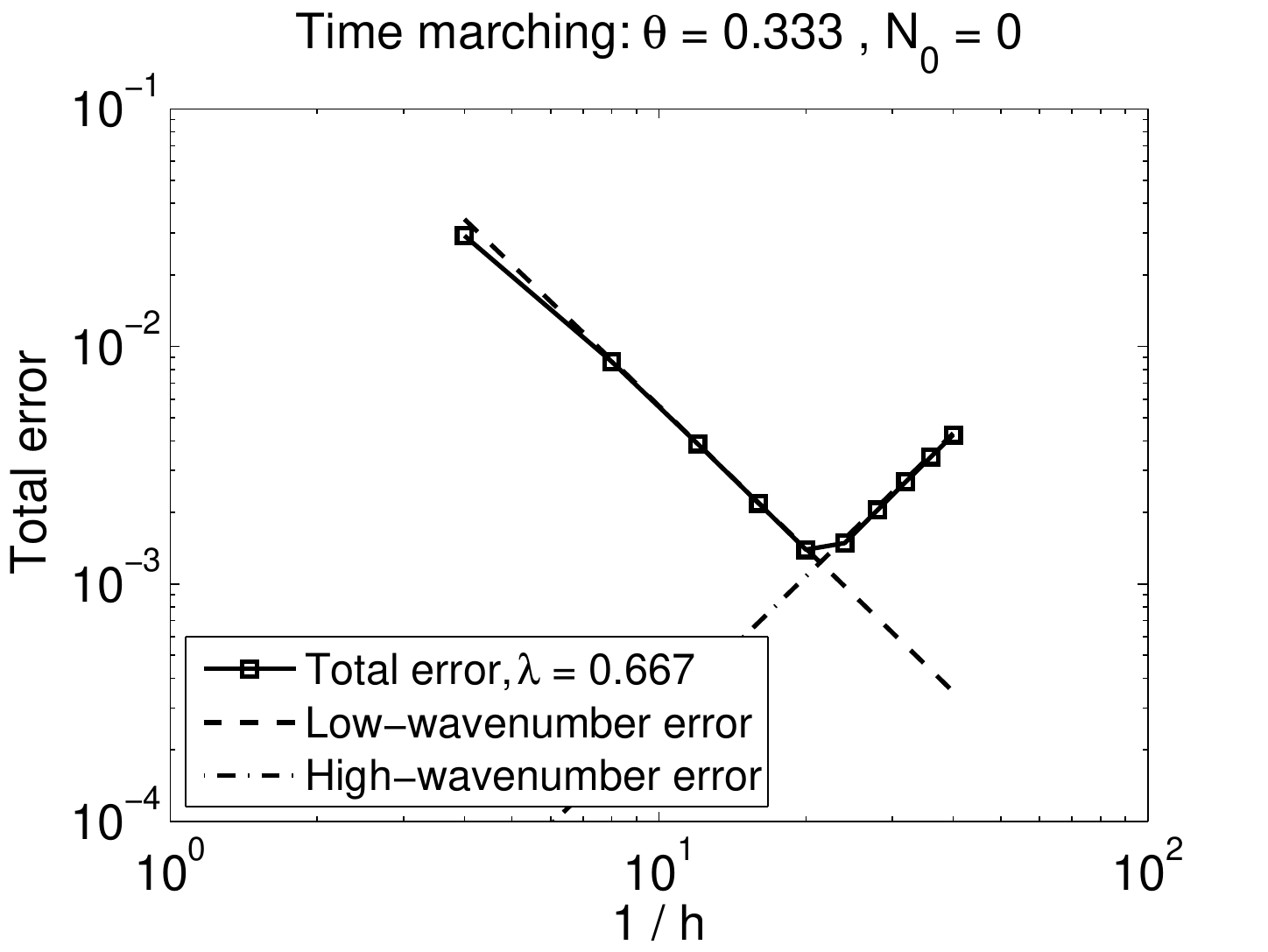} 
\includegraphics[scale=0.5]{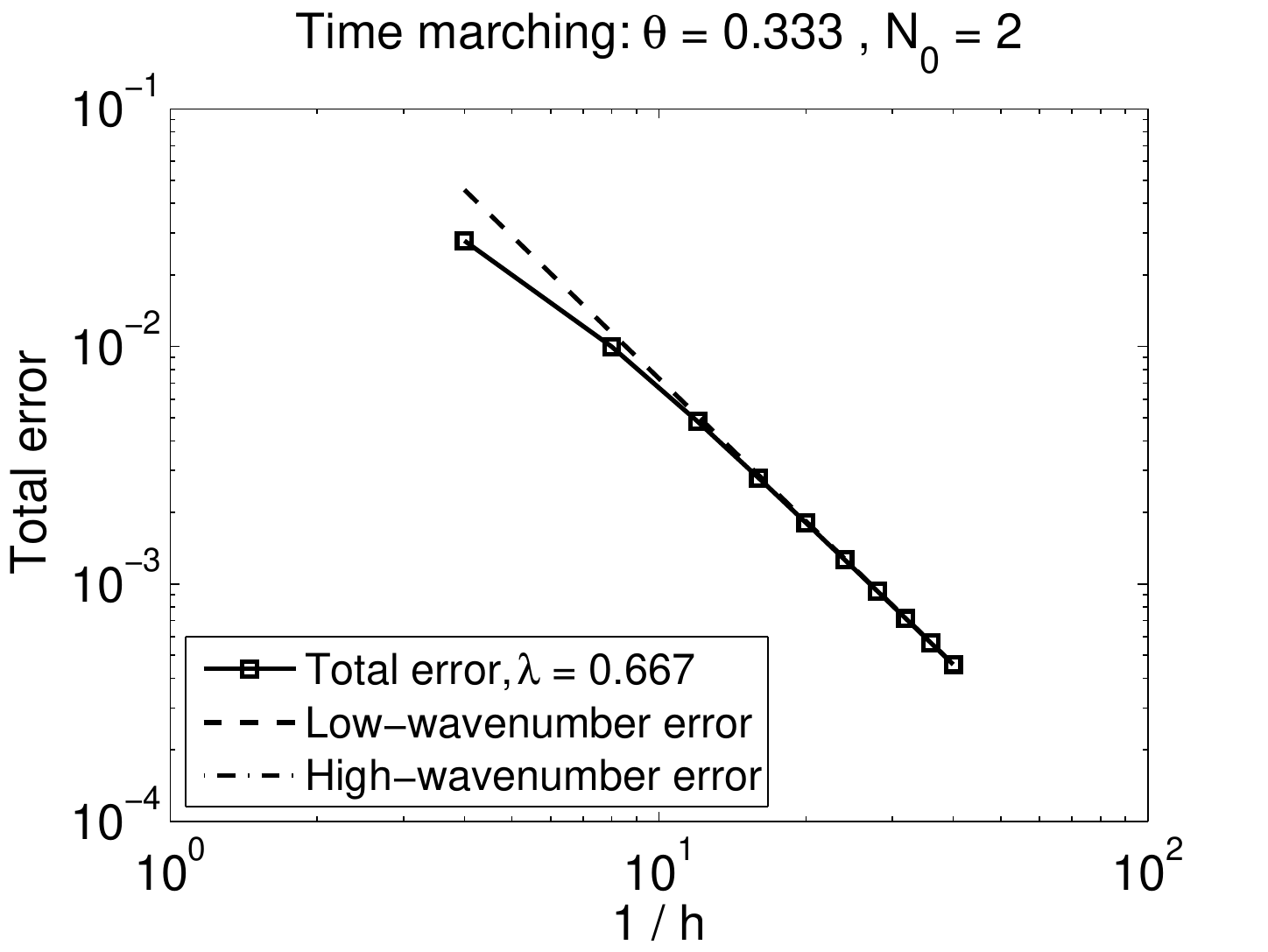} 
\caption{Convergence of the numerical solution for $N_{0}=0$ (left) and $N_{0}=2$ (right). The parameter values are: $\rho = -0.7, a_{1} = 2, a_{2} = 3, \theta = 1/3$.}
\label{fig:Convergence13}
\end{center}
\end{figure}

\begin{figure}
\begin{center}
\includegraphics[scale=0.5]{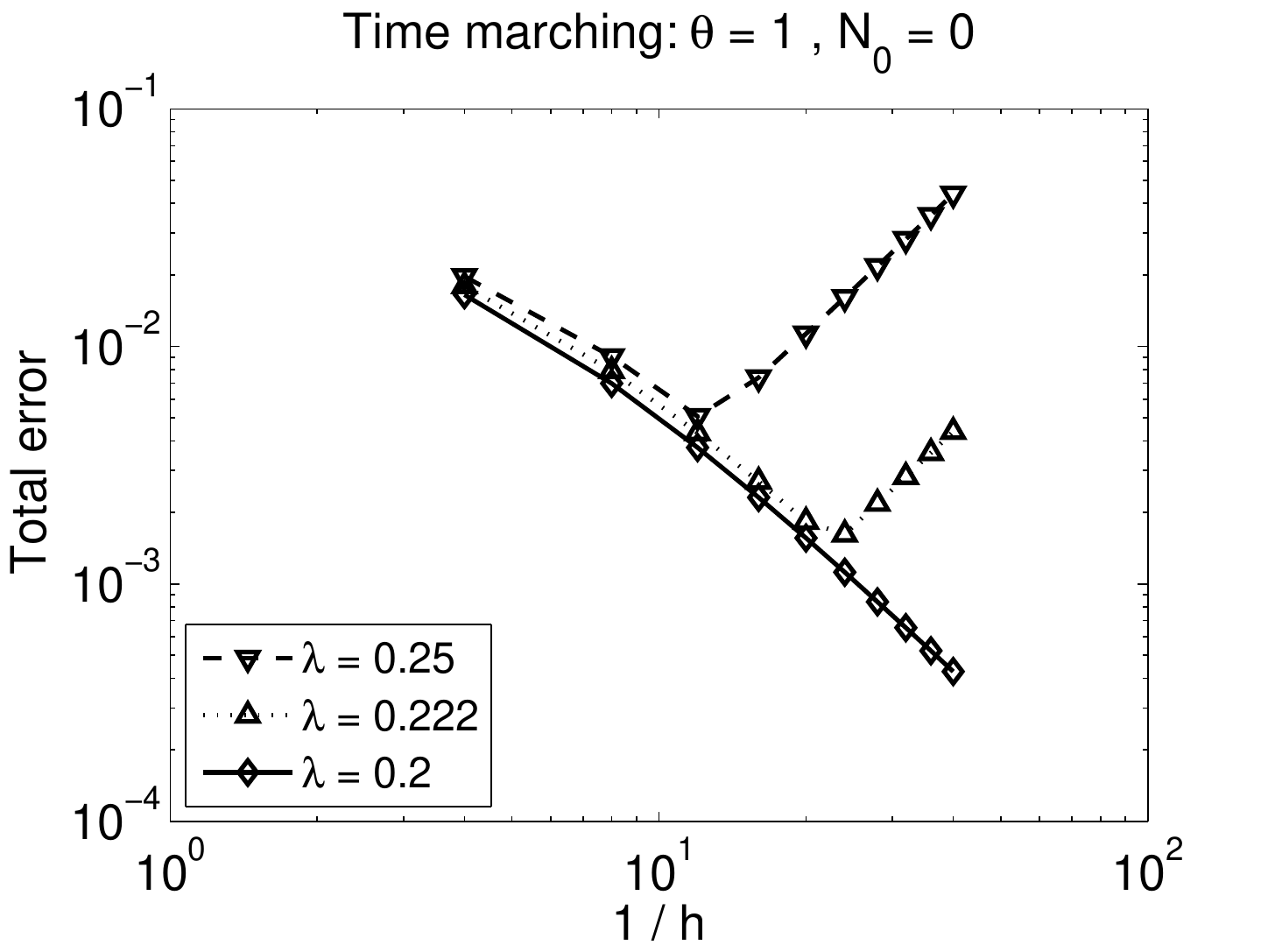} 
\includegraphics[scale=0.5]{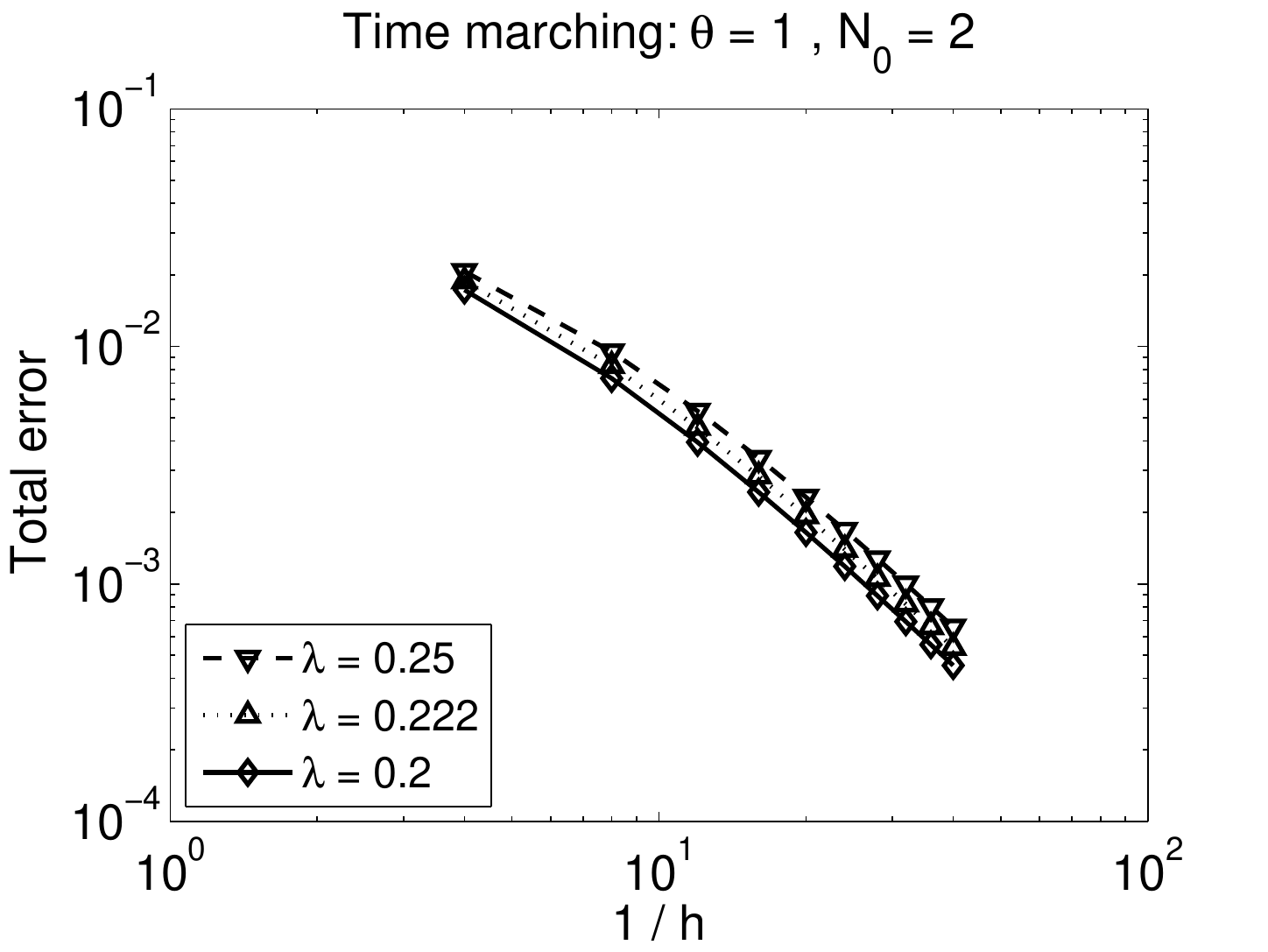} \newline \newline \newline
\includegraphics[scale=0.5]{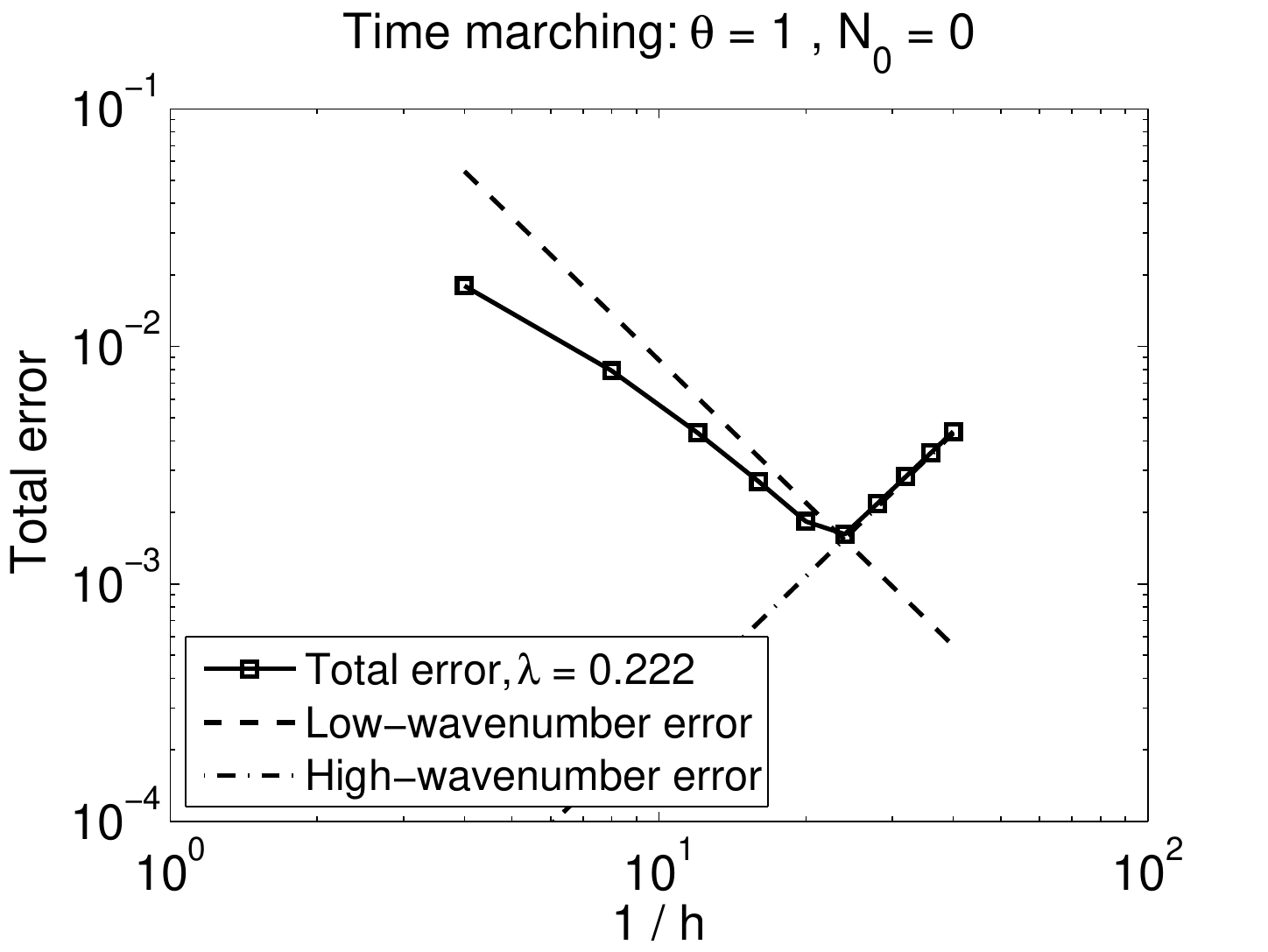} 
\includegraphics[scale=0.5]{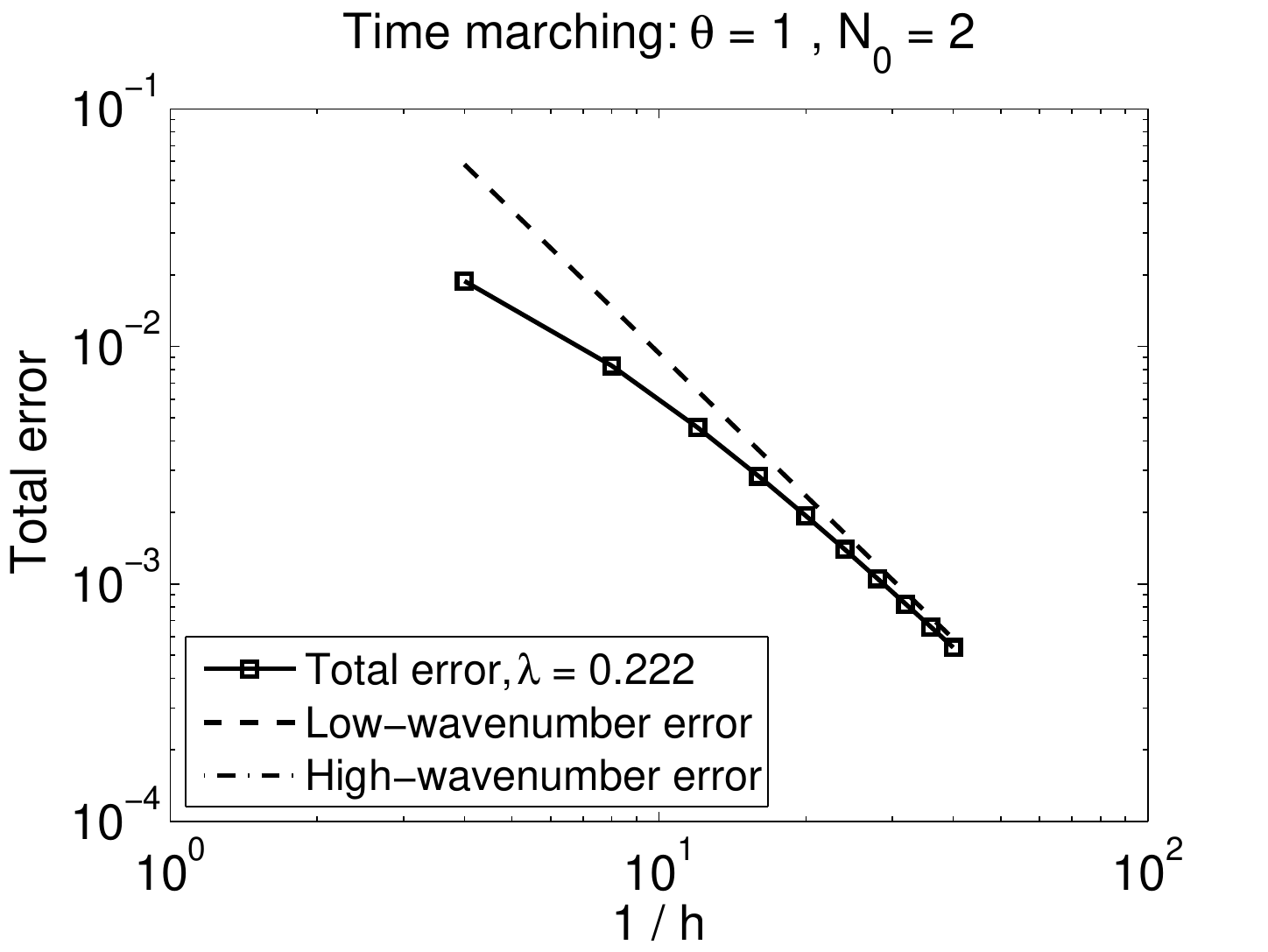} 
\caption{Convergence of the numerical solution for $N_{0}=0$ (left) and $N_{0}=2$ (right). The parameter values are: $\rho = -0.7, a_{1} = 2, a_{2} = 3, \theta = 1$.}
\label{fig:Convergence1}
\end{center}
\end{figure}

\subsection{MCS scheme with $\theta = 1/2$}

For $\theta = 1/2$, the MCS scheme reduces to the CS scheme and $\widehat{U}_{N}$ is not negligible in region 4 and region 5. Based on equalities (\ref{eq:CSWavenumberError}) and (\ref{eq:CSWavenumberError2}) we define
\begin{equation*}
\widehat{E}^{CS,4} = (-1)^{N-N_{0}} \frac{h^{2N_{0}}}{(2\lambda \sin^{2}\tfrac{\vartheta_{1}}{2})^{2N_{0}}} \exp\left( \frac{-1}{\lambda^{2} \sin^{2}\tfrac{\vartheta_{1}}{2}} - \eta^{2} + \imi a_{2} \eta \right)
\end{equation*}
and
\begin{equation*}
\widehat{E}^{CS,5} = (-1)^{N-N_{0}} \frac{(ch)^{2N_{0}}}{(2\lambda \sin^{2}\tfrac{\vartheta_{2}}{2})^{2N_{0}}} \exp\left( \frac{-c^{2}}{\lambda^{2} \sin^{2}\tfrac{\vartheta_{2}}{2}} - \kappa^{2} + \imi a_{1} \kappa \right).
\end{equation*}
Since $\vartheta_{1} = \kappa h_{1},$ $\vartheta_{2} = \eta h_{2}$ and $h_{2} = ch_{1} = ch$, $\widehat{E}^{CS,4}$ only has to be considered in region 4 and $\widehat{E}^{CS,5}$ is only not negligible in region 5. Hence, the Fourier error \eqref{eq:FourierErrorDef} can be approximated by
\begin{equation*}
 \widehat{U}_{N}(\kappa h_{1},\eta h_{2}) - \widehat{u}(\kappa,\eta,1) \approx \widehat{E}^{low} + \widehat{E}^{high} + \widehat{E}^{CS,4} + \widehat{E}^{CS,5}, \qquad \vert \kappa \vert, \vert c\eta \vert \leq \pi/h.  
\end{equation*}
The inverse mixed discrete/continuous Fourier transform of $(-1)^{N-N_{0}} ( \widehat{E}^{CS,4} + \widehat{E}^{CS,5})$ is given by
\begin{eqnarray*}
&& \frac{h^{2N_{0}}}{4\pi^{2}} \int_{-\pi / h_{2}}^{\pi / h_{2}} \int_{-\pi / h_{1}}^{\pi / h_{1}} \frac{\exp(\imi\kappa x_{j})\exp(\imi \eta y_{k})}{(2\lambda \sin^{2}\tfrac{\vartheta_{1}}{2})^{2N_{0}}} \exp\left( \frac{-1}{\lambda^{2} \sin^{2}\tfrac{\vartheta_{1}}{2}} - \eta^{2} + \imi a_{2} \eta \right) d\kappa d\eta,  \\
&&  + \ \frac{h^{2N_{0}}}{4\pi^{2}} \int_{-\pi / h_{2}}^{\pi / h_{2}} \int_{-\pi / h_{1}}^{\pi / h_{1}} \frac{\exp(\imi\kappa x_{j})\exp(\imi \eta y_{k})}{(2\tfrac{\lambda}{c} \sin^{2}\tfrac{\vartheta_{2}}{2})^{2N_{0}}} \exp\left( \frac{-c^{2}}{\lambda^{2} \sin^{2}\tfrac{\vartheta_{2}}{2}} - \kappa^{2} + \imi a_{1} \kappa \right) d\kappa d\eta . 
\end{eqnarray*}
As $h_{1},h_{2}$ tend to zero this can be approximated by
\begin{eqnarray*}
\lefteqn{ (-1)^{N-N_{0}} E^{CS}_{j,k} := \frac{h^{2N_{0}-1}}{4\pi^{2}} \int_{-\infty}^{\infty} \int_{-\pi}^{\pi} \frac{\exp(\imi j \vartheta_{1})\exp(\imi \eta y_{k})}{(2\lambda \sin^{2}\tfrac{\vartheta_{1}}{2})^{2N_{0}}} \exp\left( \frac{-1}{\lambda^{2} \sin^{2}\tfrac{\vartheta_{1}}{2}} - \eta^{2} + \imi a_{2} \eta \right) d\vartheta_{1} d\eta }  \\
&& \qquad \qquad \qquad  + \ \frac{h^{2N_{0}-1}}{4 c \pi^{2}} \int_{-\pi}^{\pi} \int_{-\infty}^{\infty} \frac{\exp(\imi\kappa x_{j})\exp(\imi k \vartheta_{2})}{(2\tfrac{\lambda}{c} \sin^{2}\tfrac{\vartheta_{2}}{2})^{2N_{0}}} \exp\left( \frac{-c^{2}}{\lambda^{2} \sin^{2}\tfrac{\vartheta_{2}}{2}}  - \kappa^{2} + \imi a_{1} \kappa \right) d\kappa d\vartheta_{2} . 
\end{eqnarray*}
Making use of a symmetry argument and a one-dimensional inverse Fourier transformation, $E^{CS}_{j,k}$ can be rewritten as
\begin{equation}
E^{CS}_{j,k} = h^{2N_{0}-1} (-1)^{N-N_{0}} (C^{CS}_{j,y_{k}} + C^{CS}_{x_{j},k}),
\label{eq:CSError}
\end{equation}
with
\begin{eqnarray*}
C^{CS}_{j,y_{k}} &=& \frac{1}{2 \sqrt{2} \pi} \phi\left( \frac{y_{k} + a_{2}}{\sqrt{2}} \right) \int_{-\pi}^{\pi} \frac{\cos(j\vartheta_{1})}{(2\lambda \sin^{2}\tfrac{\vartheta_{1}}{2})^{2N_{0}}} \exp\left( \frac{-1}{\lambda^{2} \sin^{2}\tfrac{\vartheta_{1}}{2}}\right) d\vartheta_{1}, \\
C^{CS}_{x_{j},k} &=& \frac{1}{2 \sqrt{2} c \pi} \phi\left( \frac{x_{j} + a_{1}}{\sqrt{2}} \right) \int_{-\pi}^{\pi} \frac{\cos(k\vartheta_{2})}{(2\tfrac{\lambda}{c} \sin^{2}\tfrac{\vartheta_{2}}{2})^{2N_{0}}} \exp\left( \frac{-c^{2}}{\lambda^{2} \sin^{2}\tfrac{\vartheta_{2}}{2}}\right) d\vartheta_{2}, 
\end{eqnarray*}
where $\phi$ denotes the density function of a standard normally distributed random variable. It is readily seen that $C^{CS}_{j,y_{k}}$, respectively $C^{CS}_{x_{j},k}$, reaches its highest magnitude near the points $(j,k)$ where $(x_{j},y_{k}) \approx (0,-a_{2})$, respectively $(x_{j},y_{k}) \approx (-a_{1},0)$.
For the numerical experiments, the integrals in $C^{CS}_{j,y_{k}}$ and $C^{CS}_{x_{j},k}$ are approximated by numerical integration.
Combining the expressions \eqref{eq:LowWavenumberErrorDetailed}, \eqref{eq:HighWavenumberError} and \eqref{eq:CSError} leads to the following approximation of the total error:
\begin{equation}
U_{N,j,k} - u(x_{j},y_{k},1) \approx h^{2} C^{low}_{x_{j},y_{k}} + h^{2N_{0}-2} C^{high}_{j,k} + h^{2N_{0}-1} (-1)^{N-N_{0}} (C^{CS}_{j,y_{k}} + C^{CS}_{x_{j},k}).
\label{eq:TotalErrorApproxCS}
\end{equation}

From approximation (\ref{eq:TotalErrorApproxCS}) it can be concluded that the total error is also $\mathcal{O}\left(h^{\min\{2,2N_{0}-2\}}\right)$ when CS time stepping is considered. This matches the observations from the plots in Figure \ref{fig:ConvergenceCS} which show convergence results for the same problem as in Subsection \ref{ConvergenceMCS} but with MCS parameter $\theta = 1/2$. The lower plots indicate again that our theoretical estimates for the total error are sharp. 
Without Rannacher time stepping, i.e.\ $N_{0}=0$, the results in Figure \ref{fig:ConvergenceCS} show second-order convergence in $h$ until $E^{CS}_{j,k}$ starts exceeding the low-wavenumber error. Then the total error increases in a first order way until the high-wavenumber error starts dominating. From there the total error is $\mathcal{O}\left(h^{-2}\right)$.
In case the MCS scheme is replaced in the first two timesteps by four half-timesteps of the implicit Euler scheme, i.e.\ $N_{0}=2$, Figure \ref{fig:ConvergenceCS} reveals unconditional second-order convergence in $h$.
Note that both $E^{CS}_{j,k}$ and $E^{high}_{j,k}$ are not visible in the lower-right plot because they are strongly dominated by the low-wavenumber error. The same observation as in Subsection \ref{ConvergenceMCS} can be made concerning the dependency of the low- and high-wavenumber error on the parameter $\lambda$.

\begin{figure}
\begin{center}
\includegraphics[scale=0.5]{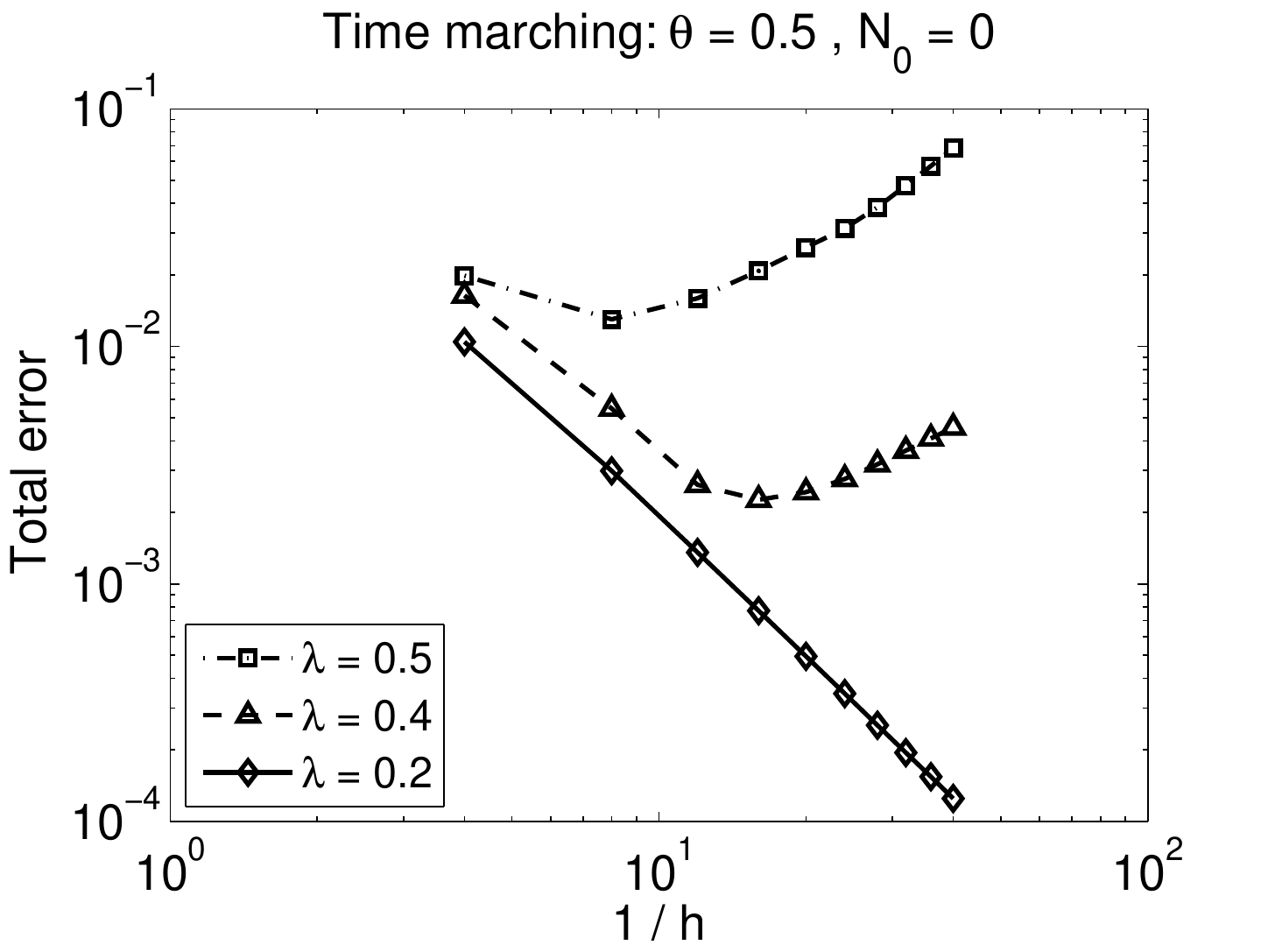} 
\includegraphics[scale=0.5]{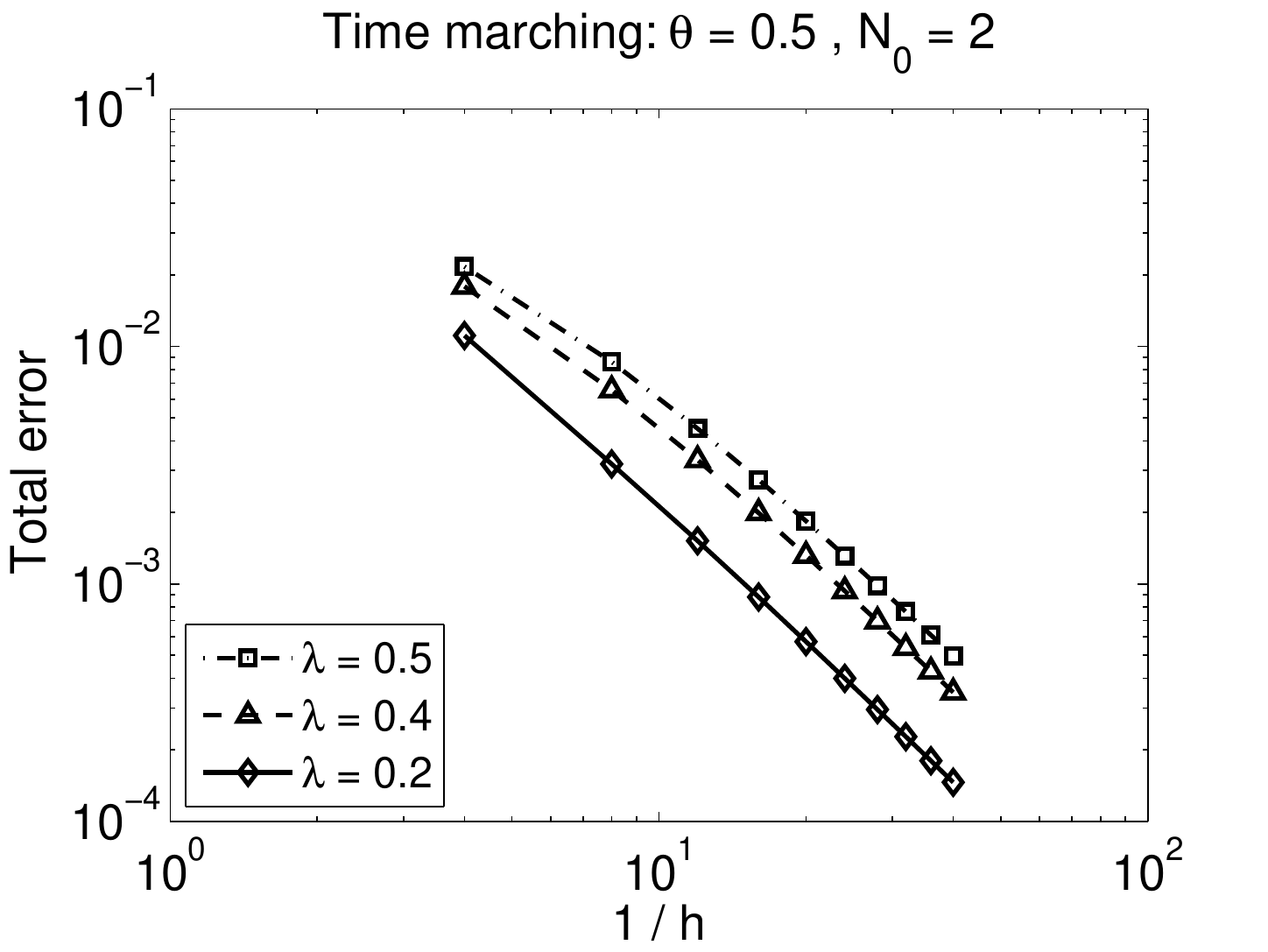} \newline \newline \newline
\includegraphics[scale=0.5]{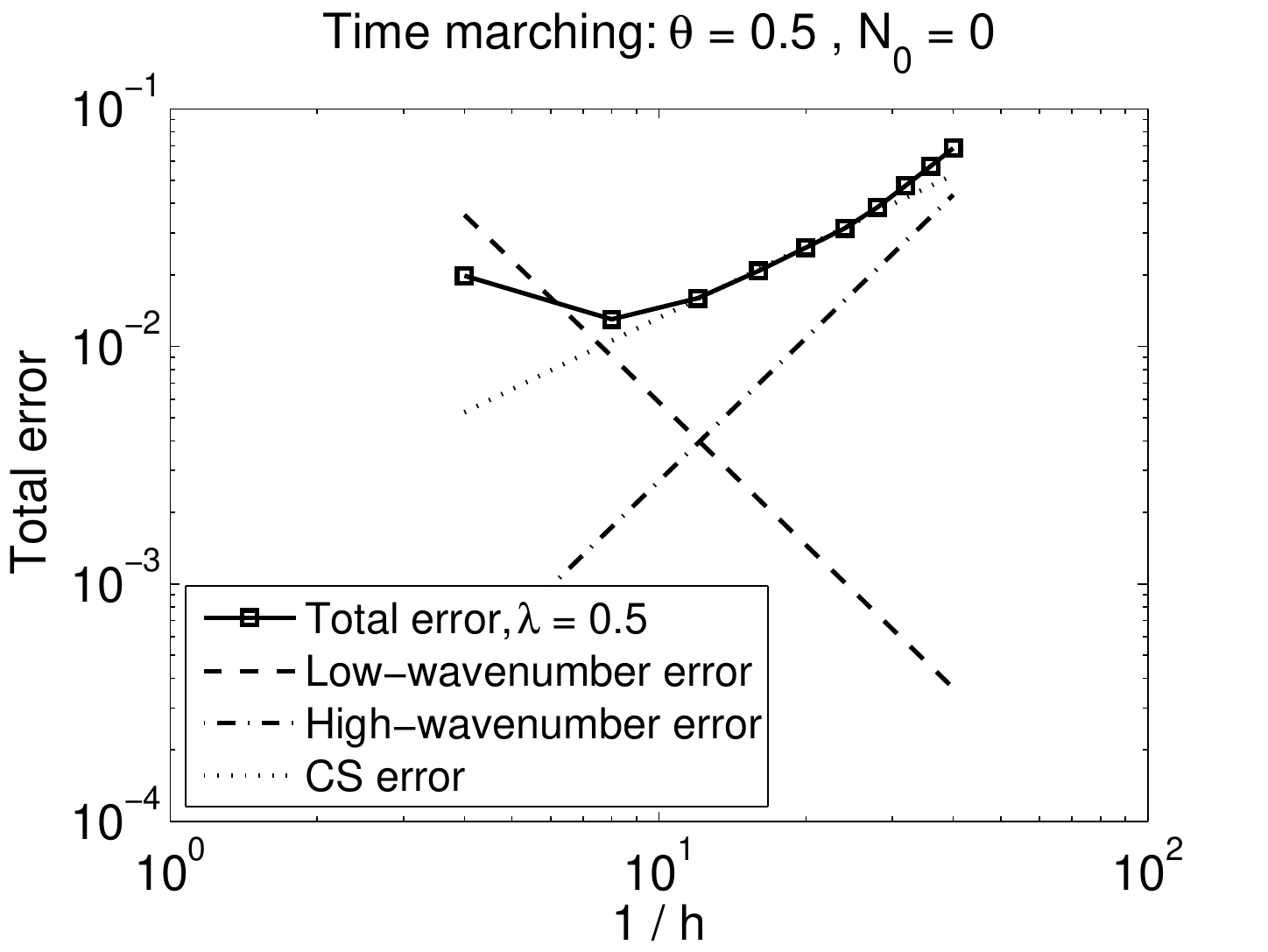} 
\includegraphics[scale=0.5]{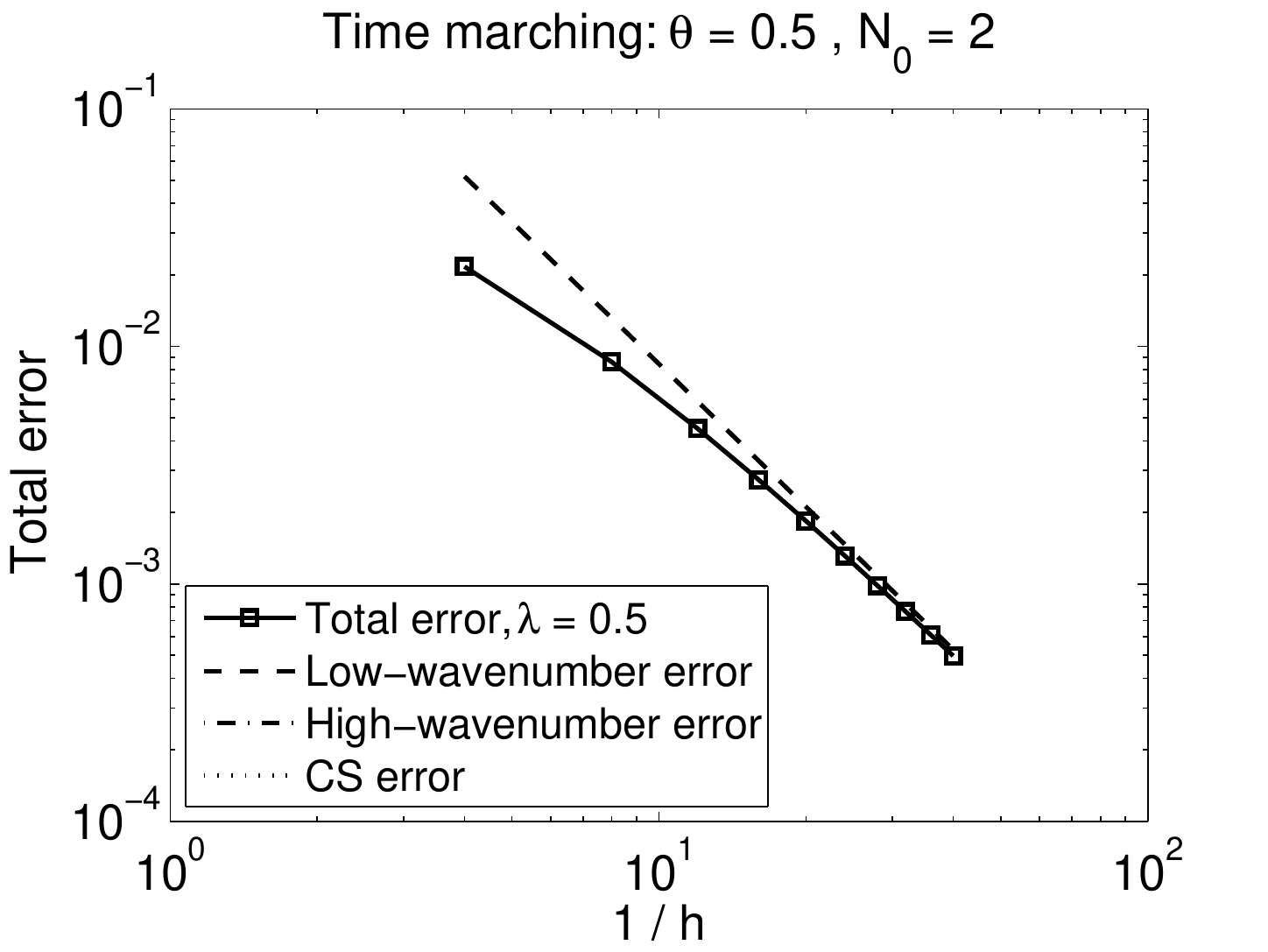} 
\caption{Convergence of the numerical solution for $N_{0}=0$ (left) and $N_{0}=2$ (right). The parameter values are: $\rho = -0.7, a_{1} = 2, a_{2} = 3, \theta = 1/2$.}
\label{fig:ConvergenceCS}
\end{center}
\end{figure}

\newpage
\setcounter{equation}{0}
\section{Conclusion}\label{conclusion}

If the initial data is nonsmooth, application of the MCS scheme for multidimensional time-dependent convection-diffusion equations with mixed-derivative terms can cause spurious erratic behaviour in the numerical solution.
A motivating example, with the two-dimensional Black--Scholes equation for a two-asset cash-or-nothing option, shows that this undesirable feature can be resolved by replacing the very first $N_{0}$ MCS timesteps by $2N_{0}$ half-timesteps of the implicit Euler scheme, with $N_{0}=2$.
We proved, by Fourier analysis, that for a model two-dimensional convection-diffusion equation with mixed-derivative term and with Dirac delta initial data, the total error can be approximated by the sum of a low-wavenumber error of $\mathcal{O}(h^{2})$ and a high-wavenumber error of $\mathcal{O}(h^{2N_{0}-2})$. In case the MCS scheme reduces to the CS scheme, i.e. when $\theta = 1/2$, this has to be augmented with an extra error term of $\mathcal{O}(h^{2N_{0}-1})$.
Hence, $N_{0}=2$ is the minimum on $N_{0}$ in order to guarantee (second-order) convergence of the numerical solution to the exact solution, in the maximum norm. 
In general this choice for $N_{0}$ is optimal since larger values will increase the low-wavenumber error.
Our convergence analysis and numerical experiments further indicate that it is favourable to consider small values of the MCS parameter $\theta$. However, it is necessary to take into account the lower bounds on $\theta$ in order for our asymptotic analysis to be valid.
The smallest value which satisfies all the restrictions, independent of the parameters of the model, is given by $\theta = 1/3$.
This is, indeed, also the most common value of $\theta$ for the MCS scheme considered in the literature.

\setcounter{equation}{0}
\section*{Acknowledgements}
The author is grateful to Karel in 't Hout for his constructive and valuable remarks. 
This work has been supported financially by a PhD Fellowship of the Research Foundation--Flanders.

\vfill\eject


\begin{thebibliography}{99}

\bibitem{B98} \textsc{Bj\"ork, T.} (1998)
\textit{Arbitrage Theory in Continuous Time}.
Oxford: Oxford University Press.

\bibitem{CS88} \textsc{Craig, I.~J.~D., Sneyd, A.~D.} (1988)
An alternating-direction implicit scheme for parabolic equations
with mixed derivatives.
\textit{Comp. Math. Appl.}, \textbf{16}, 341--350.

\bibitem{GC06} \textsc{Giles, M.~B., Carter, R.} (2006)
Convergence analysis of Crank--Nicolson and Rannacher time-marching.
\textit{J. Comp. Finan.}, \textbf{9}, 89--112.

\bibitem{IHF10} \textsc{in~'t~Hout, K.~J., Foulon, S.} (2010)
ADI finite difference schemes for option pricing in the Heston model
with correlation.
\textit{Int. J. Numer. Anal. Mod.}, \textbf{7}, 303--320.

\bibitem{IHM10} \textsc{in~'t~Hout, K.~J., Mishra, C.} (2010)
A stability result for the modified Craig--Sneyd scheme applied to 2D and 3D pure diffusion equations.
\textit{Num. Anal. and Appl. Math., AIP Conf. Proc.}, \textbf{1281}, 2029--2032.

\bibitem{IHM11} \textsc{in~'t~Hout, K.~J., Mishra, C.} (2011)
Stability of the modified Craig--Sneyd scheme for
two-dimensional convection-diffusion equations with mixed
derivative term.
\textit{Math. Comp. Simul.}, \textbf{81}, 2540--2548.

\bibitem{IHM13} \textsc{in~'t~Hout, K.~J., Mishra, C.} (2013)
Stability of ADI schemes for multidimensional diffusion equations
with mixed derivative terms.
\textit{Appl. Numer. Math.}, \textbf{74}, 83--94.

\bibitem{IHW09} \textsc{in~'t~Hout, K.~J., Welfert, B.~D.} (2009)
Unconditional stability of second-order ADI schemes applied to
multi-dimensional diffusion equations with mixed derivative terms.
\textit{Appl. Numer. Math.}, \textbf{59}, 677--692.

\bibitem{IHW15} \textsc{in~'t~Hout, K.~J., Wyns, M.} (2015)
Convergence of the Modified Craig--Sneyd scheme for two-dimensional convection-diffusion equations with mixed derivative term.
Submitted for publication.

\bibitem{HV03} \textsc{Hundsdorfer, W., Verwer, J.~G.} (2003)
\textit{Numerical Solution of Time-Dependent Advection-Diffusion-Reaction Equations.}
Berlin: Springer.

\bibitem{M14} \textsc{Mishra, C.} (2014)
\textit{Stability of Alternating Direction Implicit Schemes with Application to Financial
Option Pricing Equations}.
PhD thesis, University of Antwerp.

\bibitem{PVF03} \textsc{Pooley, D.~M., Vetzal, K.~R., Forsyth, P.~A.} (2003)
Convergence remedies for non-smooth payoffs in option pricing.
\textit{J. Comp. Finan.}, \textbf{6}, 25--40.

\bibitem{R84} \textsc{Rannacher, R.} (1984)
Finite element solution of diffusion problems with irregular data.
\textit{Numer. Math.}, \textbf{43}, 309--327.

\bibitem{S89} \textsc{Strikwerda, J.~C.} (1989)
\textit{Finite Difference Schemes and Partial Differential Equations.}
Belmont: Wadsworth Publ. Co.


\end{thebibliography}
\end{document}